\documentclass[11pt]{amsart}

\usepackage[dvipsnames]{xcolor}
\usepackage{tikz-cd}
\usepackage{changepage}
\usepackage{graphicx}
\usepackage{hyperref}
\usepackage{dsfont}
\usepackage{xfrac,faktor}
\usepackage{amsmath}
\usepackage{color}
\usepackage{caption}
\usepackage{subcaption}
\usepackage[T1]{fontenc}
\usepackage{mathtools}
\usepackage{amsfonts,amssymb,amscd,amsthm}
\usepackage{tikz}
\usepackage{comment}
\usepackage{float,url}
\usepackage{tikz-cd}
\usepackage[all]{xy}

\newlength{\mywidth}

\calclayout

\makeatletter
\DeclareFontFamily{U}{tipa}{}
\DeclareFontShape{U}{tipa}{m}{n}{<->tipa10}{}
\newcommand{\arc@char}{{\usefont{U}{tipa}{m}{n}\symbol{62}}}%

\newcommand{\arc}[1]{\mathpalette\arc@arc{#1}}

\newcommand{\arc@arc}[2]{%
  \sbox0{$\m@th#1#2$}%
  \vbox{
    \hbox{\resizebox{\wd0}{\height}{\arc@char}}
    \nointerlineskip
    \box0
  }%
}
\makeatother

\newtheorem{thm}{Theorem}[section]
\newtheorem{prop}[thm]{Proposition}
\newtheorem{lem}[thm]{Lemma}
\newtheorem{cor}[thm]{Corollary}

\newtheorem*{cor*}{Corollary}

\theoremstyle{definition}
\newtheorem{definition}[thm]{Definition}
\newtheorem{example}[thm]{Example}
\newtheorem{notation}[thm]{Notation}

\newtheorem{thmx}{Theorem}

\DeclareMathOperator{\Int}{int}

\makeatletter
\renewcommand*\env@matrix[1][\arraystretch]{%
  \edef\arraystretch{#1}%
  \hskip -\arraycolsep
  \let\@ifnextchar\new@ifnextchar
  \array{*\c@MaxMatrixCols c}}
\makeatother

\theoremstyle{remark}
\newtheorem{rem}[thm]{Remark}

\theoremstyle{question}

\newcommand{\R}{\mathbb{R}}  
\newcommand{\D}{\mathbb{D}}

\newcommand{\C}{\mathbb{C}}  
\newcommand{\Z}{\mathbb{Z}}  
\newcommand{\N}{\mathbb{N}}  

\tikzset{
  mynode/.style={fill,circle,inner sep=1pt,outer sep=0pt}
}

\DeclareMathOperator{\dist}{dist}

\begin{document}

\title[Bers Slices in Families of Univalent Maps]{Bers Slices in Families of Univalent Maps}

\author[K. Lazebnik]{Kirill Lazebnik}
\address{Department of Mathematics, California Institute of Technology, Pasadena, California 91125, USA}
\email{lazebnik@caltech.edu }

\author[N. G. Makarov]{Nikolai G. Makarov}
\address{Department of Mathematics, California Institute of Technology, Pasadena, California 91125, USA}
\email{makarov@caltech.edu}

\author[S. Mukherjee]{Sabyasachi Mukherjee}
\address{School of Mathematics, Tata Institute of Fundamental Research, 1 Homi Bhabha Road, Mumbai 400005, India}
\email{sabya@math.tifr.res.in}

\date{\today}

\maketitle

\begin{abstract} We construct embeddings of Bers slices of ideal polygon reflection groups into the classical family of univalent functions $\Sigma$. This embedding is such that the conformal mating of the reflection group with the anti-holomorphic polynomial $z\mapsto\overline{z}^d$ is the Schwarz reflection map arising from the corresponding map in $\Sigma$. We characterize the image of this embedding in $\Sigma$ as a family of univalent rational maps. Moreover, we show that the limit set of every Kleinian reflection group in the closure of the Bers slice is naturally homeomorphic to the Julia set of an anti-holomorphic polynomial.
\end{abstract}


\setcounter{tocdepth}{1}

\tableofcontents

\section{Introduction}
\label{Introduction}

In the 1980s, Sullivan proposed a dictionary between Kleinian groups and rational dynamics that was motivated by various common features shared by them \cite{Sul,SulMc}. However, the dictionary is not an automatic procedure to translate results in one setting to those in the other, but rather an inspiration for results and proof techniques. Several efforts to draw more direct connections between Kleinian groups and rational maps have been made in the last few decades (for example, see \cite{BP,McM95,LyMi,Pil03,BL20}). Amongst these, the questions of exploring dynamical relations between limit sets of Kleinian groups and Julia sets of rational maps, and binding together the actions of these two classes of conformal dynamical systems in the same dynamical plane play a central role in the current paper.

The notion of mating has its roots in the work of Bers on simultaneous uniformization of two Riemann surfaces. The simultaneous uniformization theorem allows one to mate two Fuchsian groups to obtain a quasiFuchsian group \cite{Bers60}. In the world of conformal dynamics, Douady and Hubbard introduced the notion of mating two polynomials to produce a rational map \cite{Dou83}. In each of these mating constructions, the key idea is to combine two ``similar'' conformal dynamical systems to produce a richer conformal dynamical system in the same class. Examples of ``hybrid dynamical systems'' that are \emph{conformal matings} of Kleinian reflection groups and anti-holomorphic rational maps (anti-rational for short) were constructed in \cite{LLMM1,LLMM2} as \emph{Schwarz reflection maps} associated with \emph{univalent rational maps}. Roughly speaking, this means that the dynamical planes of the Schwarz reflection maps in question can be split into two invariant subsets, on one of which the map behaves like an anti-rational map, and on the other, its grand orbits are equivalent to the grand orbits of a group. 

In the current paper, we further explore the aforementioned framework for mating Kleinian reflection groups with anti-rational maps, and show that all Kleinian reflection groups arising from (finite) circle packings satisfying a ``necklace'' condition can be mated with the anti-polynomial $\overline{z}^d$.  A \emph{necklace Kleinian reflection group} is the group generated by reflections in the circles of a finite circle packing whose contact graph is \emph{$2$-connected} and \emph{outerplanar}; i.e., the contact graph remains connected if any vertex is deleted, and has a face containing all the vertices on its boundary. The simplest example of a necklace Kleinian reflection group is given by reflections in the sides of a regular ideal $(d+1)$-gon in the unit disk $\D$ (see Definitions~\ref{necklace_schottky_group},~\ref{ideal_group}). This group, which we denote by $\pmb{\Gamma}_{d+1}$, can be thought of as a base point of the space of necklace groups generated by $(d+1)$ circular reflections. In fact, all necklace groups (generated by $(d+1)$ circular reflections) can be obtained by taking the closure of suitable quasiconformal deformations of $\pmb{\Gamma}_{d+1}$ in an appropriate topology. This yields the \emph{Bers compactification} $\overline{\beta(\pmb{\Gamma}_{d+1})}$ of the group $\pmb{\Gamma}_{d+1}$ (see Definitions~\ref{Bers_slice},~\ref{compactify_def}). 

To conformally mate a necklace group $\Gamma$ in $\overline{\beta(\pmb{\Gamma}_{d+1})}$ with an anti-polynomial, we associate a piecewise M{\"o}bius reflection map $\rho_\Gamma$ to $\Gamma$ that is \emph{orbit equivalent} to $\Gamma$ and enjoys Markov properties when restricted to the limit set (see Definition~\ref{reflection_map} and the following discussion). For $\pmb{\Gamma}_{d+1}$, the associated map $\rho_{\pmb{\Gamma}_{d+1}}$, restricted to its limit set, is topologically conjugate to the the anti-polynomial $\overline{z}^d$ on its Julia set. This yields our fundamental dynamical connection between a Kleinian limit set and a Julia set. Furthermore, the existence of the above topological conjugacy allows one to topologically glue the dynamics of $\rho_\Gamma$ on its ``filled limit set'' with the dynamics of $\overline{z}^d$ on its filled Julia set. In the spirit of the classical mating theory, it is then natural to seek a conformal realization of such a \emph{topological mating} (see Subsection~\ref{conformal_mating_subsec} for the definition of conformal mating). We remark that the aforementioned topological conjugacy is not quasisymmetric since it carries parabolic fixed points to hyperbolic fixed points, and hence, classical conformal welding techniques cannot be applied to construct the desired conformal matings. 

The definition of the map $\rho_\Gamma$ (in particular, the fact that it fixes the boundary of its domain of definition) immediately tells us that a conformal realization of the above topological mating must be an anti-meromorphic map defined on (the closure of) a simply connected domain fixing the boundary of the domain pointwise. A characterization of such maps now implies that such an anti-meromorphic map would be the Schwarz reflection map arising from a univalent rational map \cite[Lemma~2.3]{AS} (see Subsection~\ref{sigma_prelim_subsec} for the precise definitions). This observation leads us to the space $\Sigma_d^*$ of univalent rational maps. Indeed, the fact that each member $f$ of $\Sigma_d^*$ has an order $d$ pole at the origin translates to the fact that the associated Schwarz reflection map $\sigma_f$ has a super-attracting fixed point of local degree $d$ (note that $\overline{z}^d$ also has such a super-attracting fixed point in its filled Julia set). On the other hand, the space $\Sigma_d^*$ has a lot in common with the groups in the Bers compactification $\overline{\beta(\pmb{\Gamma}_{d+1})}$ too. In fact, for $f\in\Sigma_d^*$, the complement of $f(\D^*)$ resembles the bounded part of the fundamental domain of a necklace group in $\overline{\beta(\pmb{\Gamma}_{d+1})}$ (compare Figures~\ref{kleinian_cusp} and~\ref{extremal_qd_schwarz}). Using a variety of conformal and quasiconformal techniques, we prove that this resemblance can be used to construct a homeomorphism between the space $\Sigma_d^*$ of univalent rational maps and the Bers compactification $\overline{\beta(\pmb{\Gamma}_{d+1})}$, and the Schwarz reflection maps arising from $\Sigma_d^*$ are precisely the conformal matings of groups in $\overline{\beta(\pmb{\Gamma}_{d+1})}$ with the anti-polynomial $\overline{z}^d$.

\begin{thmx}\label{theorem_A} 
For each $f\in\Sigma_d^*$, there exists a unique $\Gamma_f \in \overline{\beta(\pmb{\Gamma}_{d+1})}$ such that the Schwarz reflection map $\sigma_f$ is a conformal mating of $\Gamma_f$ with $z\mapsto\overline{z}^d$. The map \begin{align} \Sigma_d^* \rightarrow \overline{\beta(\pmb{\Gamma}_{d+1})} \nonumber \\ f\mapsto\Gamma_f \phantom{asds} \nonumber \end{align} is a homeomorphism. 
\end{thmx}

\noindent We remark that when $d=2$, both the spaces $\Sigma_d^*$ and $\overline{\beta(\pmb{\Gamma}_{d+1})}$ are singletons, and the conformal mating statement of Theorem \ref{theorem_A} is given in  \cite[Theorem~1.1]{LLMM1}. 


\begin{figure}[ht!]
{\includegraphics[width=1\textwidth]{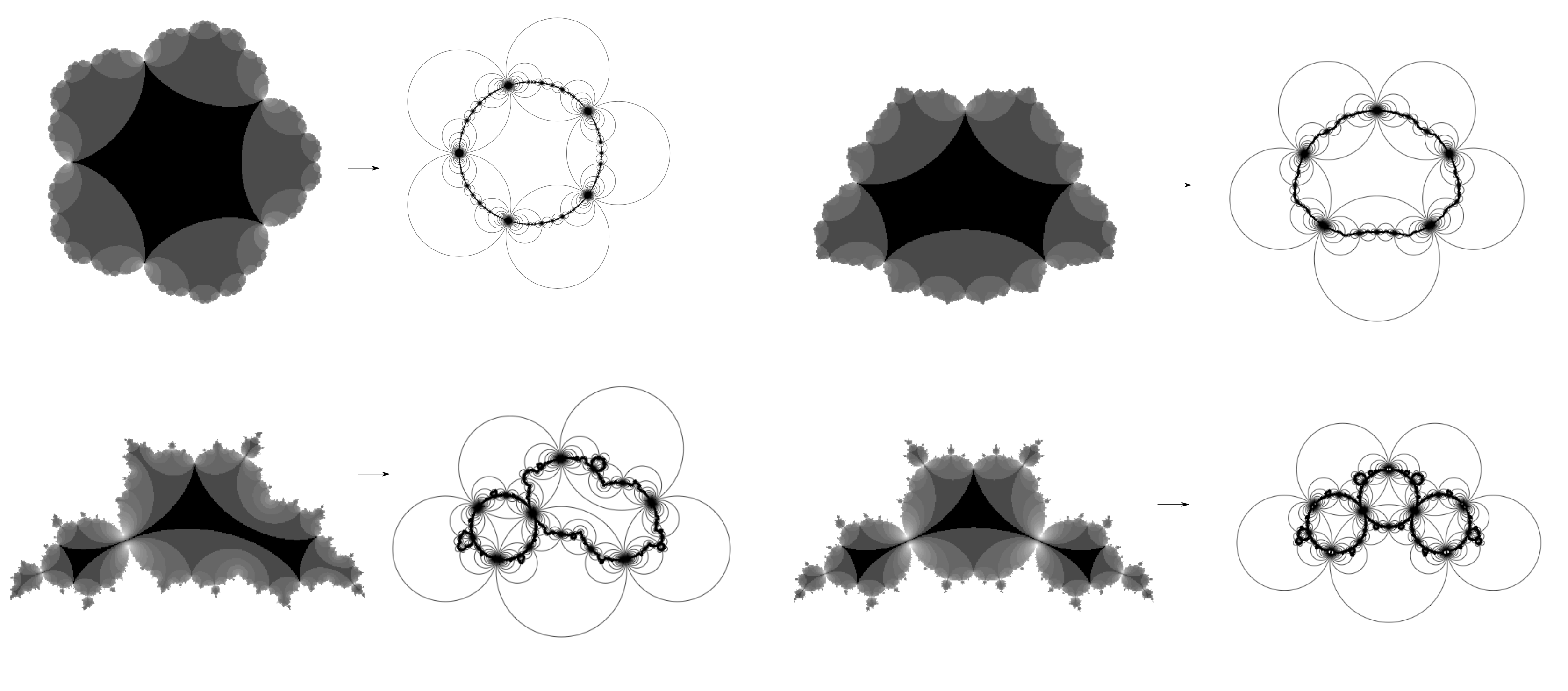}}
\caption{Illustrated is the mapping of Theorem \ref{theorem_A}. The dynamical planes of Schwarz reflection maps of elements in $\Sigma_4^*$ are illustrated next to the limit sets of the corresponding reflection groups in $\beta(\pmb{\Gamma}_{5})$. The top-left entry corresponds to the base points $z\mapsto z-1/(4z^4)$ and $\pmb{\Gamma}_{5}$ in $\Sigma_4^*$, $\beta(\pmb{\Gamma}_{5})$ respectively. The bottom-left and bottom-right dynamical planes lie on the boundary of the parameter spaces.}
\label{fig:circle_packing}
\end{figure}

It is worth mentioning that the homeomorphism between the parameter spaces appearing in Theorem~\ref{theorem_A} has a geometric interpretation. To see this, let us first note that just like the group $\pmb{\Gamma}_{d+1}$ is a natural base point in its Bers compactification, the map $f_0(z)=z-1/dz^d$ can be seen as a base point of $\Sigma_d^*$. In fact, the complement of $f_0(\D^*)$ is a $(d+1)$-gon that is conformally isomorphic to the (closure of the) bounded part of the fundamental domain of $\pmb{\Gamma}_{d+1}$. The pinching deformation technique for the family $\Sigma_d^*$, as developed in \cite{LMM1}, then shows that all other members of $\Sigma_d^*$ can be obtained from $f_0$ by quasiconformally deforming $\C\setminus f_0(\D^*)$ and letting various sides of this $(d+1)$-gon touch. Analogously, all groups in $\overline{\beta(\pmb{\Gamma}_{d+1})}$ can be obtained from $\pmb{\Gamma}_{d+1}$ by quasiconformally deforming the fundamental domain and letting the boundary circles touch (this also has the interpretation of pinching suitable geodesics on a $(d+1)$-times punctured sphere). This suggests that one can define analogues of \emph{Fenchel-Nielsen coordinates} on $\Int{\Sigma_d^*}$ and $\beta(\pmb{\Gamma}_{d+1})$ (the latter is just a real slice of the Teichm{\"u}ller space of $(d+1)$-times punctured spheres) using extremal lengths of path families connecting various sides of the corresponding $(d+1)$-gons. The homeomorphism of Theorem~\ref{theorem_A} is geometric in the sense that it respects these coordinates on $\Int{\Sigma_d^*}$ and $\beta(\pmb{\Gamma}_{d+1})$ (compare the proof of Theorem~\ref{homeo_thm}). 

We also note that the boundary of the Bers slice $\beta(\pmb{\Gamma}_{d+1})$ is considerably simpler than Bers slices of Fuchsian groups; more precisely, all groups on $\partial\beta(\pmb{\Gamma}_{d+1})$ are geometrically finite (or equivalently, cusps), and are obtained by pinching a special collection of curves on a $(d+1)$-times punctured sphere (see the last paragraph of Subsection~\ref{reflection_group_prelim_subsec}). It is this feature of the Bers slices of reflection groups that is responsible for continuity of the dynamically defined map from $\Sigma_d^*$ to $\overline{\beta(\pmb{\Gamma}_{d+1})}$. This should be contrasted with the usual Fuchsian situation where the natural map from one Bers slice to another typically does not admit a continuous extension to the Bers boundaries (see \cite{KeTh}).

We now turn our attention to the other theme of the paper. This is related to the parallel notion of \emph{laminations} that appears in the study of Kleinian groups and polynomial dynamics. The limit set of each group $\Gamma$ in $\overline{\beta(\pmb{\Gamma}_{d+1})}$ is topologically modeled as the quotient of the limit set of $\pmb{\Gamma}_{d+1}$ by a \emph{geodesic lamination} that is invariant under the reflection map $\rho_{\pmb{\Gamma}_{d+1}}$ (see Proposition~\ref{group_lamination_prop} and Remark~\ref{geod_lamination}). Due to the existence of a topological conjugacy between $\rho_{\pmb{\Gamma}_{d+1}}$ and $\overline{z}^d$, this geodesic lamination can be ``pushed forward'' to obtain a $\overline{z}^d$-invariant equivalence relation on the unit circle (such equivalence relations are known as polynomial laminations in holomorphic dynamics). Using classical results from holomorphic dynamics, we show that this $\overline{z}^d$-invariant equivalence relation is realized as the lamination of the Julia set of a degree $d$ anti-polynomial. This leads to our second main result.

\begin{thmx}\label{theorem_B}
Let $\Gamma\in\overline{\beta(\pmb{\Gamma}_{d+1})}$. Then there exists a critically fixed anti-polynomial $p$ of degree $d$ such that the dynamical systems  \begin{align} \rho_{\Gamma}: \Lambda(\Gamma) \rightarrow \Lambda(\Gamma), \nonumber \\ p: \mathcal{J}(p) \rightarrow \mathcal{J}(p) \nonumber \end{align} are topologically conjugate.
\end{thmx} 

We remark that the proof proceeds by showing that the systems in Theorem \ref{theorem_B} are both topologically conjugate to the Schwarz reflection map of an appropriate element of $\Sigma_d^*$ acting on its limit set (see Theorem \ref{theorem_B_modified}). This implies, in particular, that all the three fractals; namely, the Julia set of the anti-polynomial $p$, the limit set of the necklace group $\Gamma$, and the limit set of the Schwarz reflection map of an appropriate element of $\Sigma_d^*$, are homeomorphic. However, the incompatibility of the structures of cusp points on these fractals imply that they are not quasiconformally equivalent; i.e., there is no global quasiconformal map carrying one fractal to another (compare Figures~\ref{kleinian_cusp},~\ref{extremal_qd_schwarz}, and~\ref{fig:cubic_crit_fixed}). Theorem~\ref{theorem_B} plays an important role in the recent work \cite{LMMN}, where limit sets of necklace reflection groups are shown to be conformally removable. One of the main steps in the proof is to show that the topological conjugacy between $p\vert_{\mathcal{J}(p)}$ and $\rho_{\Gamma}\vert_{\Lambda(\Gamma)}$ (provided by Theorem~\ref{theorem_B}) can be extended to a David homeomorphism of the sphere.



Let us now briefly outline the organization of the paper. Section~\ref{Preliminaries} collects fundamental facts and known results about the objects studied in the paper. More precisely, in Subsection~\ref{sigma_prelim_subsec}, we recall the definitions of the basic dynamical objects associated with the space $\Sigma_d^*$ of univalent rational maps. Subsection~\ref{reflection_group_prelim_subsec} introduces the class of reflection groups that will play a key role in the paper. Here we define the Bers slice $\beta(\pmb{\Gamma}_d)$ of the regular ideal polygon reflection group $\pmb{\Gamma}_d$ (following the classical construction of Bers slices of Fuchsian groups), and describe its compactification $\overline{\beta(\pmb{\Gamma}_d)}$ in a suitable space of discrete, faithful representations. To each reflection group $\Gamma\in \overline{\beta(\pmb{\Gamma}_d)}$, we then associate the reflection map $\rho_\Gamma$ that is \emph{orbit equivalent} to the group (this mimics a construction of Bowen and Series \cite{BoSe}). Using the reflection group $\rho_\Gamma$, we formalize the notion of \emph{conformal mating} of a reflection group and an anti-polynomial in Subsection~\ref{conformal_mating_subsec}. Section~\ref{homeomorphism} proves half of Theorem~\ref{theorem_A}; here we prove that there is a natural homeomorphism between the spaces $\Sigma_d^*$ and $\overline{\beta(\pmb{\Gamma}_d)}$. We should mention that the results of Section~\ref{homeomorphism} depend on some facts about the space $\Sigma_d^*$ (and the associated Schwarz reflection maps) whose proofs are somewhat technical and hence deferred to Section~\ref{Conformal_mating}. A recurring difficulty in our study is the unavailability of normal family arguments since Schwarz reflection maps are not defined on all of $\widehat{\mathbb{C}}$. After proving some preliminary results about the topology of the limit set of a Schwarz reflection map arising from $\Sigma_d^*$ in Subsections~\ref{local_connectivity} and~\ref{The Limit Set is the boundary of the Tiling set}, we proceed to the proofs of the statements about $\Sigma_d^*$ that are used in Section~\ref{homeomorphism} (more precisely, Lemma~\ref{fixed_point_ray_lem}, Proposition~\ref{sigma_d^*_connected_thm}, and~{\ref{zero_ray_landing_cusp}). The rest of Section~\ref{Conformal_mating} is devoted to the proof of the conformal mating statement of Theorem~\ref{theorem_A}. This completes the proof of our first main theorem. Finally, in Section~\ref{sullivan_sec}, we use the theory of Hubbard trees for anti-holomorphic polynomials to prove Theorem~\ref{theorem_B}.

\vspace{2mm}

\noindent\textbf{Acknowledgements.} The third author was supported by an endowment from Infosys Foundation.

\section{Preliminaries}
\label{Preliminaries}

\begin{notation} We denote by $\mathbb{D}^*$ the exterior unit disc $\widehat{\mathbb{C}}\setminus\overline{\mathbb{D}}$. The Julia set of a holomorphic or anti-holomorphic polynomial $p:\widehat{\mathbb{C}}\rightarrow\widehat{\mathbb{C}}$ will be denoted by $\mathcal{J}(p)$, and its filled Julia set by $\mathcal{K}(p)$. 

\end{notation}

\subsection{The Space $\Sigma_d^*$ and Schwarz Reflection Maps}\label{sigma_prelim_subsec}

\begin{definition}\label{Sigma_d} We will denote by $\Sigma_d^*$ the following class of rational maps:

\[ \Sigma_d^* := \left\{ f(z)= z+\frac{a_1}{z} + \cdots +\frac{a_d}{z^d} : a_d=-\frac{1}{d}\textrm{ and } f|_{\mathbb{D}^*} \textrm{ is conformal.}\right\}. \]

\end{definition}

Note that for each $d\geq2$, the space $\Sigma_d^*$ can be regarded as a slice of the space of schlicht functions:
$$\Sigma := \left\{ f(z)= z+\frac{a_1}{z} + \cdots +\frac{a_d}{z^d}+\cdots :\ f\vert_{\D^*} \textrm{ is conformal}\right\}.$$

We endow $\Sigma_d^*$ with the topology of coefficient-wise convergence. Clearly, this topology is equivalent to that of uniform convergence on compact subsets of $\D^*$. 

\begin{definition}\label{Schwarz_reflection} Given $f\in\Sigma_d^*$, we define the associated \emph{Schwarz reflection map} $\sigma_f:f(\D^*)\to\widehat{\C}$ by the following diagram:

\[
  \begin{tikzcd}
    \mathbb{D}^* \arrow{r}{z\mapsto1/\bar{z}} & \mathbb{D} \arrow{d}{f} \\
     f(\mathbb{D}^*) \arrow{r}{\sigma_f} \arrow{u}{f^{-1}}  & \widehat{\mathbb{C}}
  \end{tikzcd}
\] 
\end{definition}

The map $\sigma_f:\sigma_f^{-1}(f(\D^*))\to f(\D^*)$ is a proper branched covering map of degree $d$ (branched only at $\infty$), and $\sigma_f:\sigma_f^{-1}(\Int{f(\D^*)^c})\to\Int{f(\D^*)^c}$ is a degree $(d+1)$ covering map. 

We also note that $\infty$ is a super-attracting fixed point of $\sigma_f$; more precisely, $\infty$ is a fixed critical point of $\sigma_f$ of multiplicity $(d-1)$. 

 \begin{definition} Let $f\in\Sigma_d^*$. We define the \emph{basin of infinity} for $\sigma_f$ as  \[ \mathcal{B}_\infty(\sigma_f) := \{ z \in \widehat{\mathbb{C}} : \sigma_f^{\circ n}(z) \xrightarrow{n\rightarrow\infty} \infty \}.  \] \end{definition}
 
\begin{rem}\label{bottcher_coordinate} 
Let $f\in\Sigma_d^*$. Since $\sigma_f$ has no critical point other than $\infty$ in $\mathcal{B}_\infty(\sigma_f)$, the proof of \cite[Theorem~9.3]{MR2193309} may be adapted to show the existence of a \emph{B\"ottcher coordinate} for $\sigma_f$: a conformal map \begin{equation}\label{bottcher_conjugacy} \phi_{\sigma_f}: \mathbb{D}^* \rightarrow \mathcal{B}_\infty(\sigma_f) \textrm{ such that } \phi_{\sigma_f}^{-1} \circ \sigma_f \circ \phi_{\sigma_f}(u)= \overline{u}^d,\ \forall\ u\in\D^*.\end{equation} Since \[\sigma_f(z)=-\frac{\overline{z}^d}{d}+O(\overline{z}^{d-1})\ \textrm{as}\ z\to\infty,
\] we may choose $\phi_{\sigma_f}$ such that \begin{align}\label{bot_norm}\phi_{\sigma_f}'(\infty)=d^{\frac{1}{d-1}}e^{\frac{i\pi}{d+1}}.\end{align} As in  \cite[Theorem~9.3]{MR2193309}, any B\"ottcher coordinate for $\sigma_f$ is unique up to multiplication by a $d+1^{\textrm{st}}$ root of unity. Thus, (\ref{bot_norm}) determines a \emph{unique} B\"ottcher coordinate $\phi_{\sigma_f}$ which we will henceforth refer to as the B\"ottcher coordinate for $\sigma_f$.
\end{rem}

The set $\widehat{\mathbb{C}}\setminus f(\mathbb{D}^*)$ is called the \emph{droplet}, or \emph{fundamental tile}, and is denoted by $T(\sigma_f)$. By \cite[Proposition~2.8]{LMM1} and \cite[Lemma~2.4]{2014arXiv1411.3415L}, the curve $\partial T=f(\mathbb{T})$ has $(d+1)$ distinct cusps and at most $(d-2)$ double points.The \emph{desingularized droplet} $T^o(\sigma_f)$ is defined as \[ T^o(\sigma_f) := T(\sigma_f) \setminus \left\{ \zeta: \zeta \textrm{ is a cusp or double point of } f(\mathbb{T}) \right\}. \] 

\begin{definition}\label{tiling_set_def}
The \emph{tiling set} $\mathcal{T}_\infty(\sigma_f)$ is defined as:

\[ \mathcal{T}_\infty(\sigma_f) := T^o(\sigma_f) \cup \left\{ z\in\widehat{\mathbb{C}} : \sigma_f^{\circ n}(z) \in T^o(\sigma_f) \textrm{ for some } n\geq1 \right\}. \] Lastly, we define the \emph{limit set} of $\sigma_f$ by $\Lambda(\sigma_f):=\partial \mathcal{T}_\infty(\sigma_f)$.
\end{definition}

For more details on the space $\Sigma_d^*$ and the associated Schwarz reflection maps, we refer the readers to \cite{LMM1}.

\subsection{Reflection Groups and the Bers Slice}\label{reflection_group_prelim_subsec}


\begin{notation} We denote by $\textrm{Aut}^\pm(\widehat{\C})$ be the group of all M{\"o}bius and anti-M{\"o}bius automorphisms of $\widehat{\C}$. \end{notation}

\begin{definition}\label{Kleinian_reflection_group} 
A discrete subgroup $\Gamma$ of $\textrm{Aut}^\pm(\widehat{\C})$ is called a \emph{Kleinian reflection group} if $\Gamma$ is generated by reflections in finitely many Euclidean circles.

\end{definition}

\begin{rem}\label{discreteness_rem} For a Euclidean circle $C$, consider the upper hemisphere $S\subset\mathbb{H}^3:=\{(x,y,t)\in\R^3: t>0\}$ such that $\partial S\cap\partial\mathbb{H}^3= C$. Reflection in the Euclidean circle $C$ extends naturally to reflection in $S$, and defines an orientation-reversing isometry of $\mathbb{H}^3$.  Hence, a Kleinian reflection $\Gamma$ group can be thought of as a $3$-dimensional hyperbolic reflection group. 

Since a Kleinian reflection group is discrete, by \cite[Part~II, Chapter~5, Proposition~1.4]{Vin1}, we can choose its generators to be reflections in Euclidean circles $C_1,\cdots,C_d$ such that: \begin{align}\label{condition_2} \tag{$\star$} 
\textrm{ For each } i\textrm{, the closure of the bounded component of }\widehat{\mathbb{C}}\setminus C_i \textrm{ does not contain any other }C_j. \end{align}




\noindent We will always assume that a chosen generating set for a Kleinian reflection group $\Gamma$ satisfies Conditions~($\star$). 
\end{rem}


\begin{definition}\label{limit_regular_def} Let $\Gamma$ be a Kleinian reflection group. The \emph{domain of discontinuity} of $\Gamma$, denoted $\Omega(\Gamma)$, is the maximal open subset of $\widehat{\C}$ on which the elements of $\Gamma$ form a normal family. The \emph{limit set} of $\Gamma$, denoted by $\Lambda(\Gamma)$, is defined by $\Lambda(\Gamma):=\widehat{\mathbb{C}}\setminus\Omega(\Gamma)$. 


\end{definition}

\begin{figure}[ht!]
\includegraphics[width=0.45\textwidth]{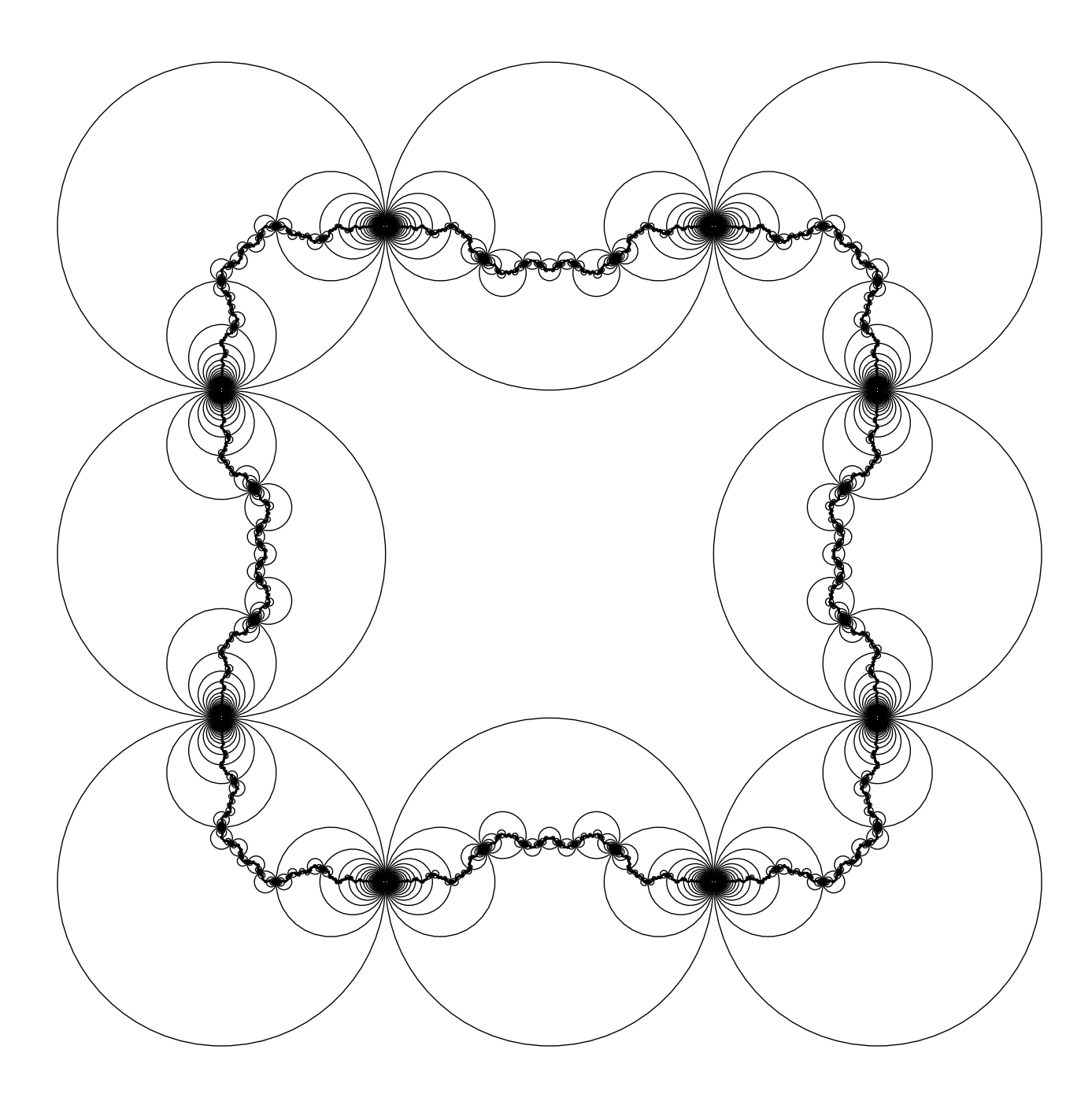}\ \includegraphics[width=0.45\textwidth]{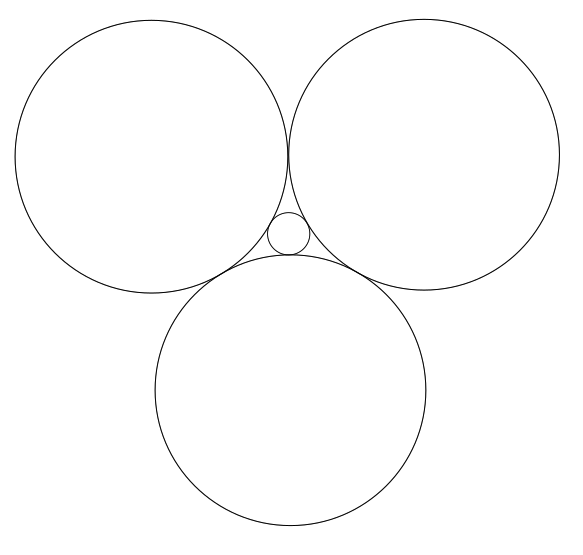}
\caption{On the left is an interior necklace group, but the group generated by the circles pictured on the right violates condition (2) of Definition \ref{necklace_schottky_group}.  }
\label{fig:necklace_schottky_group}
\end{figure}

For a Euclidean circle $C$, the bounded complementary component of $C$ will be called the \emph{interior} of $C$, and will be denoted by $\Int{C}$.

\begin{definition}\label{necklace_schottky_group} 
Let $\Gamma$ be a Kleinian reflection group. We say $\Gamma$ is a \emph{necklace group} (see Figure \ref{fig:necklace_schottky_group}) if it can be generated by reflections in Euclidean circles $C_1,\cdots, C_d$ such that: 
\begin{enumerate} 
\item each circle $C_i$ is tangent to $C_{i+1}$ (with $i+1$ taken mod $d$), 

\item the boundary of the unbounded component of $\widehat{\mathbb{C}}\setminus\cup_iC_i$ intersects each $C_i$, and 

\item the circles $C_i$ have pairwise disjoint interiors. 
\end{enumerate} 
If, furthermore, $C_{i-1}$ and $C_{i+1}$ are the only circles to which any $C_i$ is tangent, then $\Gamma$ is an \emph{interior necklace group}. 
\end{definition}

\begin{rem} In Definition~\ref{necklace_schottky_group}, Condition (2) ensures that each circle $C_i$ is ``seen'' from $\infty$ - see Figure \ref{fig:necklace_schottky_group}. When choosing a generating set for a necklace group, we always assume the generating set is chosen so as to satisfy Conditions (1)-(3), and the circles $C_1,\cdots, C_d$ are labelled clockwise around $\infty$. We note that a necklace group $\Gamma$ generated by reflections in $d$ Euclidean circles is isomorphic to the free product of $d$ copies of $\Z/2\Z$. 
 \end{rem}

\begin{notation} Given a necklace group $\Gamma$ with generating set given by reflections in circles $C_1,\cdots, C_d$,  let  \[ \mathcal{F}_\Gamma:=\displaystyle\widehat{\C}\setminus\left(\bigcup_{i=1}^{d}(\Int{C_i}\cup\{ C_j \cap C_i : j\not =i \}) \right).  \]  \end{notation}

\begin{prop}\label{fund_dom_prop}
Let $\Gamma$ be a necklace group. Then $\mathcal{F}_\Gamma$ is a fundamental domain for $\Gamma$.
\end{prop}
\begin{proof}
Let $\mathcal{P}_\Gamma$ be the convex hyperbolic polyhedron (in $\mathbb{H}^3$) whose relative boundary in $\mathbb{H}^3$ is the union of the hyperplanes $S_i$ (see Remark~\ref{discreteness_rem}). Then, by \cite[Part~II, Chapter~5, Theorem~1.2]{Vin1}, $\mathcal{P}_\Gamma$ is a fundamental domain for the action of $\Gamma$ on $\mathbb{H}^3$. It now follows that $\mathcal{F}_\Gamma=\overline{\mathcal{P}_\Gamma}\cap\Omega(\Gamma)$ (where the closure is taken in $\Omega(\Gamma)\cup\mathbb{H}^3$) is a fundamental domain for the action of $\Gamma$ on $\Omega(\Gamma)$ \cite[\S 3.5]{Marden}. 
\end{proof}

\noindent It will be useful in our discussion to have a canonical interior necklace group to refer to:

\begin{definition}\label{ideal_group} Consider the Euclidean circles $\mathbf{C}_1,\cdots, \mathbf{C}_{d}$ where $\mathbf{C}_j$ intersects $|z|=1$ at right-angles at the roots of unity $\exp{(\frac{2\pi i\cdot(j-1)}{d})}$, $\exp{(\frac{2\pi i\cdot j}{d})}$. Let $\rho_j$ be the reflection map in the circle $\mathbf{C}_j$. By \cite[Part~II, Chapter~5, Theorem~1.2]{Vin1}, this defines a necklace group $$\pmb{\Gamma}_d:=\langle \rho_1,\cdots, \rho_{d}: \rho_1^2=\cdots=\rho_{d}^2=1\rangle,$$ that acts on the Riemann sphere.
\end{definition}

\begin{definition} Let $\Gamma$ be a discrete subgroup of $\textrm{Aut}^\pm(\widehat{\C})$. An isomorphism 
$$
\xi:\pmb{\Gamma}_{d}\rightarrow \Gamma
$$
is said to be \emph{weakly type-preserving}, or \emph{w.t.p.}, if 
\begin{enumerate}\upshape
\item $\xi(g)$ is orientation-preserving if and only if $g$ is orientation-preserving, and 

\item $\xi(g)\in\Gamma$ is a parabolic M{\"o}bius map for each parabolic M{\"o}bius map $g\in \pmb{\Gamma}_{d}$.
\end{enumerate}
\end{definition}

\noindent In order to construct the \emph{Bers slice} of the group $\pmb{\Gamma}_d$ and describe its compactification, we need to define a representation space for $\pmb{\Gamma}_d$. For necklace groups, the information encoded by a representation (defined below) is equivalent to the data given by a labeling of the underlying circle packing. We will see in Section~\ref{homeomorphism} that working with the space of representations (as opposed to the space of necklace groups without a labeling of the underlying circle packings) is crucial for the homeomorphism statement of Theorem~\ref{theorem_A} (compare Figure~\ref{non-uniqueness}).

\begin{definition}\label{D_def} We define 
\[ \mathcal{D}(\pmb{\Gamma}_{d}):= \lbrace   \xi :  \pmb{\Gamma}_{d} \to \Gamma\vert\ \Gamma \textrm{ is a discrete subgroup of } \textrm{Aut}^\pm(\widehat{\C}),\\\textrm{ and } \xi \textrm{ is a w.t.p. isomorphism} \rbrace.\]

\noindent We endow $\mathcal{D}(\pmb{\Gamma}_d)$ with the topology of \emph{algebraic convergence}: we say that a sequence $(\xi_n)_{n=1}^{\infty}\subset\mathcal{D}(\pmb{\Gamma}_d)$ converges to $\xi\in\mathcal{D}(\pmb{\Gamma}_d)$ if $\xi_n(\rho_i)\to\xi(\rho_i)$ coefficient-wise (as $n\rightarrow\infty$) for $i=1,\cdots,d$.

\end{definition}

\begin{rem}\label{rep_var_def_rem} Let $\xi\in\mathcal{D}(\pmb{\Gamma}_{d})$. Since for each $i\in \Z/d\Z$, the M{\"o}bius map $\rho_i\circ\rho_{i+1}$ is parabolic (this follows from the fact that each $\mathbf{C}_i$ is tangent to $\mathbf{C}_{i+1}$), the w.t.p. condition implies that $\xi(\rho_i)\circ\xi(\rho_{i+1})$ is also parabolic. As each $\xi(\rho_i)$ is an anti-conformal involution, it follows that $\xi(\rho_i)$ is M{\"o}bius conjugate to the circular reflection $z\mapsto 1/\overline{z}$ or the antipodal map $z\mapsto -1/\overline{z}$. A straightforward computation shows that the composition of $-1/\overline{z}$ with either the reflection or the antipodal map with respect to any circle has two distinct fixed points in $\widehat{\C}$, and hence not parabolic. Therefore, it follows that no $\xi(\rho_i)$ is M{\"o}bius conjugate to the antipodal map $-1/\overline{z}$. Hence, each $\xi(\rho_i)$ must be the reflection in some Euclidean circle $C_i$. Thus, $\Gamma=\xi(\pmb{\Gamma}_{d})$ is generated by reflections in the circles $C_1, \cdots, C_{d}$. The fact that $\xi(\rho_i)\circ\xi(\rho_{i+1})$ is parabolic now translates to the condition that each $C_i$ is tangent to $C_{i+1}$ (for $i\in \Z/d\Z$). However, new tangencies among the circles $C_i$ may arise. Moreover, that $\xi$ is an isomorphism rules out non-tangential intersection between circles $C_i$, $C_j$ (indeed, a non-tangential intersection between $C_i$ and $C_j$ would introduce a new relation between $\xi(\rho_i)$ and $\xi(\rho_j)$, compare \cite[Part~II, Chapter~5, \S 1.1]{Vin1}). Therefore, $\Gamma=\xi(\pmb{\Gamma}_{d})$ is a Kleinian reflection group satisfying properties (1) and (3) of necklace groups.
\end{rem}

\begin{definition} Let $\tau$ be a conformal map defined in a neighborhood of $\infty$ with $\tau(\infty)=\infty$. We will say $\tau$ is \emph{tangent to the identity at $\infty$} if $\tau'(\infty)=1$. We will say $\tau$ is \emph{hydrodynamically normalized} if \[ \tau(z)=z+O(1/z) \textrm{ as } z\rightarrow\infty. \] 
\end{definition}

\begin{definition}\label{Bers_slice}
Let $\textrm{Bel}_{\pmb{\Gamma}_d}$ denote those Beltrami coefficients $\mu$ invariant under $\pmb{\Gamma}_d$, satisfying $\mu=0$ a.e. on $\D^*$. Let $\tau_\mu:\mathbb{C}\rightarrow\mathbb{C}$ denote the quasiconformal integrating map of $\mu$, with the hydrodynamical normalization. The \emph{Bers slice} of $\Gamma$ is defined as \[ \beta(\pmb{\Gamma}_d) := \{\xi\in\mathcal{D}(\pmb{\Gamma}_d)\textrm{ | } \xi(g)= \tau_\mu\circ g\circ \tau_\mu^{-1}\ \textrm{for all}\ g\in\pmb{\Gamma}_d,\ \textrm{where}\ \mu \in  \textrm{Bel}_{\pmb{\Gamma}_d} \}. \]
\end{definition}

\begin{rem} There is a natural free $\textrm{PSL}_2(\mathbb{C})$-action on $\mathcal{D}(\pmb{\Gamma}_d)$ given by conjugation, and so it is natural to consider the space $\textrm{AH}(\pmb{\Gamma}_d):=\mathcal{D}(\pmb{\Gamma}_d)/\textrm{PSL}_2(\mathbb{C})$. The following definition of the Bers slice, where no normalization for $\tau_\mu$ is specified, is more aligned with the classical Kleinian group literature: 
\begin{align}\label{alternative_bers} 
\{\xi\in\textrm{AH}(\pmb{\Gamma}_d)\textrm{ | } \xi(g)= \tau_\mu\circ g\circ \tau_\mu^{-1}\ \textrm{for all}\ g\in\pmb{\Gamma}_d,\ \textrm{where}\ \mu \in  \textrm{Bel}_{\pmb{\Gamma}_d} \} \tag{$\star$}.
\end{align}
Our Definition~\ref{Bers_slice} of $\beta(\pmb{\Gamma}_d)$ is simply a canonical choice of representative from each equivalence class of (\ref{alternative_bers}), and will be more appropriate for the present work. 
\end{rem}

\begin{lem}\label{qc_images} Let $\mu\in\emph{Bel}_{\pmb{\Gamma}_d}$, and $\tau_\mu:\mathbb{C}\rightarrow\mathbb{C}$ an integrating map. Then $ \tau_\mu \circ \pmb{\Gamma}_d \circ \tau_\mu^{-1}$ is a necklace group.
\end{lem}

\begin{proof} 
By definition, the maps $\tau_\mu \circ \rho_i \circ \tau_\mu^{-1}$ generate the group $ \tau_\mu \circ \pmb{\Gamma}_d \circ \tau_\mu^{-1}$. By invariance of $\mu$, each $\tau_\mu \circ \rho_i \circ \tau_\mu^{-1}$ is an anti-conformal involution of $\widehat{\mathbb{C}}$, hence an anti-M\"obius transformation. Since $\tau_\mu \circ \rho_i \circ \tau_\mu^{-1}$ fixes $\tau_\mu(C_i)$, and interchanges its two complementary components, it follows that $\tau_\mu(C_i)$ is a Euclidean circle and hence $\tau_\mu \circ \rho_i \circ \tau_\mu^{-1}$ is reflection in the circle $\tau_\mu(C_i)$. One readily verifies that the circles $\tau_\mu(C_i)$ satisfy the conditions of Definition~\ref{necklace_schottky_group}, and so the result follows.
\end{proof}

\begin{prop}\label{compactify_prop}
The Bers slice $\beta(\pmb{\Gamma}_d)$ is pre-compact in $\mathcal{D}(\pmb{\Gamma}_d)$, and for each $\xi\in\overline{\beta(\pmb{\Gamma}_d)}$, the group $\xi(\pmb{\Gamma}_d)$ is a necklace group.
\end{prop}

\begin{proof} Let $(\xi_n)_{n=1}^{\infty}$ be a sequence in $\beta(\pmb{\Gamma}_d)$, and $\tau_n: \mathbb{C}\rightarrow\mathbb{C}$ the associated quasiconformal maps as in Definition \ref{Bers_slice}. Since each $\tau_n$ is conformal in $\mathbb{D}^*$ and is hydrodynamically normalized, by a standard normal family result (see \cite[Theorem~1.10]{CG1} for instance) there exists a conformal map $\tau_{\infty}$ of $\mathbb{D}^*$ such that $\tau_n \rightarrow \tau_{\infty}$ uniformly on compact subsets of $\mathbb{D}^*$, perhaps after passing to a subsequence which we reenumerate $(\tau_n)$. 

By Lemma \ref{qc_images}, each $\tau_n(\mathbf{C}_i)$ is a Euclidean circle which we denote by $C_i^n$. Since $\tau_n \rightarrow \tau_{\infty}$ uniformly on compact subsets of $D^*$, $\tau_{\infty}(C_i\cap\mathbb{D}^*)$ must be a subarc of a Euclidean circle which we denote by $C_i^\infty$. Denote furthermore by $\rho_i^n$, $\rho_i^{\infty}$ the reflections in the circles $C_i^n$, $C_i^{\infty}$ (respectively), and by $\Gamma_\infty$ the group generated by reflections in the circles $(C_i^\infty)_{i=1}^d$. Let $\xi_{\infty}: \pmb{\Gamma}_d\rightarrow \Gamma_\infty$ be the homomorphism defined by $\xi_{\infty}(\rho_i):=\rho_i^{\infty}$. We see that $C_i^n\rightarrow C_i^\infty$ as $n\rightarrow\infty$ in the Hausdorff sense, whence it follows that $\rho_i^n\rightarrow\rho_i^{\infty}$. This proves algebraic convergence $(\xi_n) \rightarrow \xi_{\infty}$. Hausdorff convergence of $C_i^n\rightarrow C_i^\infty$ also implies that each $C_i^{\infty}$ intersects tangentially with $C_{i+1}^\infty$, so that $\xi_\infty$ is weakly type preserving. Similar considerations show that the circles $C_i^\infty$ have pairwise disjoint interiors, and the boundary of the unbounded component of $\widehat{\mathbb{C}}\setminus\cup_i C_i^\infty$ intersects each $C_i^\infty$. In particular, there cannot be any non-tangential intersection among the circles $C_i^\infty$. It now follows that $\xi_\infty$ is indeed an isomorphism, and $\Gamma_\infty=\xi_\infty(\pmb{\Gamma}_d)$ is a necklace group. 
\end{proof}

\begin{definition}\label{compactify_def} We refer to $\overline{\beta(\pmb{\Gamma}_d)} \subset \mathcal{D}(\pmb{\Gamma}_d)$ as the \emph{Bers compactification} of the Bers slice $\beta(\pmb{\Gamma}_d)$. We refer to $\overline{\beta(\pmb{\Gamma}_d)}\setminus\beta(\pmb{\Gamma}_d)$ as the \emph{Bers boundary}.
\end{definition}

\begin{rem}\label{rep_group_identify_rem} 
We will often identify $\xi\in\overline{\beta(\pmb{\Gamma}_d)}$ with the group $\Gamma:=\xi(\pmb{\Gamma}_d)$, and simply write $\Gamma\in\overline{\beta(\pmb{\Gamma}_d)}$, but always with the understanding of an associated representation $\xi:\pmb{\Gamma}_d \rightarrow \Gamma$. Since $\xi$ is completely determined by its action on the generators $\rho_1,\cdots, \rho_{d}$ of $\pmb{\Gamma}_d$, this is equivalent to remembering the `labeled' circle packing $C_1,\cdots,C_d$, where $\xi(\rho_i)$ is reflection in the circle $C_i$, for $i=1,\cdots,d$. 
\end{rem}

\begin{rem}\label{apollo_rem}
The Apollonian gasket reflection group (see the right-hand side of Figure \ref{fig:necklace_schottky_group}) is an example of a Kleinian reflection group in $\mathcal{D}(\pmb{\Gamma}_d)\setminus\overline{\beta(\pmb{\Gamma}_d)}$.
\end{rem}

\begin{notation} For $\Gamma\in\overline{\beta(\pmb{\Gamma}_d)}$, we denote the component of $\Omega(\Gamma)$ containing $\infty$ by $\Omega_\infty(\Gamma)$.
\end{notation}

\begin{prop}\label{inv_comp_limit_boundary}
Let $\Gamma\in\overline{\beta(\pmb{\Gamma}_d)}$. Then the following hold true.
\begin{enumerate}
\item $\Omega_\infty(\Gamma)$ is simply connected, and $\Gamma$-invariant.

\item $\partial\Omega_\infty(\Gamma)=\Lambda(\Gamma)$.
 
\item $\Lambda(\Gamma)$ is connected, and locally connected.

\item  All bounded components of $\Omega(\Gamma)$ are Jordan domains.
\end{enumerate}
\end{prop}
\begin{proof}
1) It is evident from the construction of the Bers compactification that for each $\Gamma\in\overline{\beta(\pmb{\Gamma}_d)}$, there is a conformal map from $\D^*$ onto $\Omega_\infty(\Gamma)$ that conjugates the action of $\pmb{\Gamma}_d$ on $\D^*$ to that of $\Gamma$ on $\Omega_\infty(\Gamma)$. Hence, $\Omega_\infty(\Gamma)$ is simply connected, and invariant under $\Gamma$.

2) This follows from (1) and the fact that the boundary of an invariant component of the domain of discontinuity is the entire limit set. 

3) Connectedness of $\Lambda(\Gamma)$ follows from (2) and that $\Omega_\infty(\Gamma)$ is simply connected. For local connectivity, first note that the index two Kleinian subgroup $\Gamma^+$ consisting of words of even length of $\Gamma$ is geometrically finite. Then $\Lambda(\Gamma)=\Lambda(\Gamma^+)$, hence $\Lambda(\Gamma^+)$ is connected. Since $\Gamma^+$ is geometrically finite with a connected limit set, it now follows from \cite{AnMa} that $\Lambda(\Gamma^+)=\Lambda(\Gamma)$ is locally connected.

4) By (3), each component of $\Omega(\Gamma)\setminus\Omega_\infty(\Gamma)$ is simply connected with a locally connected boundary. That such a component $\mathcal{U}$ is Jordan follows from the fact that $\partial\mathcal{U}\subset\Lambda(\Gamma)=\partial\Omega_\infty(\Gamma)$.
\end{proof}

\noindent To a group $\Gamma\in\overline{\beta(\pmb{\Gamma}_d)}$, we now associate a reflection map $\rho_\Gamma$ that will play an important role in the present work.

\begin{definition}\label{reflection_map} Let $\Gamma\in\overline{\beta(\pmb{\Gamma}_d)}$, generated by reflections $(r_i)_{i=1}^d$ in circles $(C_i)_{i=1}^d$. We define the associated \emph{reflection map} $\rho_\Gamma$ by: \begin{align} \rho_{\Gamma} : \bigcup_{i=1}^d \overline{\textrm{int}(C_i)} \rightarrow \widehat{\mathbb{C}} \nonumber \hspace{17mm} \\ \hspace{10mm} z\longmapsto r_i(z) \textrm{        if } z \in \overline{\textrm{int}(C_i)}.  \nonumber\end{align}
\end{definition}

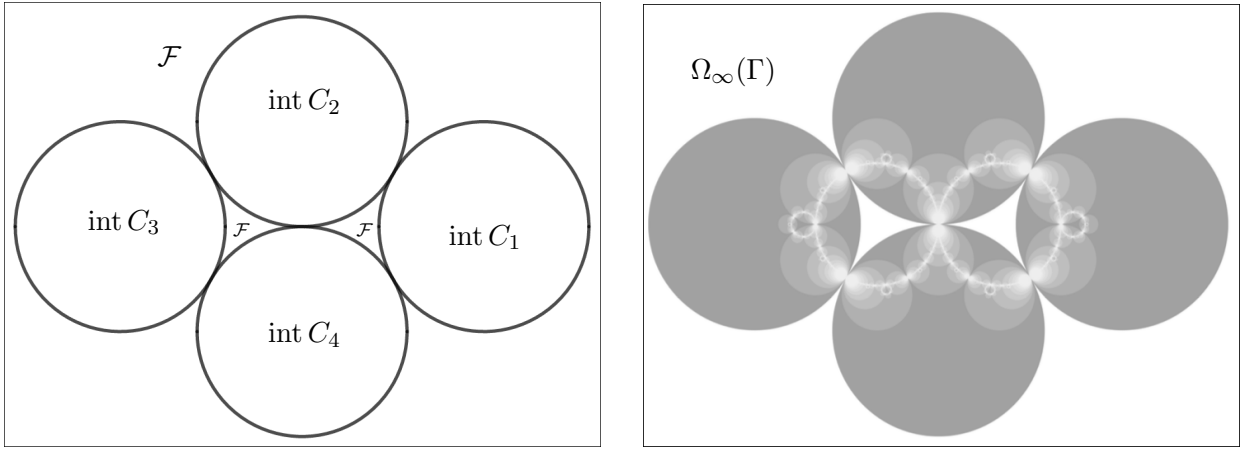
\begin{figure}[ht!]
\begin{tikzpicture}
 \node[anchor=south west,inner sep=0] at (-1,-1) {\includegraphics[width=0.48\textwidth]{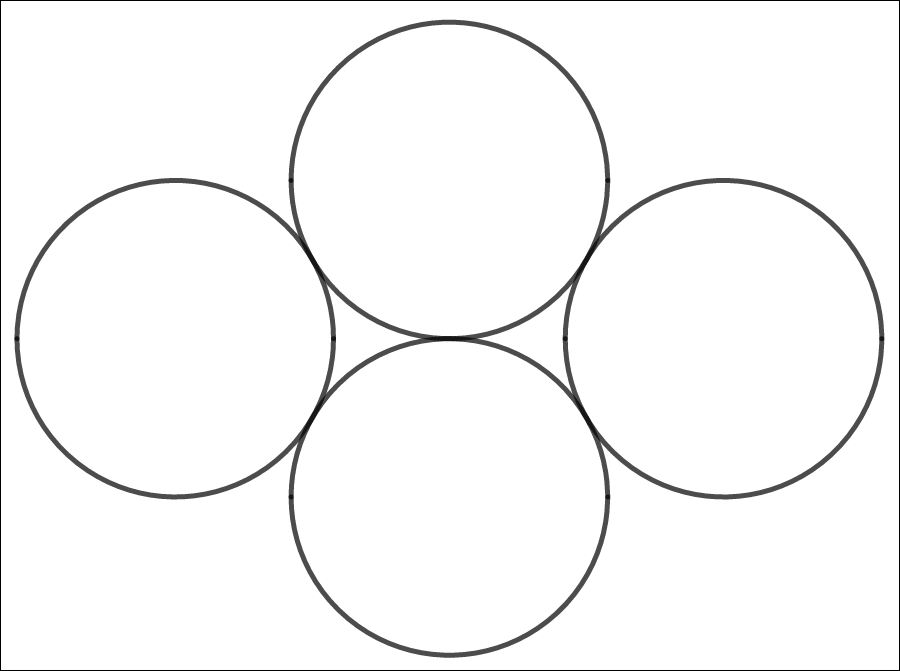}};
\node[anchor=south west,inner sep=0] at (7.5,-1) {\includegraphics[width=0.48\textwidth]{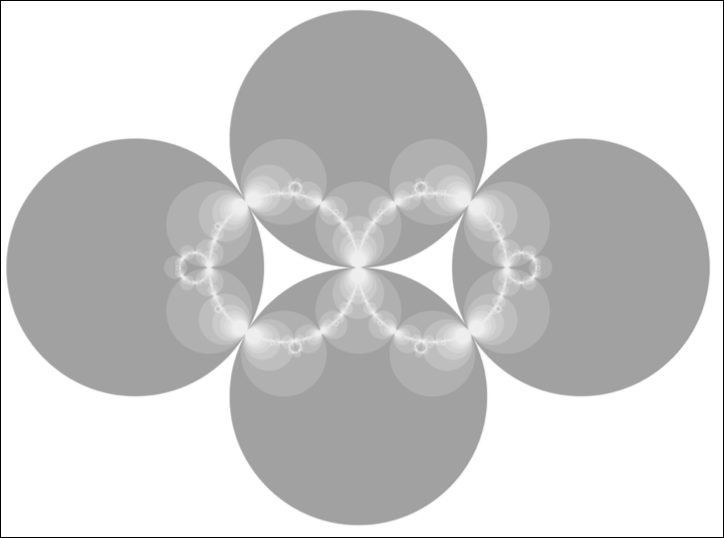}};
\node at (5.4,1.8) {$\Int{C_1}$};
\node at (3,3.6) {$\Int{C_2}$};
\node at (0.6,2) {$\Int{C_3}$};
\node at (3,0.5) {$\Int{C_4}$};
\node at (1.2,4.2) {$\mathcal{F}$};
\node at (3.8,1.9) {\begin{tiny}$\mathcal{F}$\end{tiny}};
\node at (2.16,1.9) {\tiny$\mathcal{F}$};
\node at (8.7,4) {$\Omega_\infty(\Gamma)$};
\end{tikzpicture}
\caption{Left: The circles $C_{i}$ generate a Kleinian reflection group $\Gamma\in\partial\beta(\pmb{\Gamma}_4)$. The map $\rho_\Gamma$ is defined piece-wise on the union of the closed disks $\overline{\Int{C_i}}$. The fundamental domain $\mathcal{F}=\mathcal{F}_\Gamma$ (for the action of $\Gamma$ on $\Omega(\Gamma)$) is the complement of these open disks with the singular boundary points removed. The connected components of $\mathcal{F}$ are marked. Right: The unbounded component of the domain of discontinuity $\Omega(\Gamma)$ is $\Omega_\infty(\Gamma)$. Every point in $\Omega(\Gamma)$ escapes to $\mathcal{F}$ under iterates of $\rho_\Gamma$. The point of tangential intersection of $C_{2}$ and $C_{4}$ is the fixed point of an accidental parabolic of the index two Kleinian subgroup $\widetilde{\Gamma}$.}
\label{kleinian_cusp}
\end{figure}

\begin{definition} Let $\Gamma$ be a Kleinian reflection group, and $f: D\rightarrow\widehat{\mathbb{C}}$ a mapping defined on a domain $D$. We say that $\Gamma$ and $f$ are \emph{orbit-equivalent} if for any two points $z,w\in\widehat{\mathbb{C}}$, there exists $g\in\Gamma$ with $g(z)=w$ if and only if there exist non-negative integers $n_1, n_2$ such that $f^{\circ n_1}(z)=f^{\circ n_2}(w)$.
\end{definition}

\begin{prop}\label{orbit_equiv_prop}
Let $\Gamma\in\overline{\beta(\pmb{\Gamma}_d)}$. The map $\rho_\Gamma$ is orbit equivalent to $\Gamma$ on $\widehat{\C}$.
\end{prop}
\begin{proof}
Suppose $z,w\in\widehat{\C}$ are such that there exist $n_1, n_2\in\mathbb{N}$ such that $\rho^{n_1}(z)=\rho^{n_2}(w)$. Since $\rho_\Gamma$ acts by the generators $r_i$ of the group $\Gamma$, it follows directly that there exists $g\in\Gamma$ with $g(z)=w$. Conversely, let $z,w\in\widehat{\C}$ be such that there exists $g\in \Gamma$ with $g(z)=w$. By definition, we have that $g=r_{s_1}r_{s_2}\cdots r_{s_n}$, for some $s_1,\cdots s_n\in\{1,\cdots, d\}$. Suppose first that $n=1$. Note that either $z$ or $w$ must belong to $\overline{\Int{C_{s_1}}}$. Since $r_{s_1}(z)=w$ implies $r_{s_1}(w)=z$, there is no loss of generality in assuming that $z\in\overline{\Int{C_{s_1}}}$. Now, the condition $r_{s_1}(z)=w$ can be written as $\rho_\Gamma(z)=w$. The case $n>1$ now follows by induction. 
\end{proof}

\begin{notation} For $\Gamma\in\overline{\pmb{\Gamma}_d}$, we will denote by $T^o(\Gamma)$ the union of all bounded components of the fundamental domain $\mathcal{F}_\Gamma$ (see Proposition~\ref{fund_dom_prop}), and by $\Pi^o(\Gamma)$ the unique unbounded component of $\mathcal{F}_\Gamma$. We also set $$T(\Gamma):=\overline{T^o(\Gamma)},\ \textrm{and}\ \Pi(\Gamma):=\overline{\Pi^o(\Gamma)}.$$
\end{notation}

\begin{rem}\label{notation_analogy}
The set $T(\Gamma)$ should be thought of as the analogue of a droplet $T(\sigma_f)$ (this analogy will become transparent in Proposition~\ref{existence_of_group}). On the other hand, the notation $\Pi(\Gamma)$ is supposed to remind the readers that (the closure of) the unbounded component of $\mathcal{F}_\Gamma$ is a ``polygon."
\end{rem}

\begin{prop}\label{regular_fund_pre_image_prop} Let $\Gamma\in\overline{\beta(\pmb{\Gamma}_d)}$. Then:
\[ \Omega(\Gamma)=\displaystyle\bigcup_{n\geq0}\rho_\Gamma^{-n}(\mathcal{F}_\Gamma)\emph{, and }\Omega_\infty(\Gamma)=\displaystyle\bigcup_{n\geq0}\rho_\Gamma^{-n}(\Pi^o(\Gamma)).\] In particular, $\Lambda(\Gamma)$ is completely invariant under $\rho_\Gamma$.
\end{prop}
\begin{proof}
This follows from Propositions~\ref{fund_dom_prop},~\ref{inv_comp_limit_boundary},~\ref{orbit_equiv_prop}.\end{proof}

\begin{rem}\label{defn_of_conjug} Let $\Gamma\in\overline{\beta(\pmb{\Gamma}_d)}$. We now briefly describe the covering properties of $\rho_\Gamma:\Lambda(\Gamma)\to\Lambda(\Gamma)$. To this end, first note that $$\Lambda(\Gamma)=\bigcup_{i=1}^{d}\left( \overline{\Int{C_{i}}}\cap\Lambda(\Gamma)\right)$$
(see Figure~\ref{kleinian_cusp}). The interiors of these $d$ ``partition pieces'' are disjoint, and $\rho_\Gamma$ maps each of them injectively onto the union of the others. This produces a Markov partition for the degree $d$ orientation-reversing covering map $\rho_\Gamma: \Lambda(\Gamma)\to\Lambda(\Gamma)$. In the particular case of the base group $\pmb{\Gamma}_d$, the above discussion yields a Markov partition $$\mathbb{T}=\displaystyle\bigcup_{j=1}^{d}\left[\exp{\left(\frac{2\pi i(j-1)}{d}\right)},\exp{\left(\frac{2\pi ij}{d}\right)}\right]$$ of the map $\rho_{\pmb{\Gamma}_d}:\mathbb{T}\to\mathbb{T}$. Note that the expanding map $$\overline{z}^{d-1}:\mathbb{T}\to\mathbb{T},$$ or equivalently, $$m_{-(d-1)}:\R/\Z\to\R/\Z,\ \theta\mapsto-(d-1)\theta$$ also admits the same Markov partition with the same transition matrix (identifying $\mathbb{T}$ with $\R/\Z$). Following \cite[\S 3.2]{LLMM1}, one can define a homeomorphism $$\mathcal{E}_{d-1}:\mathbb{T}\to\mathbb{T}$$ via the coding maps of $\rho_{\pmb{\Gamma}_d}\vert_{\mathbb{T}}$ and $\overline{z}^{d-1}\vert_{\mathbb{T}}$ such that $\mathcal{E}_{d-1}$ maps $1$ to $1$, and conjugates $\rho_{\pmb{\Gamma}_d}$ to $\overline{z}^{d-1}$ (or $m_{-(d-1)}$). Since both $\rho_{\pmb{\Gamma}_d}$ and $\overline{z}^{d-1}$ commute with the complex conjugation map and $\mathcal{E}_{d-1}$ fixes $1$, one sees that $\mathcal{E}_{d-1}$ commutes with the complex conjugation map as well.
\end{rem}

The next result provides us with a model of the dynamics of $\rho_\Gamma$ on the limit set $\Lambda(\Gamma)$ as a quotient of the action of $\rho_{\pmb{\Gamma}_d}$ on the unit circle.

\begin{prop}\label{group_lamination_prop} Let $\Gamma \in \overline{\beta(\pmb{\Gamma}_d)}$. There exists a conformal map $\phi_\Gamma: \mathbb{D}^* \rightarrow \Omega_{\infty}(\Gamma)$ such that 
\begin{align}\label{group_conjugacy}    
\rho_{\pmb{\Gamma}_d} (z) = \phi_\Gamma^{-1} \circ \rho_{\Gamma} \circ \phi_\Gamma(z) \textrm{, for } z\in  \mathbb{D}^*\setminus \Int{\Pi(\pmb{\Gamma}_d)}.   
\end{align}  
The map $\phi_\Gamma$ extends continuously to a semi-conjugacy $\phi_\Gamma: \mathbb{T} \rightarrow \Lambda(\Gamma)$ between $\rho_{\pmb{\Gamma}_d}|_{\mathbb{T}}$ and $\rho_{\Gamma}|_{\Lambda(\Gamma)}$, and $\phi_\Gamma$ sends cusps of $\partial \Pi(\pmb{\Gamma}_d)$ to cusps of $\partial \Pi(\Gamma)$ with labels preserved. 
\end{prop}

\begin{proof} Recall that $\pmb{\Gamma}_d$ is generated by reflections $\rho_i$ in circles $(\mathbf{C}_i)_{i=1}^d$. It follows from Propositions~\ref{fund_dom_prop} and~\ref{inv_comp_limit_boundary} that $\Pi^o(\pmb{\Gamma}_d)$ and $\Pi^o(\Gamma)$ are fundamental domains for the actions of $\pmb{\Gamma}_d$ and $\Gamma$ on $\mathbb{D}^*$ and $\Omega_{\infty}(\Gamma)$, respectively. By the proofs of Lemma~\ref{qc_images} and Proposition~\ref{compactify_prop}, there is a conformal mapping $\phi_\Gamma: \Int{\Pi(\pmb{\Gamma}_d)} \rightarrow \Int{\Pi(\Gamma)}$ whose extension to $\partial\Pi(\pmb{\Gamma}_d)$ is a label-preserving homeomorphism onto $\partial\Pi(\Gamma)$. Thus, by the Schwarz reflection principle, we may extend $\phi_\Gamma$ to a conformal mapping $\phi_\Gamma: \mathbb{D}^* \rightarrow \Omega_{\infty}(\Gamma)$ which satisfies (\ref{group_conjugacy}) by construction. 

Note that $\partial\Omega_{\infty}(\Gamma)=\Lambda(\Gamma)$. By local connectedness of $\Lambda(\Gamma)$ (see Proposition~\ref{inv_comp_limit_boundary}), the map $\phi_\Gamma$ extends continuously to a semi-conjugacy $\mathbb{T}\rightarrow\Lambda(\Gamma)$.
\end{proof}

\noindent We will now introduce the notion of a \emph{label-preserving homeomorphism}, which will play an important role in the proof of Theorem~\ref{theorem_A}.

\begin{rem}\label{labeling} For $f_0(z):=z-1/(dz^d)$, label the non-zero critical points of $f_0$ as $\xi_1^{f_0},\cdots, \xi_{d+1}^{f_0}$ in counter-clockwise order with $\xi_1^{f_0}=e^{\frac{i\pi}{d+1}}$. Note that the critical points of $f$ vary continuously depending on $f\in\Sigma_d^*$, and $f\in\Sigma_d^*$ can not have a double critical point on $\mathbb{T}$. Since $\Sigma_d^*$ is connected by Proposition~\ref{sigma_d^*_connected_thm}, there is a unique labeling $\xi_1^f, \cdots, \xi_{d+1}^f$ of critical points of any $f\in\Sigma_d^*$ such that $f\mapsto \xi_i^f$ is continuous (for $i=1,\cdots,d+1$). This in turn determines a labeling of the cusps $\zeta_i^f:=f(\xi_i^f)$ of $f(\mathbb{T})$ such that $f\mapsto \zeta_i^f$ is continuous ($i=1,\cdots,d+1$).

Similarly, label the cusps of $\partial T(\pmb{\Gamma}_d)$ as $\eta_1, \cdots, \eta_{d}$ in counter-clockwise order with $\eta_1=1$. This determines a labeling of cusps of $\partial T(\Gamma)$ for any $\Gamma \in \overline{\beta(\pmb{\Gamma}_d)}$ as the group $\Gamma$ is the image under a representation of $\pmb{\Gamma}_d$. 
\end{rem}

\begin{definition}\label{label-preserving} Let $f\in\Sigma_d^*$ and $\Gamma \in \overline{\beta(\pmb{\Gamma}_{d+1})}$. We say that a homeomorphism $h: T(\Gamma) \rightarrow T(\sigma_f)$ is \emph{label-preserving} if $h$ maps cusps of $\partial T(\Gamma)$ to cusps of $\partial T(\sigma_f)$, and  $h$ preserves the labeling of cusps of $\partial T(\Gamma)$ and $\partial T(\sigma_f)$. 

Similarly, for $f, f'\in\Sigma_d^*$ (respectively, for $\Gamma,\Gamma' \in \overline{\beta(\pmb{\Gamma}_{d+1})}$), a homeomorphism $h: T(\sigma_f) \rightarrow T(\sigma_{f'})$ (respectively, $h: T(\Gamma) \rightarrow T(\Gamma')$) is called \emph{label-preserving} if $h$ maps the boundary cusps to the boundary cusps preserving their labels.\end{definition}

We conclude this subsection with a discussion of the connection between the Bers slice of the reflection group $\pmb{\Gamma}_d$ and a classical Teichm{\"u}ller space. Let $\pmb{\Gamma}_d^+$ be the index two subgroup of $\pmb{\Gamma}_d$ consisting of all M{\"o}bius maps in $\pmb{\Gamma}_d$. Then, $\pmb{\Gamma}_d^+$ is Fuchsian group (it preserves $\D$ and $\D^*$). Using Proposition~\ref{fund_dom_prop}, it is seen that the top and bottom surfaces $S^+:=\D^*/\pmb{\Gamma}_d^+$ and $S^-:=\D/\pmb{\Gamma}_d^+$ associated with the Fuchsian group $\pmb{\Gamma}_d^+$ are $d$ times punctured spheres. Moreover, the anti-M{\"o}bius reflection $\rho_i$ in the circle $\mathbf{C}_i$ descends to anti-conformal involutions on $S^\pm$ fixing all the punctures (the resulting involution is independent of $i\in\{1,\cdots,d\}$). We will denote this involution on $S^-$ by $\iota$.

By definition, each $\xi\in\mathcal{D}(\pmb{\Gamma}_d)$ defines a discrete, faithful, w.t.p. representation of $\pmb{\Gamma}_d^+$ into $\textrm{PSL}_2(\C)$. If $\xi\in\beta(\pmb{\Gamma}_d)$, then $\xi$ is induced by a quasiconformal map that is conformal on $\D^*$. Hence, such a representation of $\pmb{\Gamma}_d^+$ lies in the Bers slice of $\pmb{\Gamma}_d^+$. Thus, $\beta(\pmb{\Gamma}_d)$ embeds into the Teichm{\"u}ller space of a $d$ times punctured sphere. 

On the other hand, each $\xi\in\overline{\beta(\pmb{\Gamma}_d)}\setminus\beta(\pmb{\Gamma}_d)$ induces a representation of $\pmb{\Gamma}_d^+$ that lies on the boundary of the Bers slice of the Fuchsian group $\pmb{\Gamma}_d^+$. The index two Kleinian group $\Gamma^+$ of $\Gamma:=\xi(\pmb{\Gamma}_d)$ is geometrically finite (a fundamental polyhedron for the action of $\Gamma^+$ on $\mathbb{H}^3$ is obtained by ``doubling'' a fundamental polyhedron for $\Gamma$, and hence it has finitely many sides). In fact, $\Gamma^+$ is a \emph{cusp group} that is obtained by pinching a special collection of simple closed curves on $S^-$. Indeed, since $S^-$ is equipped with a natural involution $\iota$, any $\pmb{\Gamma}_d$-invariant Beltrami coefficient on $\D$ induces an $\iota$-invariant Beltrami coefficient on $S^-$. Hence, the simple closed geodesics on $S^-$ that can be pinched via quasiconformal deformations with $\pmb{\Gamma}_d$-invariant Beltrami coefficients are precisely the ones invariant under $\iota$. Moreover, the $\iota$-invariant simple closed geodesics on $S^-$ bijectively correspond to pairs of non-tangential circles $\mathbf{C}_i$ and $\mathbf{C}_j$; more precisely, they are the projections to $S^-$ of hyperbolic geodesics of $\D$ with end-points at the two fixed points of the loxodromic M{\"o}bius map $\rho_i\circ\rho_j$. Hence, a group $\Gamma^+$ on the Bers boundary is obtained as a limit of a sequence of quasiFuchsian deformations of $\pmb{\Gamma}_d^+$ that pinch a disjoint union of $\iota$-invariant simple, closed, essential geodesics on the bottom surface $S^-$ without changing the (marked) conformal equivalence class of the top surface $S^+$. If $\xi(\rho_i)$ is reflection in the circle $C_i$ (for $i=1,\cdots,d$), then a point of intersection of some $C_i$ and $C_j$ with $j\neq i,i\pm1\ (\textrm{mod}\ d)$ corresponds to an \emph{accidental parabolic} $\xi(\rho_i\circ\rho_j)$ for $\xi(\pmb{\Gamma}_d^+)$. Furthermore, the quotient $$\mathcal{M}(\Gamma^+):=\left(\mathbb{H}^3\cup\Omega(\Gamma^+)\right)/\Gamma^+$$ is an infinite volume $3$-manifold whose conformal boundary $\partial\mathcal{M}(\Gamma^+):=\Omega(\Gamma^+)/\Gamma^+$ consists of finitely many punctured spheres.

\subsection{Conformal Mating}\label{conformal_mating_subsec} In this Subsection we define the notion of conformal mating in Theorem \ref{theorem_A}. Our definitions follow \cite{AFST_2012_6_21_S5_839_0}, to which we refer for a more extensive discussion of conformal mating. 

\begin{notation} For $\Gamma\in\overline{\beta(\pmb{\Gamma}_{d})}$, recall $\Omega_\infty(\Gamma)$ denotes the unbounded component of $\Omega(\Gamma)$. We let $\mathcal{K}(\Gamma):=\mathbb{C}\setminus\Omega_\infty(\Gamma)$.
\end{notation}

\begin{rem}\label{setting_up_equiv_reltn} Let $w\mapsto p(w)$ be a monic, anti-holomorphic polynomial such that $\mathcal{J}(p)$ is connected and locally connected. Let $d:=\textrm{deg}(p)$, and denote by $\phi_p: \mathbb{D}^*\rightarrow \mathcal{B}_\infty(p)$ the B\"ottcher coordinate for $p$ such that $\phi_p'(\infty)=1$. We note that since $\partial \mathcal{K}(p)=\mathcal{J}(p)$ is locally connected by assumption, it follows that $\phi_p$ extends to a continuous semi-conjugacy between $z\mapsto\overline{z}^{d}|_{\mathbb{T}}$ and $p|_{\mathcal{J}(p)}$. Now let $\Gamma\in\overline{\beta(\pmb{\Gamma}_{d+1})}$. As was shown in Proposition~\ref{group_lamination_prop}, there is a natural continuous semi-conjugacy $\phi_\Gamma: \mathbb{T} \rightarrow \Lambda(\Gamma)$ between $\rho_{\pmb{\Gamma}_{d+1}}|_{\mathbb{T}}$ and $\rho_\Gamma|_{\Lambda(\Gamma)}$. Recall from Remark \ref{defn_of_conjug} that $\mathcal{E}_d: \mathbb{T} \rightarrow \mathbb{T}$ is a topological conjugacy between $\rho_{\pmb{\Gamma}_{d+1}}\vert_{\mathbb{T}}$ and $z\mapsto\overline{z}^{d}|_{\mathbb{T}}$.
 \end{rem}

\begin{definition}\label{conf_mating_equiv_reltn} 
Let notation be as in Remark~\ref{setting_up_equiv_reltn}. We define an equivalence relation $\sim$ on $\mathcal{K}(\Gamma) \sqcup \mathcal{K}(p)$ by specifying $\sim$ is generated by $\phi_\Gamma(t)\sim\phi_p(\overline{\mathcal{E}_d(t)})$ for all $t\in\mathbb{T}$.
\end{definition}

\begin{definition}\label{mating} Let $\Gamma\in\overline{\beta(\pmb{\Gamma}_{d+1})}$, $p$ a monic, anti-holomorphic polynomial such that $\mathcal{J}(p)$ is connected and locally connected, and $f\in\Sigma_d^*$. We say that $\sigma_f$ is a \emph{conformal mating} of $\Gamma$ with $p$ if there exist continuous maps \[ \psi_p: \mathcal{K}(p) \rightarrow \widehat{\mathbb{C}}\setminus \mathcal{T}_\infty(\sigma_f) \textrm{ and } \psi_\Gamma: \mathcal{K}(\Gamma) \rightarrow \overline{\mathcal{T}_\infty(\sigma_f)},  \] conformal on $\Int{\mathcal{K}(p)}$, $\Int{\mathcal{K}(\Gamma)}$, respectively, such that \begin{enumerate} \item $\psi_p\circ p(w)=\sigma_f\circ\psi_p(w)$ for $w\in \mathcal{K}(p)$, \item$\psi_\Gamma: T(\Gamma) \rightarrow T(\sigma_f)$ is label-preserving and $\psi_\Gamma\circ \rho_\Gamma(z)=\sigma_f\circ\psi_\Gamma(z)$ for $z\in \mathcal{K}(\Gamma)\setminus \Int{T^o(\Gamma)}$, \item $\psi_\Gamma(z)=\psi_p(w) \textrm{ if and only if } z\sim w$ where $\sim$ is as in Definition \ref{conf_mating_equiv_reltn}. \end{enumerate}
\end{definition}

\subsection{Convergence of Quadrilaterals}\label{quad_conv_subsec}

\noindent We conclude Section \ref{Preliminaries} by recalling a notion of convergence for quadrilaterals (see \cite[\S I.4.9]{MR0344463}) which will be useful to us in the proof of Theorem \ref{theorem_A}. We will usually denote a topological quadrilateral by $Q$, and its modulus by $M(Q)$.

\begin{definition}\label{definition_of_quadrilateral_convergence} The sequence of quadrilaterals $Q_n$ (with a-sides $a_i^n$ and b-sides $b_i^n$, $i=1,2$, $n\in\mathbb{N}$) converges to the quadrilateral $Q$ (with a-sides $a_i$ and b-sides $b_i$, $i=1,2$) if to every $\varepsilon>0$ there corresponds an $n_{\varepsilon}$ such that for $n\geq n_{\varepsilon}$, every point of $a_i^n$, $b_i^n$, $i=1,2$, and every interior point of $Q_n$ has a spherical distance of at most $\varepsilon$ from $a_i$, $b_i$, and $Q$, respectively. 
\end{definition}

\begin{thm}\label{convergence_of_quadrilaterals}\cite[\S I.4.9]{MR0344463} If the sequence of quadrilaterals $Q_n$ converges to a quadrilateral $Q$, then \[ \lim_{n\rightarrow\infty} M(Q_n)=M(Q). \]
\end{thm}

\section{A Homeomorphism Between Parameter Spaces}
\label{homeomorphism}

The purpose of this Section is to define the mapping in Theorem \ref{theorem_A} and prove that it is a homeomorphism. We will prove the conformal mating statement in Theorem \ref{theorem_A} in Section \ref{Conformal_mating}. First we will need the following rigidity result.

\begin{prop}\label{uniqueness} Let $\Gamma$, $\Gamma'$ be necklace groups. Suppose there exist homeomorphisms \[ h_1: T(\Gamma) \rightarrow T(\Gamma')\emph{ and }h_2: \Pi(\Gamma) \rightarrow \Pi(\Gamma')  \] which agree on cusps of $\partial T(\Gamma)$, and map cusps of $\partial T(\Gamma)$ to cusps of $\partial T(\Gamma')$. Suppose furthermore that $h_1$, $h_2$ are conformal on $T^o(\Gamma')$, $F^o(\Gamma')$, respectively. Then $h_1$, $h_2$ are restrictions of a common $M \in \emph{Aut}(\mathbb{C})$ such that \[\Gamma'=M\circ\Gamma\circ M^{-1}.\] 
\end{prop}

\begin{proof} By iterated Schwarz reflection, we may extend the disjoint union of the maps $h_1$, $h_2$ to a conformal isomorphism of the ordinary sets $\Omega(\Gamma)$, $\Omega(\Gamma')$. Since $\Gamma$, $\Gamma'$ are geometrically finite, the conclusion then follows from \cite[Theorem~4.2]{MR783351}.
\end{proof}

\begin{prop}\label{existence_of_group} Let $f\in\Sigma_d^*$. There exists a unique $\Gamma_f\in \overline{\beta(\pmb{\Gamma}_{d+1})}$ such that there is a label-preserving homeomorphism \[ h: T(\Gamma_f) \rightarrow T(\sigma_f) \] with $h$ conformal on \emph{int }$T(\Gamma_f)$. 
\end{prop}


\begin{proof}[Proof of Existence.] We first assume that $f(\mathbb{T})$ has no double points. Let \[ g: T(\pmb{\Gamma}_{d+1}) \rightarrow T(\sigma_f) \] be a label-preserving diffeomorphism such that \[||g_{\overline{z}}/g_{z}||_{L^{\infty}(T(\pmb{\Gamma}_{d+1}))}<1.\] Define a Beltrami coefficient $\mu_g$ by \[ \mu_g(u):=g_{\overline{z}}(u)/g_{z}(u) \textrm{ for } u\in T(\pmb{\Gamma}_{d+1}),  \] and

\[ \mu_g(u) := \begin{cases} 
      \mu_g(r_i^{\circ n}(u)) & \textrm{ if } u \in r_i^{-n}(T(\pmb{\Gamma}_{d+1})) \textrm{ for } 1\leq i\leq d+1 \textrm{ and } n\geq1, \\
      0 & \textrm{otherwise.}
   \end{cases}
\]

\noindent Denote by $\tau_g:\widehat{\mathbb{C}}\rightarrow\widehat{\mathbb{C}}$ the integrating map of $\mu_g$, normalized so that \[\tau_g(z)=z+O(1/|z|) \textrm{ as } z\rightarrow\infty.\] We claim that $\Gamma_f:=\tau_g\circ\pmb{\Gamma}_{d+1}\circ\tau_g^{-1}$ satisfies the conclusions of Proposition \ref{existence_of_group}. Indeed, \[ \tau_g\circ\pmb{\Gamma}_{d+1}\circ\tau_g^{-1} \in \beta({\pmb{\Gamma}_{d+1}}) \] since $\mu_g\equiv0$ on $\D^*$. The map \[ h := g\circ\tau_g^{-1}: T(\Gamma_f) \rightarrow T(\sigma_f) \] is conformal on  $\textrm{int } T_{\Gamma_f}$  since $\tau_g$ is the integrating map for $g_{\overline{z}}/g_z$. Lastly, we see that $h$ is label-preserving since $\tau_g^{-1}$ and $g$ are both label-preserving by definition.

Next we consider the case that $f(\mathbb{T})$ has at least one double point. We claim the existence of $\Gamma \in \overline{\beta(\pmb{\Gamma}_{d+1})}$ such that there is a label-preserving diffeomorphism $ g: T(\Gamma) \rightarrow T(\sigma_f). $ Given the existence of such a $\Gamma$, the same quasiconformal deformation argument as above produces the desired group $\Gamma_f$ and homeomorphism $h$.

The existence of such a $\Gamma$ may be proven by pinching geodesics on the $(d+1)$-times punctured sphere $\D/\pmb{\Gamma}_d^+$ (where $\pmb{\Gamma}_d^+$ is the index $2$ Kleinian subgroup of $\pmb{\Gamma}_{d+1}$ consisting of orientation-preserving automorphisms of $\mathbb{C}$), or adapting the techniques used in the proof of \cite[Theorem~4.11]{LMM1}. Alternatively, we may prove the existence of $\Gamma$ by associating a planar vertex $v_i$, for $1\leq i \leq d+1$, to each analytic arc connecting two cusps of $f(\mathbb{T})$, as in Figure \ref{fig:circle_packing}. Connect two vertices $v_i$, $v_j$ by an edge if and only if the corresponding analytic arcs have non-empty intersection. This defines a simplicial 2-complex $K$ in the plane. $K$ is a combinatorial closed disc, and hence \cite[Proposition~6.1]{MR2131318} shows that there is a circle packing $(C_i')_{i=1}^{d+1}$ of $\mathbb{D}$ for $K$, with each $C_i'$ tangent to $\partial\mathbb{D}$. Quasiconformally deforming this circle packing group so that there is a label-preserving conformal map to $\Pi(\pmb{\Gamma}_{d+1})$ gives the desired $\Gamma\in\overline{\beta(\pmb{\Gamma}_{d+1})}$ (up to M{\"o}bius conjugacy).
\end{proof}

\begin{figure}[ht!]
{\includegraphics[width=1\textwidth]{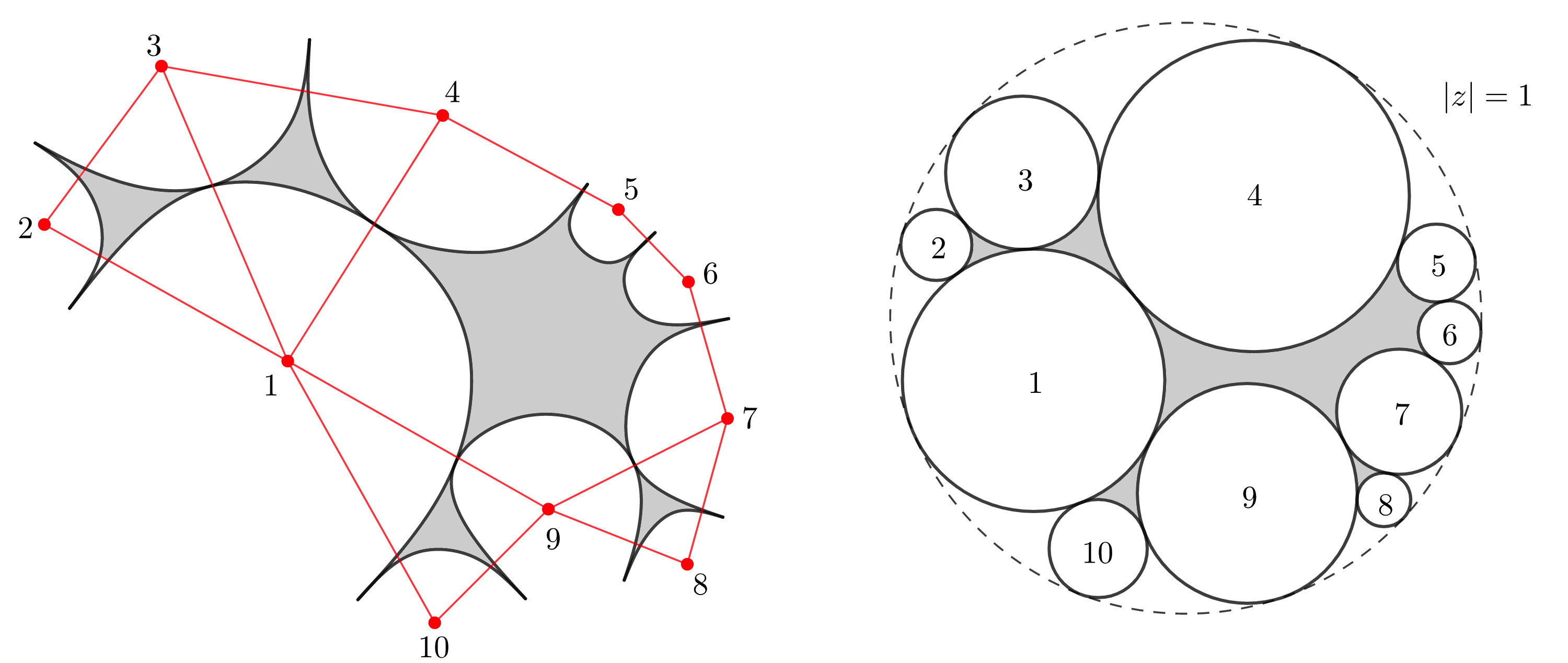}}
\caption{Illustrated is the procedure of associating a circle packing to an element of $\Sigma_d^*$.}
\label{fig:circle_packing}
\end{figure}

\begin{proof}[Proof of Uniqueness.] If $\Gamma$, $\Gamma'$ both satisfy the conclusions of the Proposition, we may take $h_1$, $h_2$ as in Proposition \ref{uniqueness}, where $h_2(z)=z+O(1/z)$ as $z\rightarrow\infty$ since $\Gamma$, $\Gamma' \in\overline{\beta(\pmb{\Gamma}_{d+1})}$. Thus as $h_2$ extends to an automorphism of $\mathbb{C}$ by Proposition \ref{uniqueness}, it follows that $h_2=\textrm{id}$, and hence $\Gamma=\Gamma'$. 
\end{proof}

\begin{rem} The requirement that $h$ be label-preserving is essential to the uniqueness statement in the conclusion of Proposition \ref{existence_of_group}: see Figure \ref{non-uniqueness}. \end{rem}

\begin{figure}[ht!]
\begin{tikzpicture}
\draw (1,0) circle (1);
\draw (1,2.4) circle (1);
\draw (2,1.2) circle (0.55);
\draw (0,1.2) circle (0.55);

\node at (1,2.4) {$C_1$};
\node at (0,1.2) {$C_2$};
\node at (1,0) {$C_3$};
\node at (2,1.2) {$C_4$};

\draw (-4.4,0.2) circle (0.55);
\draw (-4.4,2.2) circle (0.55);
\draw (-3.2,1.2) circle (1);
\draw (-5.6,1.2) circle (1);

\node at (-4.4,0.2) {$C_3$};
\node at (-4.4,2.2) {$C_1$};
\node at (-3.2,1.2) {$C_4$};
\node at (-5.6,1.2) {$C_2$};

  \draw [-] (-11,3.2) to [out=300,in=240] (-9.6,3.2);
  \draw [-] (-11,0.4) to [out=60,in=120] (-9.6,0.4);
  \draw [-] (-11,3.2) to [out=300,in=60] (-11,0.4); 
   \draw [-] (-9.6,0.4) to [out=120,in=240] (-9.6,3.2);
   
\node at (-10.4,0) {$f(\mathbb{T})$};
\node at (-9.2,3.2) {$\zeta_1^f$};
\node at (-11.4,3.2) {$\zeta_2^f$};
\node at (-11.4,0.4) {$\zeta_3^f$};
\node at (-9.2,0.4) {$\zeta_4^f$};

\end{tikzpicture}
\caption{ Let $f(z):=z+t/z-1/(3z^3)\in \Sigma_3^*$ with $t>0$. Then there are exactly two elements of $\overline{\beta(\pmb{\Gamma}_{4})}$ whose corresponding interior fundamental domains are conformally isomorphic (with cusps preserved) to $T(\sigma_f)$. However, only one of these conformal isomorphisms is label-preserving.}
\label{non-uniqueness}
\end{figure}
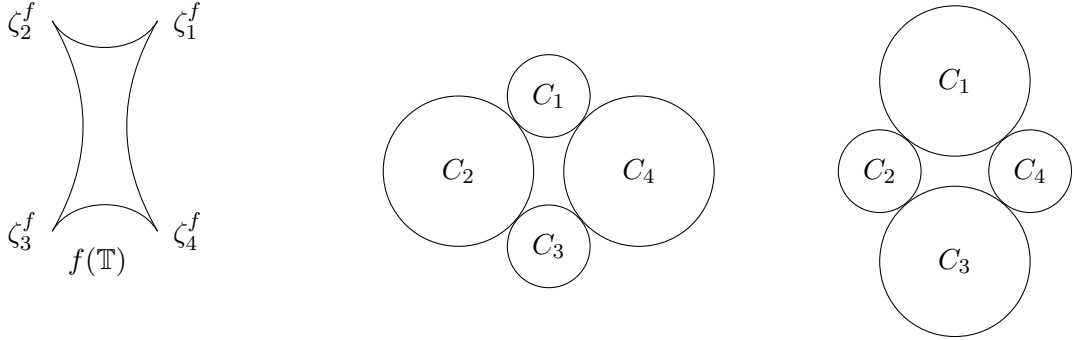

\begin{prop}\label{bijection_statement} The mapping \begin{align} \Sigma_d^* \rightarrow \overline{\beta(\pmb{\Gamma}_{d+1})}   \tag{$\star$} \label{map_defn} \\ f\mapsto\Gamma_f \phantom{aasds} \nonumber \end{align} defined in Proposition \ref{existence_of_group} is a bijection.
\end{prop}

\begin{proof} We sketch a proof of surjectivity of (\ref{map_defn}). Suppose first that $\Gamma\in\overline{\beta(\pmb{\Gamma}_{d+1})}$ is an interior necklace group, and let $f_0(z):=z-1/(dz^d)  \in \Sigma_d^*$. Pull back the standard conformal structure on $T^o(\Gamma)$ by a quasiconformal mapping $\mathfrak{h}: T(\sigma_{f_0}) \rightarrow T^o(\Gamma)$ which preserves vertices, spread this conformal structure under the action of $\sigma_f$ and extend elsewhere by the standard conformal structure, then straighten. This gives the desired element of $\Sigma_d^*$ which maps to $\Gamma$ (see the proof of \cite[Theorem~4.11]{LMM1} for details on quasiconformal deformations of $f$). If $\Gamma\in\overline{\beta(\pmb{\Gamma}_{d+1})}$ is not an interior necklace group, $\Gamma$ still satisfies Condition (2) of Definition \ref{necklace_schottky_group} by Proposition \ref{compactify_prop}, and so by \cite[Theorem~4.11]{LMM1} there exists $f\in\Sigma_d^*$ and a quasiconformal mapping $\mathfrak{h}: T(\sigma_f) \rightarrow T(\Gamma)$ preserving singularities, whence the above arguments apply. 

We now show injectivity of (\ref{map_defn}). Let $f$, $f' \in \Sigma_d^*$ such that $\Gamma_f=\Gamma_{f'}$. Recall from Remark \ref{bottcher_coordinate} that the B\"ottcher coordinates for $\sigma_f$, $\sigma_{f'}$ are both tangent to $z\mapsto wz$ at $\infty$ for the same $w$. Thus there is a conjugacy $\Psi$ between $\sigma_f$, $\sigma_{f'}$ in a neighborhood of $\infty$ satisfying $\Psi'(\infty)=1$. Since $\Gamma_f=\Gamma_{f'}$, there is a label-preserving conformal isomorphism of $T(\sigma_f)\rightarrow T(\sigma_{f'})$, which defines $\Psi$ in a finite part of the plane (disjoint from the neighborhood of $\infty$ in which $\Psi$ is a conjugacy). Since the $0$-rays for $\sigma_f$, $\sigma_{f'}$ both land at a cusp with the same label (see Proposition~\ref{zero_ray_landing_cusp}), the definition of $\Psi$ in the finite part of the plane and near $\infty$ can be connected along the $0$-ray such that $\Psi$ is a conjugacy along the $0$-ray. The pullback argument of \cite[Theorem~5.1]{LMM1} now applies to show that $\Psi$ is the restriction of a M\"obius transformation $M$. Since $f$, $f\in\Sigma_d^*$, the map $M$ is multiplication by a $d+1^{\textrm{st}}$ root of unity. Since $\Psi'(\infty)=1$, it follows that $\Psi(z)\equiv z$. 
\end{proof}

\noindent We now wish to show that the mapping of Proposition \ref{bijection_statement} is in fact a homeomorphism, for which we first need the following lemma:

\begin{lem}\label{quadrilaterals_lemma} Let $U$, $V$ be Jordan domains. Let $n\geq4$, and suppose $u_1,\cdots, u_n \in \partial U$  and $v_1,\cdots, v_n \in \partial V$ are oriented positively with respect to $U$, $V$ (respectively). Suppose furthermore that the quadrilaterals \[ U(u_j, u_{j+1}, u_{k}, u_{k+1}), \hspace{4mm} V(v_j, v_{j+1}, v_{k}, v_{k+1}) \] have the same modulus for each $j,k$ with $1\leq j < j+2 \leq k \leq n-1$. Then there is a conformal map \[ f: U \rightarrow V \textrm{ such that } f(u_i) = v_i\textrm{, } 1\leq i \leq n.\]
\end{lem}

\begin{proof} Let $\Phi_U: U\rightarrow\mathbb{D}$, $\Phi_V: V \rightarrow\mathbb{D}$ be conformal maps such that \begin{align}\label{does_fix} \Phi_U(u_j)=\Phi_V(v_j)  \textrm{ for } 1\leq j \leq 3.\end{align} Suppose by way of contradiction that \begin{align}\label{doesnt_fix} \Phi_U(u_4)\not=\Phi_V(v_4). \end{align} Since $U(u_1, u_2, u_3, u_4)$, $V(v_1, v_2, v_3, v_4)$ have the same modulus, it follows that there is a conformal map $g: U\rightarrow V$ with $g(u_j)=v_j$ for $1\leq j \leq 4$. But then \[  \Phi_V \circ g \circ \Phi_U^{-1}: \mathbb{D} \rightarrow \mathbb{D} \] is a Mobius transformation which fixes $\Phi_U(u_1)$, $\Phi_U(u_2)$, $\Phi_U(u_3) \in \partial\mathbb{D}$ by (\ref{does_fix}), but is not the identity by (\ref{doesnt_fix}), and this is a contradiction. This shows that \[ \Phi_U(u_4)=\Phi_V(v_4), \] and the same argument applied recursively shows that \[ \Phi_U(u_j)=\Phi_V(v_j), \textrm{ for } 4\leq j \leq n. \] The lemma follows by taking $f=\Phi_V^{-1}\circ\Phi_U$.
\end{proof}

\begin{thm}\label{homeo_thm} The mapping \begin{align} \Sigma_d^* \rightarrow \overline{\beta(\pmb{\Gamma}_{d+1})} \nonumber \\ f\mapsto\Gamma_f \phantom{aasds} \nonumber \end{align} defined in Proposition~\ref{existence_of_group} is a homeomorphism.
\end{thm}

\begin{proof} Let $(f_n)_{n=1}^{\infty}\in\Sigma_d^*$, and suppose $f_n \rightarrow f_\infty \in\Sigma_d^*$. We abbreviate $\Gamma_n:=\Gamma_{f_n}$, $\Gamma_\infty:=\Gamma_{f_\infty}$. We want to show that \[ \Gamma_n \xrightarrow{n\rightarrow\infty} \Gamma_\infty \textrm{ in } \overline{\beta(\pmb{\Gamma}_{d+1})}.\] As $\overline{\beta(\pmb{\Gamma}_{d+1})}$ is compact, we may assume, after passing to a subsequence, that $\Gamma_n$ converges to some $\Gamma_\infty' \in \overline{\beta(\pmb{\Gamma}_{d+1})}$. 

We denote the critical values of $f_n$ by $\zeta_1^n,\cdots, \zeta_{d+1}^n$ (with the labeling chosen in Remark~\ref{labeling}). For $j,k$ with $1\leq j < j+2 \leq k \leq d$, we consider the quadrilateral \[ Q_n:=T(\sigma_{f_n})( \zeta_j^n, \zeta_{j+1}^n, \zeta_k^n, \zeta_{k+1}^n ), \] where we allow for the possibility that \begin{align}\label{a_sides_touch}\arc{\zeta_j^n\zeta_{j+1}^n}\cap \arc{\zeta_k^n \zeta_{k+1}^n} \not=\emptyset.  \end{align} Note that the arcs in (\ref{a_sides_touch}) may intersect in at most one point (see \cite[Proposition 4.8]{LMM1}), in which case we define \begin{align} M(Q_n):=\infty. \nonumber \end{align} Similarly, if  \begin{align}\label{b_sides_touch}\arc{\zeta_{j+1}^n\zeta_k^n} \cap \arc{\zeta_{k+1}^n\zeta_j^n} \not=\emptyset\textrm{, then } M(Q_n):=0.  \end{align} We note that only one of (\ref{a_sides_touch}) or (\ref{b_sides_touch}) may occur (see \cite[Proposition 4.8]{LMM1}). 

We also consider the quadrilaterals \[R_n:= T(\Gamma_n)(h_n^{-1}(\zeta_j^n), h_n^{-1}(\zeta_{j+1}^n), h_n^{-1}(\zeta_k^n), h_n^{-1}(\zeta_{k+1}^n) ), \] where \[h_n: R_n \rightarrow Q_n \] is a label-preserving conformal isomorphism by Proposition \ref{existence_of_group}, so that \begin{align}\label{equality_of_modulus} M(Q_n) = M(R_n) \textrm{ for all } n.\end{align} Since $f_n \rightarrow f_\infty$ in $\Sigma_d^*$, it follows from Theorem \ref{convergence_of_quadrilaterals} that \begin{align}\label{equality_of_modulus2} M(Q_n)\rightarrow M(Q_\infty) \textrm{ as } n\rightarrow\infty. \end{align} 

Now consider the quadrilateral \[ T(\Gamma_{\infty}')(\eta_j^\infty, \eta_{j+1}^\infty, \eta_k^\infty, \eta_{k+1}^\infty  ), \] where $\eta_j^{\infty}:=\lim_n h_n^{-1}(\zeta_j^n)$ is a cusp of $T(\Gamma_{\infty}')$. Since $\Gamma_n \rightarrow \Gamma_\infty'$ in $\overline{\beta(\pmb{\Gamma}_{d+1})}$, it follows that \[M(R_n) \rightarrow M( T(\Gamma_{\infty}')( \eta_j^\infty, \eta_{j+1}^\infty, \eta_k^\infty, \eta_{k+1}^\infty  ) ) \textrm{ as } n\rightarrow\infty. \] Thus by (\ref{equality_of_modulus}) and (\ref{equality_of_modulus2}), \[ M(Q_\infty) = M( T(\Gamma_{\infty}')( \eta_j^\infty, \eta_{j+1}^\infty, \eta_k^\infty, \eta_{k+1}^\infty  ) ).\]

\noindent As $j$, $k$ are arbitrary, Lemma \ref{quadrilaterals_lemma} applied to bounded components of $\mathbb{C}\setminus f_\infty(\mathbb{T})$ and $\Int{T(\Gamma_\infty')}$ yields a label-preserving conformal isomorphism \[ T(\Gamma_\infty') \rightarrow T(\sigma_{f_\infty}). \] Thus by the uniqueness of Proposition \ref{existence_of_group}, we have \[ \Gamma_\infty' = \Gamma_\infty, \] as needed. We conclude that  \begin{align} \Sigma_d^* \rightarrow \overline{\beta(\pmb{\Gamma}_{d+1})} \nonumber \\ \nonumber f \mapsto \Gamma_f \phantom{aasds} \end{align} is continuous, and the proof of continuity of the inverse is similar. 
\end{proof}


\section{Conformal Matings of Reflection groups and Polynomials}
\label{Conformal_mating}

The purpose of Section \ref{Conformal_mating} is to prove the conformal mating statement of Theorem \ref{theorem_A}. In Section \ref{local_connectivity} we will show that $\partial \mathcal{B}_\infty(\sigma_f)$ is locally connected for $f\in\Sigma_d^*$, whence in Section \ref{The Limit Set is the boundary of the Tiling set} we will show that $\mathcal{B}_\infty(\sigma_f)$ and $\mathcal{T}_\infty(\sigma_f)$ share a common boundary. Sections \ref{External Rays in the Schwarz Reflection Plane} and \ref{External Rays for the Reflection Group} study laminations of $\mathbb{T}$ induced by $\sigma_f$ and necklace groups $\Gamma$, whence it is shown in Section \ref{lamination_relation_sec} that for $f\in\Sigma_d^*$, the laminations induced by $\sigma_f$ and $\Gamma_f$ are compatible. Finally, in Section \ref{Proof of Conformal Mating}, we deduce that $\sigma_f$ is a conformal mating of $\Gamma_f$ and $w\mapsto\overline{w}^d$ (see Definition~\ref{mating}).
 
\subsection{Local Connectivity}
\label{local_connectivity}

\begin{lem}\label{domain_hyperbolicity} Let $f\in\Sigma_d^*$. Then \[\left|\overline{\partial}\sigma_f(z)\right| > 1\textrm{, } \forall z \in f(\mathbb{D}^*).\]
\end{lem}
\begin{proof}
Let $\sigma:=\sigma_f$. From Definition \ref{Schwarz_reflection}, we see that \begin{equation}\label{domain_hyperbolicity1}\sigma(f(w))=f\left(1/\overline{w}\right)\textrm{, for } \vert w\vert>1. \end{equation} Taking the $\overline\partial$-derivative of (\ref{domain_hyperbolicity1}) yields 
\begin{equation}
\overline{\partial}\sigma\left(f(w)\right)\cdot\overline{f'(w)}=-\frac{1}{\overline{w}^2}\cdot f'\left(\frac{1}{\overline{w}}\right),\ \textrm{for}\ \vert w\vert>1. 
\label{deltoid_schwarz_derivative}
\end{equation} 
By \cite[Lemma~2.6]{2014arXiv1411.3415L}, we also have that 
\begin{equation}
f'\left(\frac{1}{\overline{w}}\right)=\overline{w}^{d+1}\overline{f'(w)},\ \forall w\in \C.
\label{self_dual}
\end{equation}
Combining (\ref{deltoid_schwarz_derivative}) and (\ref{self_dual}), we conclude that $$\vert\overline{\partial}\sigma(f(w))\vert=\vert w\vert^{d-1}>1 \textrm{ for } \vert w\vert>1.$$
\end{proof}

\begin{prop}\label{local_connectivity_prop} Let $f\in\Sigma_d^*$. Then $\partial \mathcal{B}_\infty(\sigma_f)$ is locally connected.
\end{prop}

\begin{rem} Our proof follows the strategy taken in \cite[Chapter 10]{orsay}. \end{rem}

\begin{proof} 
Let $\sigma:=\sigma_f$, $\phi_\sigma:=\phi_{\sigma_f}$ be as in Remark~\ref{bottcher_coordinate}, and $X:=\mathcal{B}_\infty(\sigma)\cap\mathbb{C}$. Note that $\sigma: X \rightarrow X$ is a $d\hspace{-1mm}:\hspace{-1mm}1$ covering map. Define an equipotential curve \[ E(r) := \phi_{\sigma}\left(\{ z\in\mathbb{C} :  |z| = r\} \right). \] We define, for $n\geq1$, parametrizations $\gamma_n: \mathbb{T} \rightarrow E(2^{1/d^n})$ by: \[ \gamma_n(e^{2\pi i\theta}):=\phi_{\sigma}(2^{1/d^{n}}e^{2\pi i\theta}). \] By (\ref{bottcher_conjugacy}), we have: $$\sigma\circ \gamma_{n+1}(e^{2\pi i\theta})=\gamma_n(e^{-2\pi i d\theta}).$$ We will show that the sequence $(\gamma_n)_{n=1}^{\infty}$ forms a Cauchy sequence in the complete metric space $C(\mathbb{T}, \mathbb{C})$. We will denote the length of a curve $\gamma$ by $l(\gamma)$, and the lift of $\gamma$ under $\sigma$ by $\widetilde{\gamma}$.

To this end, define $h: \mathbb{R}_{\geq0}\rightarrow \mathbb{R}_{\geq0}$ by \begin{equation} h(s):= \sup_{ \substack{\gamma\in C(\mathbb{T}, X) \\ l(\gamma)\leq s} } \left\{ l(\widetilde{\gamma}) : \sigma(\widetilde{\gamma}) = \gamma \right\}. \end{equation} We claim that \begin{equation}\label{equation_relations} h(s) < s \textrm{ and } h(ks) \leq kh(s)\textrm{, }\forall s > 0 \textrm{ and } k\in\mathbb{N}. \end{equation} Indeed, the first inequality of (\ref{equation_relations}) follows from Lemma \ref{domain_hyperbolicity}. The second inequality in (\ref{equation_relations}) follows from the triangle inequality. It follows from (\ref{equation_relations}) that $s-h(s)\rightarrow\infty$ as $s\rightarrow\infty$. Let $\ell:=\dist(\gamma_0,\gamma_1)$, and choose $L>\ell$ sufficiently large such that $L-h(L)>\ell$. One has: \[\dist(\gamma_2, \gamma_{0}) \leq \dist(\gamma_2, \gamma_{1}) + \dist(\gamma_{1},\gamma_0) \leq h( \ell ) + \ell \leq h(L) + \ell < L. \] Similarly, an inductive procedure yields \[ \dist(\gamma_n, \gamma_{0}) < L\textrm{, } \forall n\geq0. \] Observe that \[ \dist(\gamma_n, \gamma_{n+p}) \leq h^{\circ n}( \dist(\gamma_0, \gamma_p  ) ) < h^{\circ n}(L).  \] The sequence $(h^{\circ n}(L))_{n=1}^{\infty} \in \mathbb{R}_{\geq0}$ is decreasing and converges to a fixpoint of $h$, and this fixpoint must be $0$ by (\ref{equation_relations}). Thus $(\gamma_n)_{n=1}^{\infty}$ is a Cauchy sequence, and the limit is a continuous extension of \[ \phi_{\sigma}: \mathbb{D}^*\rightarrow \mathcal{B}_\infty(\sigma) \textrm{ to } \phi_{\sigma}: \mathbb{T} \rightarrow \partial \mathcal{B}_\infty(\sigma). \] Local connectivity of $\partial \mathcal{B}_\infty(\sigma)$ follows from a theorem of Carath\'eodory. 
\end{proof}

\subsection{The Limit Set is The Boundary of The Basin of Infinity}
\label{The Limit Set is the boundary of the Tiling set}

The goal of this Subsection is to prove the following:

\begin{prop}\label{equality_prop} Let $f\in\Sigma_d^*$. Then $\partial \mathcal{B}_\infty(\sigma_f)= \partial \mathcal{T}_{\infty}(\sigma_f)$. 

\end{prop}

\noindent The proof of Proposition \ref{equality_prop} will be carried out by way of several lemmas below. First we record the following definition:

\begin{definition}\label{external_rays_sigma} Let $f\in\Sigma_d^*$. An \emph{external ray} for $\sigma_f$ is a curve \[ t \mapsto \phi_{\sigma_f}(te^{i\theta})\textrm{, } t\in(1,\infty) \] for some $\theta\in[0,2\pi)$, where $\phi_{\sigma_f}$ is the B\"ottcher coordinate of Remark \ref{bottcher_coordinate}. For $\theta\in[0,2\pi)$, we refer to \[ \{ \phi_{\sigma_f}(te^{i\theta}) : t\in(1,\infty) \} \] as the $\theta$-ray of $\sigma_f$.
\end{definition}

\begin{rem} By Proposition \ref{local_connectivity_prop}, each external ray of $\sigma_f$ lands, in other words $\lim_{t\rightarrow1^+}\phi_{\sigma_f}(te^{i\theta})$ exists for each $\theta\in[0,2\pi)$.
\end{rem}

\begin{notation} Let $\Sigma_{d,k}^*$ denote the collection of those $f\in\Sigma_d^*$ such that $f(\mathbb{T})$ has exactly $k$ double points. 
\end{notation}

\noindent For the remainder of this subsection we fix $f\in\Sigma_{d,k}^*$ and denote $\sigma:=\sigma_f$. Let us first record the straightforward inclusion:

\begin{lem}\label{easy_inclusion_lem}
$\partial \mathcal{B}_\infty(\sigma)\subset\partial \mathcal{T}_\infty(\sigma)$.
\end{lem}
\begin{proof}
We note that $\mathcal{T}_\infty(\sigma)$ is open, whence the relation \begin{align}\label{decomp} \widehat{\C}=\mathcal{T}_\infty(\sigma)\sqcup\partial \mathcal{T}_\infty(\sigma)\sqcup\mathcal{B}_\infty(\sigma) \end{align} follows from the classical classification of periodic Fatou components and the observation that $\sigma$ has only one singular value (at $\infty$). The Lemma follows from (\ref{decomp}).
\end{proof}

\noindent The proof of the opposite inclusion is a bit more involved, and we split the main arguments into a couple of lemmas.

\begin{lem}\label{sing_on_limit}
The landing points of the fixed external rays of $\sigma$ are singular points of $f(\mathbb{T})$. 
\end{lem}
\begin{proof}
Note that the landing points of the fixed rays of $\sigma$ are necessarily fixed points of $\sigma$ on $\partial \mathcal{B}_\infty(\sigma)$. Since $\partial \mathcal{B}_\infty(\sigma)\subset\partial \mathcal{T}_\infty(\sigma)$ by Lemma \ref{easy_inclusion_lem}, the result will follow if we can prove that the only fixed points of $\sigma|_{\partial \mathcal{T}_\infty(\sigma)}$ are the singular points of $f(\mathbb{T})$. This will be shown via the Lefschetz fixed-point formula.

As $f$ has $k$ double points, there are $k+1$ forward-invariant components $\mathcal{U}_1,\cdots, \mathcal{U}_{k+1}$ of $\mathcal{T}_\infty(\sigma)$, each containing a single component of $T^o(\sigma)$. Let $T_i:=\mathcal{U}_i\cap T^o(\sigma)$, so that $T_i$ is naturally a $(3+j_i)$-gon ($j_i\geq0$) whose vertices are the $3+j_i$ singularities of $f(\mathbb{T})$ lying on $\partial\mathcal{U}_i$. Each $\mathcal{U}_i$ is a simply connected domain as it can be written as an increasing union of pullbacks (under $\sigma$) of $T_i$ (see also \cite[Proposition~5.6]{LLMM1}). Moreover, we can map each $T_i$ conformally to a $(3+j_i)$-gon in $\mathbb{D}$ whose edges are geodesics of $\mathbb{D}$. By iterated Schwarz reflection, one now obtains a Riemann map from $\mathbb{D}$ onto $\mathcal{U}_i$. Using Lemma~\ref{domain_hyperbolicity}, one can mimic the proof of Proposition~\ref{local_connectivity_prop} to show that each $\partial \mathcal{U}_i$ is locally connected. 

We now consider a quasiconformal homeomorphism $\chi_i:T^o(\pmb{\Gamma}_{3+j_i})\to T_i$ that sends the boundary cusps to the boundary cusps. Lifting $\chi_i$ by $\rho_{\pmb{\Gamma}_{3+j_i}}$ and $\sigma$, we obtain a quasiconformal homeomorphism $\chi_i:\D\to\mathcal{U}_i$ that conjugates $\rho_{\pmb{\Gamma}_{3+j_i}}$ to $\sigma$. Since $\partial \mathcal{U}_i$ is locally connected, $\chi_i$ extends continuously to the boundary, and yields a topological semi-conjugacy between $\rho_{\pmb{\Gamma}_{3+j_i}}\vert_{\mathbb{T}}$ and $\sigma\vert_{\partial\mathcal{U}_i}$. It now follows from Remark~\ref{defn_of_conjug} that $$\widehat{\chi}_i:=\chi_i\circ\mathcal{E}_{2+j_i}^{-1}:\mathbb{T}\to\partial\mathcal{U}_i$$ is a topological semi-conjugacy between $\overline{z}^{2+j_i}\vert_{\mathbb{T}}$ and $\sigma\vert_{\partial\mathcal{U}_i}$. Let $\widehat{\chi}_i:\overline{\D}\to\overline{\mathcal{U}_i}$ be an arbitrary continuous extension of $\widehat{\chi}_i\vert_{\mathbb{T}}$ such that $\widehat{\chi}_i$ maps $\D$ homeomorphically onto $\mathcal{U}_i$.

We now (topologically) glue attracting basins into the domains $\mathcal{U}_i$:

\begin{equation*}
\check{\sigma}(w):=\left\{\begin{array}{ll}
                    \sigma(w) & \mbox{on}\ \widehat{\C}\setminus\bigcup_{i=1}^{k+1}\mathcal{U}_i, \\
                     \widehat{\chi}_i\left(\overline{\widehat{\chi}_i^{-1}(w)}^{2+j_i}\right) & \mbox{on}\ \mathcal{U}_i,\mbox{ for}\ i=1,\cdots, k+1.
                                          \end{array}\right. 
\end{equation*}

\noindent The map $\check{\sigma}$ is a degree $d$ orientation-reversing branched cover of $\widehat{\mathbb{C}}$. We argue that each fixed point of $\check{\sigma}$ is either attracting or repelling. By Lemma \ref{domain_hyperbolicity} and construction of $\check{\sigma}$, this is the case for each fixed point in $\widehat{\mathbb{C}}\setminus \cup_i\partial\mathcal{U}_i$. We note that $\check{\sigma}$ has $k+2$ attracting fixed points (one in each $\mathcal{U}_i$ and one at $\infty$). Also by Lemma \ref{domain_hyperbolicity} and construction of $\check{\sigma}$, any fixed point of $\check{\sigma}$ on $\partial\mathcal{U}_i\setminus\{ \textrm{singular values of } f(\mathbb{T}) \}$ must be repelling. The singular values of $f(\mathbb{T})$ are fixed under $\check{\sigma}$ by construction. Such fixed points exhibit \emph{parabolic} behavior under $\sigma$, and near such a fixed point, the complement of $\cup_{i=1}^{k+1} \mathcal{U}_i$ lies in the corresponding repelling petals (compare \cite[Propositions~6.10, 6.11]{LLMM1}). Moreover, by construction of $\check{\sigma}$, the singular values of $f(\mathbb{T})$ are also repelling for $\check{\sigma}|_{\cup_{i=1}^{k+1} \mathcal{U}_i}$. Thus singular values of $f(\mathbb{T})$ are repelling fixed points of $\check{\sigma}$, and so each fixed point of $\check{\sigma}$ is either attracting or repelling.
 
By the Lefschetz fixed-point formula (see \cite[Lemma~6.1]{2014arXiv1411.3415L}), we may thus conclude that $\check{\sigma}$ has $(d+2k+3)$ fixed points in $\mathbb{S}^2$. We have already counted that $\check{\sigma}$ has $k+2$ attracting fixed points, and $d+k+1$ repelling fixed points at singular values of $f(\mathbb{T})$, so that we can conclude there are no other fixed points of $\check{\sigma}$. Since $\sigma$ and $\check{\sigma}$ have the same fixed points on $\partial \mathcal{T}_\infty(\sigma)$, it follows that the singular points of $\partial T(\sigma)$ are the only fixed points of $\sigma$ on $\partial \mathcal{T}_\infty(\sigma)$.  
\end{proof}

\begin{lem}\label{basin_exterior_tiling}
$\Int{\overline{\mathcal{T}_\infty(\sigma)}}=\mathcal{T}_\infty(\sigma)$. In particular, each component of $\mathcal{T}_\infty(\sigma)$ is a Jordan domain.
\end{lem}
\begin{proof}

Let $\mathcal{U}$ denote a component of $\mathcal{T}_\infty(\sigma)$. We first show that $\partial\mathcal{U}\subset \partial \mathcal{B}_\infty(\sigma)$. First assume $\mathcal{U}$ is the forward-invariant component of $\mathcal{T}_\infty(\sigma)$ which contains the landing point $p$ of the $0$-ray of $\sigma$. As in the proof of Lemma \ref{sing_on_limit}, we note that $\sigma|_{\partial \mathcal{U}}$ is topologically semi-conjugate to $\rho_{\pmb{\Gamma}_{2+j_i}}|_{\mathbb{T}}$. Thus the iterated pre-images of $p$ under $\sigma$ are dense in $\partial\mathcal{U}$. Since $p\in\partial \mathcal{B}_\infty(\sigma)$ by Lemma \ref{sing_on_limit}, and $\partial \mathcal{B}_\infty(\sigma)$ is completely invariant, it follows that $\partial \mathcal{U}\subset\partial \mathcal{B}_\infty(\sigma)$. A similar argument applies to show that the boundary of any forward-invariant component of $\mathcal{T}_\infty(\sigma)$ is contained in $\partial \mathcal{B}_\infty(\sigma)$. Lastly, any other component of $\mathcal{T}_\infty(\sigma)$ maps (under some iterate of $\sigma$) onto one of the invariant components of $\mathcal{T}_\infty(\sigma)$, so that $\partial\mathcal{U}\subset \partial \mathcal{B}_\infty(\sigma)$ for any component $\mathcal{U}$ of $\mathcal{T}_\infty(\sigma)$. 

Note that since $\mathcal{T}_\infty(\sigma)$ is open, we have $\mathcal{T}_\infty(\sigma)\subset \textrm{int }{\overline{\mathcal{T}_\infty(\sigma)}}$. Let us now pick a component $W$ of $ \textrm{int }{\overline{\mathcal{T}_\infty(\sigma)}}$. Since $\partial \mathcal{T}_\infty(\sigma)$ is nowhere dense in $\mathbb{C}$, it follows that $W$ must intersect some component $\mathcal{U}$ of $\mathcal{T}_\infty(\sigma)$. As $W$ is a maximal open connected subset of $\overline{\mathcal{T}_\infty(\sigma)}$, it follows $\mathcal{U}\subset W$. However, if $\mathcal{U}\subsetneq W$, then $\partial\mathcal{U}$ must contain some point not belonging to $\partial\overline{\mathcal{T}_\infty(\sigma)}=\partial \mathcal{B}_\infty(\sigma)$, and this contradicts what was shown in the previous paragraph. Thus $\mathcal{U} = W$, and so $\textrm{int }{\overline{\mathcal{T}_\infty(\sigma)}} \subset \mathcal{T}_\infty(\sigma)$. The conclusion of the lemma follows.
\end{proof}

\begin{proof}[Proof of Proposition~\ref{equality_prop}] Let \[x\in\partial \mathcal{T}_\infty(\sigma)\subset\overline{\mathcal{T}_\infty(\sigma)}= \left( \textrm{int }{\overline{\mathcal{T}_\infty(\sigma)}}\right)\sqcup\partial\overline{\mathcal{T}_\infty(\sigma)}.\] By Lemma~\ref{basin_exterior_tiling} and the openness of $\mathcal{T}_\infty(\sigma)$, it follows that \[x\in\partial\overline{\mathcal{T}_\infty(\sigma)}=\partial\mathcal{B}_\infty(\sigma).\] Hence, $\partial \mathcal{T}_\infty(\sigma)\subset\partial\mathcal{B}_\infty(\sigma)$, and together with Lemma~\ref{easy_inclusion_lem}, this proves Proposition \ref{equality_prop}.
\end{proof}

\begin{cor}\label{decomp_cor} Let $f\in\Sigma_d^*$. Then \[ \widehat{\mathbb{C}} = \mathcal{B}_\infty(\sigma_f) \sqcup \Lambda(\sigma_f) \sqcup \mathcal{T}_\infty(\sigma_f), \] where $\Lambda(\sigma_f)= \partial\mathcal{B}_\infty(\sigma_f)=\partial\mathcal{T}_{\infty}(\sigma_f)$ is the limit set of $\sigma_f$.
\end{cor}

\begin{proof} 
This follows from (\ref{decomp}), Proposition \ref{equality_prop}, and the definition of $\Lambda(\sigma_f)$ (see Subsection~\ref{sigma_prelim_subsec}). 
\end{proof}


\subsection{Lamination for The Limit Set $\Lambda(\sigma_f)$}
\label{External Rays in the Schwarz Reflection Plane}

In this Subsection, we study further the external rays of $\sigma_f$ introduced already in Definition \ref{external_rays_sigma}. Recall the B\"ottcher coordinate $\phi_{\sigma_f}: \mathbb{D}^* \rightarrow \mathcal{B}_\infty(\sigma_f)$ of Remark \ref{bottcher_coordinate}.

\begin{rem} 
By Proposition \ref{local_connectivity_prop}, $\phi_{\sigma_f}$ extends continuously to  a surjection \[ \phi_{\sigma_f}: \mathbb{T} \rightarrow \Lambda(\sigma_f) \textrm{ such that } \sigma_f\circ\phi_{\sigma_f}(u) = \phi_{\sigma_f}(\overline{u}^d).  \] We denote the set of all fibers of points of $\Lambda(\sigma_f)$ under the semi-conjugacy $\phi_{\sigma_f}$ by $\lambda(\sigma_f)$. Clearly, $\lambda(\sigma_f)$ defines an equivalence relation on $\mathbb{T}$. The description of $\lambda(\sigma_f)$ will be crucial to the proof of the conformal mating statement in Theorem \ref{theorem_A}. 
\end{rem}

\begin{rem} 
We will usually identify $\mathbb{T}$ with $\mathbb{R}/\mathbb{Z}$ and the map $z\mapsto\overline{z}^d$ on $\mathbb{T}$ with the map \[ m_{-d}: \mathbb{R}/\mathbb{Z} \rightarrow \mathbb{R}/\mathbb{Z},\] defined by $m_{-d}(x):=-dx$.
\end{rem}

\begin{lem}\label{fixed_point_ray_lem} 
Let $f\in\Sigma_d^*$. Then: 
\begin{enumerate}
\item Each cusp of $f(\mathbb{T})$ is the landing point of a unique external ray of $\mathcal{B}_\infty(\sigma_f)$, and the angle of this ray is fixed under $m_{-d}$.

\item Each double point of $f(\mathbb{T})$ is the landing point of exactly two external rays of $\mathcal{B}_\infty(\sigma_f)$, and the angles of the corresponding two rays form a $2$-cycle under $m_{-d}$. 
\end{enumerate}
\end{lem}
\begin{proof}
We abbreviate $\sigma:=\sigma_f$ and $T:=T(\sigma_f)$. Let $\zeta$ be a singular point of $f(\mathbb{T})$. Since $\zeta\in\partial \mathcal{T}_{\infty}(\sigma)$, Propositions \ref{local_connectivity_prop} and \ref{equality_prop} imply that $\zeta$ is the  landing point of at least one external ray of $\sigma$.

\bigskip

\noindent \emph{Proof of (1).}  Assume $\zeta$ is a cusp of $f(\mathbb{T})$. We will show that $\zeta$ is not a cut-point of $\overline{\mathcal{T}_\infty(\sigma)}$. Let $S$ be the connected component of $\overline{\mathcal{T}_\infty(\sigma)}\setminus\{\zeta\}$ which contains $T\setminus\{\zeta\}$. It follows from the covering properties of $\sigma$ that \[ \bigcup_{k=0}^n\sigma^{-k}(T\setminus\{\zeta\}) \] is a connected subset of $\overline{\mathcal{T}_\infty(\sigma)}\setminus\{\zeta\}$, so that \[ \bigcup_{k=0}^n\sigma^{-k}(T\setminus\{\zeta\}) \subset S\textrm{, } \forall n\geq0. \] It follows that \[ \overline{\mathcal{T}_\infty(\sigma)}\subset\overline{\bigcup_{r=0}^\infty\sigma^{-r}(T\setminus\{\zeta\})}\subset \overline{S} \subset S\cup\{\zeta\}. \] By our choice of $S$, we have that \[ \overline{\mathcal{T}_\infty(\sigma)} = S\cup\{\zeta\}, \] so that $\zeta$ is not a cut point of $\overline{\mathcal{T}_\infty(\sigma)}$. Hence, $\zeta$ is the landing point of exactly one external ray $\gamma$ of $\mathcal{B}_\infty(\sigma)$. Since $\sigma(\zeta)=\zeta$, it follows that $\sigma(\gamma)$ also lands at $\zeta$, whence $\sigma(\gamma)=\gamma$. In particular, the angle of $\gamma$ is fixed under $m_{-d}$. \qed

\bigskip

\noindent \emph{Proof of (2).} Suppose now that $\zeta$ is a double point of $f(\mathbb{T})$. Let $S_1$, $S_2$ be the two components of $\overline{\mathcal{T}_\infty(\sigma)}\setminus\{\zeta\}$ such that $T\setminus\{\zeta\}\subset S_1\cup S_2$. A similar argument as in the proof of (1) yields that \[ \overline{\mathcal{T}_\infty(\sigma)}=S_1\cup S_2\cup\{\zeta\}. \] In particular, $S_1$, $S_2$ are the only components of $\overline{\mathcal{T}_\infty(\sigma)}\setminus\{\zeta\}$. Thus there are only two accesses to $\zeta$ from $\mathcal{B}_\infty(\sigma)$, and hence there are exactly two external rays landing at $\zeta$. By (1), the $(d+1)$ fixed rays land at the $(d+1)$ distinct cusps on $f(\mathbb{T})$. Therefore, the angles of the two rays landing at $\zeta$ must be of period two, forming a $2$-cycle under $m_{-d}$.
\end{proof}

\begin{notation} We denote by $\mathcal{A}^{\textrm{cusp}}$ the angles of external rays of $\sigma_f$ landing at the cusps of $f(\mathbb{T})$, and by $\mathcal{A}^{\textrm{double}}$ the angles of external rays of $\sigma_f$ landing at the double points of $f(\mathbb{T})$.  
\end{notation}

\begin{rem} By Proposition~\ref{fixed_point_ray_lem}, $\mathcal{A}^{\textrm{cusp}}$ is the set of angles fixed under $m_{-d}$. On the other hand, for $f\in\Sigma_{d,k}^*$, $\mathcal{A}^{\textrm{double}}$ consists of $2k$ angles of period two which we enumerate as $\{\alpha_1,\alpha_1',\cdots,\alpha_{k},\alpha_{k}'\}$ where the rays at angles $\alpha_i,\alpha_i'$ land at a common point.
\end{rem}

\begin{figure}[ht!]
\includegraphics[scale=0.5]{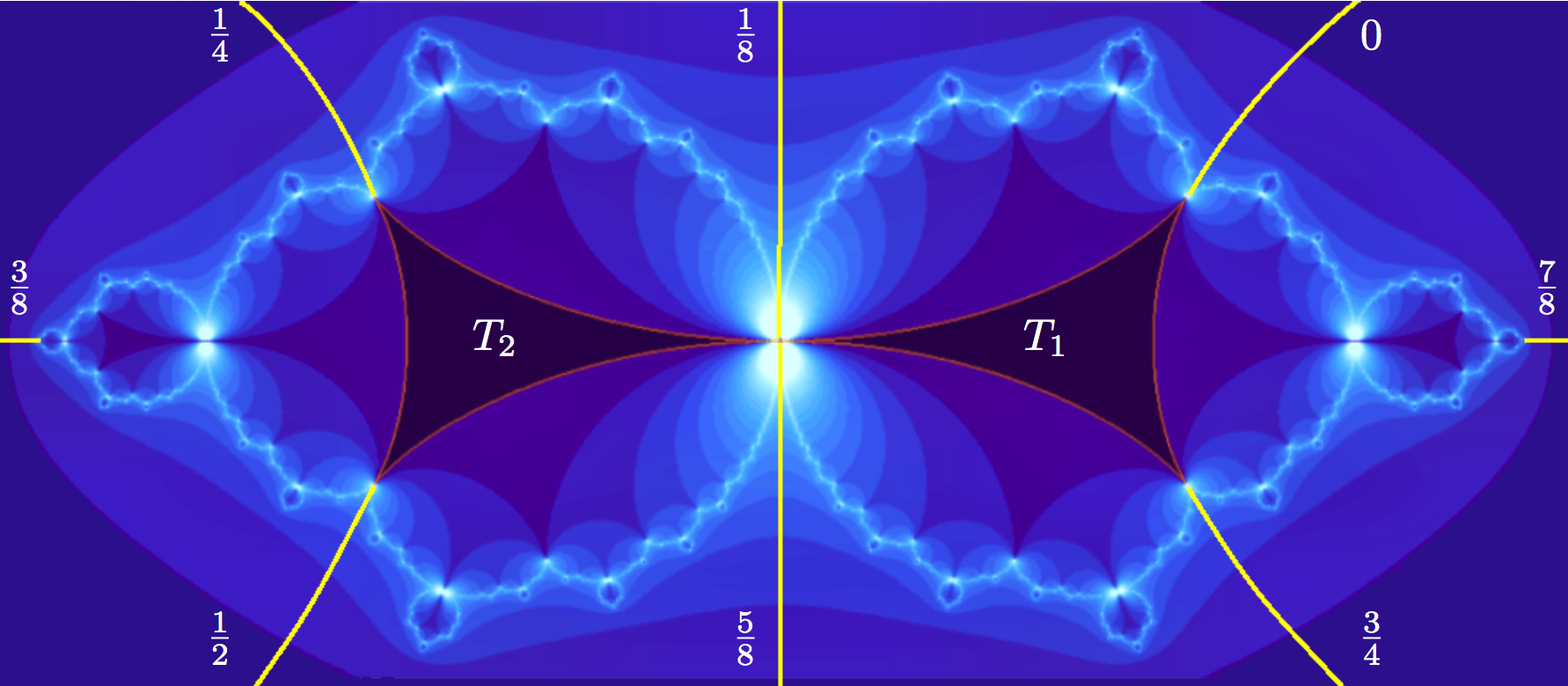}
\caption{Shown is $\mathcal{T}_{\infty}(\sigma_f)$ for $f(z):=z-2/(3z)-1/(3z^3)$. Also pictured are several external rays for $\sigma_f$: here $\mathcal{A}^{\textrm{cusp}} = \{ 0, 1/4, 1/2, 3/4\}$, and $\mathcal{A}^{\textrm{double}}=\{1/8, 5/8\}$. This figure was made by Seung-Yeop Lee.}
\label{extremal_qd_schwarz}
\end{figure}

\begin{rem}\label{Markov} 
Let $f\in\Sigma_{d,k}^*$. The union of $T(\sigma_f)$ with the external rays of $\sigma_f$ at angles  $\mathcal{A}^{\textrm{cusp}}$ and $\mathcal{A}^{\textrm{double}}$ cut the limit set $\Lambda(\sigma)$ into $d+2k+1$ pieces that form a Markov partition for the dynamics $\sigma:\Lambda(\sigma)\to\Lambda(\sigma)$ (see Figure \ref{extremal_qd_schwarz}). Correspondingly, the angles in $\mathcal{A}^{\textrm{cusp}}\cup\mathcal{A}^{\mathrm{double}}$ determine a Markov partition for $m_{-d}:\R/\Z\to\R/\Z$, and the elements of this Markov partition have diameter at most $1/d$. 
\end{rem}

\begin{prop}\label{schwarz_lamination_prop} 
Let $f\in\Sigma_{d,k}^*$, and $x\in\lambda(\sigma_f)$ be a non-trivial equivalence class. Then, the following hold true.
\begin{enumerate}
\item $\vert x\vert=2$, and $m_{-d}^{\circ n}(x)=\{\alpha_i, \alpha_i'\}$ for some $n\geq0$ and $1\leq i \leq k$. 

\item If $n_0$ is the smallest non-negative integer with $m_{-d}^{\circ n_0}(x)=\{\alpha_i, \alpha_i'\}$, then $x$ is contained in a connected component of $\mathbb{T}\setminus m_{-d}^{-(n_0-1)}(\mathcal{A}^{\textrm{double}}\cup\mathcal{A}^{\textrm{cusp}})$.
\end{enumerate}
\end{prop}
\begin{proof}[Proof of (1).] We abbreviate $\sigma:=\sigma_f$.  If $x\in\lambda(\sigma_f)$ is non-trivial, $x$ is a collection of $\geq2$ angles whose corresponding external rays for $\sigma$ land at a cut-point $w$ of $\Lambda(\sigma)$. Let $\theta$, $\theta'$ be distinct angles in $x$. Suppose $w$ is not a double point of $f(\mathbb{T})$. As $w$ is a cut-point of $\Lambda(\sigma)$, it follows from Lemma~\ref{fixed_point_ray_lem} that $w$ is not a cusp of $f(\mathbb{T})$. 

Suppose by way of contradiction that no iterate of $\sigma$ maps $w$ to a double point of $f(\mathbb{T})$. Note that no iterate of $\sigma$ can map $w$ to a cusp of $f(\mathbb{T})$, as $\sigma$ is a local homeomorphism on $\Lambda(\sigma)\setminus f(\mathbb{T})$ and cusps of $f(\mathbb{T})$ are not cut-points of $\Lambda(\sigma)$ by Lemma \ref{fixed_point_ray_lem}. Thus, $w$ has a well-defined itinerary (or symbol sequence) with respect to the Markov partition of $\Lambda(\sigma)$ in Remark \ref{Markov}. This implies that the angles $\theta$ and $\theta'$ have the same (well-defined) itinerary with respect to the corresponding Markov partition of $\R/\Z$. However, this contradicts expansivity of the map $m_{-d}:\R/\Z\to\R/\Z$, as the distance between $m_{-d}^{\circ j}(\theta)$ and $m_{-d}^{\circ j}(\theta')$ must exceed $1/(d+1)$ for some $j\in\N$. This contradiction proves that $\sigma^{\circ n}(w)$ is a double point of $f(\mathbb{T})$ for some $n\geq0$. Let use choose the smallest $n$ with this property, and call it $n_0$.

By Lemma \ref{fixed_point_ray_lem}, there are exactly two rays $\{\alpha_i, \alpha_i'\}$ landing at the double point $\sigma^{\circ n_0}(w)$. As $\sigma$ is a local homeomorphism on $\Lambda(\sigma)\setminus f(\mathbb{T})$, it follows that $\theta$, $\theta'$ are the only two rays landing at $w$, and that $\sigma^{\circ n_0}$ maps the pair of rays at angles $\{\theta,\theta'\}$ to the pair of rays at angles $\{\alpha_i,\alpha_i'\}$. In other words, $x=\{\theta, \theta' \}$ and $m_{-d}^{\circ n_0}(x)=\{\alpha_i, \alpha_i'\}$.  
\end{proof}

\begin{proof}[Proof of (2).] This follows from the landing patterns of the rays corresponding to the angles in $\mathcal{A}^{\textrm{double}}\cup\mathcal{A}^{\textrm{cusp}}$ and injectivity of $\sigma$ on the interior of each piece of the Markov partition of $\Lambda(\sigma)$ defined in Remark~\ref{Markov}.  
\end{proof}

We conclude this subsection with a proof of connectedness of $\Sigma_d^*$ (which was used to define the labeling of the cusps on $f(\mathbb{T})$ in Remark~\ref{labeling}), and a dynamical characterization of the cusp $\zeta_1^f$ as the landing point of the $0$-ray of $\sigma_f$, for $f\in\Sigma_d^*$ (which was used in the injectivity step of the proof of Proposition~\ref{bijection_statement}).

\begin{prop}\label{sigma_d^*_connected_thm}
$\Sigma_d^*$ is connected.
\end{prop}
\begin{proof}
The main ideas of the proof are already present in \cite{LMM1}, so we only give a sketch.

Let $f_0(z)=z-1/dz^d$. By \cite[Proposition 3.1]{LMM1}, we have $f_0\in\Sigma_d^*$ and $f_0(\mathbb{T})$ is a Jordan curve. We will denote the connected component of $\Sigma_d^*$ containing $f_0$ by $\widetilde{\Sigma_d^*}$. 

For a $(d+1)$-st root of unity $\omega$, the map $M_{\omega}$ is defined as $M_{\omega}(z)=\omega z$. The map $M_\omega$ induces a homeomorphism $$(M_\omega)_\ast:\Sigma_d^*\to\Sigma_d^*,\ f \mapsto M_{\omega}\circ f\circ M_{\omega}^{-1}.$$ Since $(M_\omega)_\ast(f_0)=f_0$, it follows that $\widetilde{\Sigma_d^*}$ is invariant under $(M_\omega)_\ast$.

Let $f\in\Sigma_d^*$, and suppose that $f(\mathbb{T})$ has $k$ double points, for some $0\leq k\leq d-2$. Since $f_0(\mathbb{T})$ is a Jordan curve, we can think of $f_0(\mathbb{T})$ as a $(d+1)$-gon with vertices at the cusp points. By repeated applications of \cite[Theorem~4.11]{LMM1} and quasiconformal deformation of Schwarz reflection maps, one can now ``pinch'' $k$ suitably chosen pairs of non-adjacent sides of $f_0(\mathbb{T})$ producing some $\widetilde{f}\in\Sigma_d^*$ such that there exists a homeomorphism $\mathfrak{h}:T(\sigma_{\widetilde{f}})\to T(\sigma_f)$ that is conformal on $\Int{T(\sigma_{\widetilde{f}})}$. Note that the proof of \cite[Theorem~4.11]{LMM1} consists of two steps; namely, quasiconformally deforming Schwarz reflection maps and extracting limits of suitable sequences in $\Sigma_d^*$. Thanks to the parametric version of the Measurable Riemann Mapping Theorem and continuity of normalized Riemann maps, one can now conclude that $\widetilde{f}\in\widetilde{\Sigma_d^*}$. Finally, due to the existence of a homeomorphism $\mathfrak{h}:T(\sigma_{\widetilde{f}})\to T(\sigma_f)$ that is conformal on $\Int{T(\sigma_{\widetilde{f}})}$, the arguments of \cite[Theorem~5.1]{LMM1} apply mutatis mutandis to the current setting, and provide us with affine map $A$ with $\sigma_f\equiv A\circ \sigma_{\widetilde{f}}\circ A^{-1}$; i.e., $A(\widetilde{f}(\D^*))=f(\D^*)$. Arguing as in the injectivity step of \cite[Proposition~2.14]{LMM1}, one now sees that $f=M_\omega\circ \widetilde{f}\circ M_\omega^{-1}=(M_\omega)_\ast(\widetilde{f})$, where $\omega$ is a $(d+1)$-st root of unity. Since $\widetilde{f}\in\widetilde{\Sigma_d^*}$ and $\widetilde{\Sigma_d^*}$ is invariant under $(M_\omega)_\ast$, it follows that $f\in\widetilde{\Sigma_d^*}$. Hence, $\Sigma_d^*=\widetilde{\Sigma_d^*}$; i.e., $\Sigma_d^*$ is connected.
\end{proof}

Recall from  that the cusps on $f(\mathbb{T})$ were labeled as $\zeta_1^f,\cdots, \zeta_{d+1}^f$ so that $f\mapsto \zeta_i^f$ is continuous.

\begin{prop}\label{zero_ray_landing_cusp} Let $f_0(z):=z-1/(dz^d)$, $f\in\Sigma_d^*$ and $\omega_0:=e^\frac{i\pi}{d+1}$. Then:

\begin{enumerate} \item The $0$-ray of $\sigma_{f_0}$ lands at the cusp $\zeta_1^{f_0}=(1+1/d)\omega_0$ of $f_0(\mathbb{T})$. 

\item The $0$-ray of $\sigma_{f}$ lands at the cusp point $\zeta_1^f$ of $f(\mathbb{T})$. \end{enumerate}
\end{prop}

\begin{proof} We abbreviate $\sigma:=\sigma_{f_0}$. We will also employ our notation $\phi_{\sigma}:\mathbb{D}^* \rightarrow \mathcal{B}_\infty(\sigma_f) $ for the B\"ottcher coordinate for $\sigma$, where we recall the normalization $\phi_{\sigma}'(\infty)=d^{\frac{1}{d-1}}\omega_0$. 

\bigskip

\noindent \emph{Proof of (1).} We first note that $\zeta_1^{f_0}=f_0(\xi_1^{f_0})=f_0(\omega_0)=(1+1/d)\omega_0$. A simple computation shows that \begin{align}\label{simple_relation} f(\omega_0\cdot x)=\omega_0\cdot\left(x+\frac{1}{dx^d}\right) \textrm{ for } x\in\mathbb{R}^+. \end{align} Next we note that \begin{align}\label{simple_relation2} x+\frac{1}{dx^d} > 1+\frac{1}{d} \textrm{ for } 0<x<1. \end{align} Let $\gamma:=\{ t\omega_0: t>1+1/d \}$. It follows from (\ref{simple_relation}) and (\ref{simple_relation2}) that $\sigma(\gamma) \subset \gamma$. Moreover, the endpoints $(1+1/d)\omega_0$, $\infty$ of $\gamma$ are fixed by $\sigma$, so since $\vert\overline{\partial}\sigma\vert>1$ on $\gamma$ by Lemma \ref{domain_hyperbolicity}, it follows that $\gamma\subset \mathcal{B}_{\infty}(\sigma)$.

We claim it follows then that $\gamma$ must be the $0$-ray for $\sigma$. Indeed, suppose by way of contradiction that $te^{i\theta}\in\phi_{\sigma}^{-1}(\gamma)$ where $t>1$ and $\theta\in(0,2\pi)$, where we may assume $\theta$ is not a pre-image of $0$ under $m_{-d}$. Then $t^{d^n}e^{i(-d)^n\theta}\in\phi_{\sigma}^{-1}(\gamma)$ for all $n>1$. Thus there exists $\theta'\in(0,2\pi)$ and a sequence $(z_n)_{n=1}^\infty\in\phi_{\sigma}^{-1}(\gamma)$ with $z_n\rightarrow\infty$ and $\textrm{arg}(z_n)\rightarrow\theta'$ as $n\rightarrow\infty$. But then since $\textrm{arg}(\phi_{\sigma}'(\infty))=\textrm{arg}(\omega_0)$, we have $\textrm{arg}(\phi_{\sigma}(z_n))\rightarrow\theta'+\textrm{arg}(\omega_0)$ as $n\rightarrow\infty$. This is a contradiction since $\theta'+\textrm{arg}(\omega_0)\not=\textrm{arg}(\omega_0)$ (mod $2\pi$), but $\arg(\phi_{\sigma}(z))=\textrm{arg}(\omega_0)$ for all $z\in\phi_{\sigma}^{-1}(\gamma)$. \qed

\bigskip

\noindent \emph{Proof of (2).} For $i\in\{1,\cdots,d+1\}$, let us denote by $X_i$ the set of $f\in\Sigma_d^*$ for which the $0$-ray of $\sigma_f$ lands at $\zeta_i^f$. 

We claim that each $X_i$ is an open set. To this end, suppose that $f\in X_i$. It follows from the parabolic behavior of the cusps that the tail of the $0$-ray of $\sigma_f$ is contained in a repelling petal at $\zeta_i^f$. In particular, we can assume that there exists some $r_0>0$ such that the part of the $0$-ray of $\sigma_f$ between potentials $r_0/d$ and $r_0$ is contained in a sufficiently small repelling petal at $\zeta_i^f$. Note that as cusps of $f(\mathbb{T})$ move continuously, so does a repelling petal at the cusp. It now follows from continuity of normalized B{\"o}ttcher coordinates that for $f'\in\Sigma_d^*$ close to $f$, the part of the $0$-ray of $\sigma_{f'}$ between potentials $r_0/d$ and $r_0$ is contained in a repelling petal at $\zeta_i^{f'}$. Since a repelling petal is invariant under the inverse branch of $\sigma_{f'}$ fixing $\zeta_i^{f'}$, we conclude that for $f'\in\Sigma_d^*$ close to $f$, the tail of the $0$-ray of $\sigma_{f'}$ is contained in a repelling petal at $\zeta_i^{f'}$. By Lemma~\ref{fixed_point_ray_lem}, the $0$-ray of $\sigma_{f'}$ must land at a cusp. Since a (sufficiently small) repelling petal has a unique cusp in its closure, it follows that the $0$-ray of $\sigma_{f'}$ must land at $\zeta_i^{f'}$, for all $f'\in\Sigma_d^*$ close to $f$. This proves the claim.

Again, we have that $\Sigma_d^*=\sqcup_{i=1}^{d+1} X_i$ by Lemma~\ref{fixed_point_ray_lem}. Now, connectedness of $\Sigma_d^*$ (Proposition~\ref{sigma_d^*_connected_thm}) and openness of each $X_i$ together imply that $\Sigma_d^*=X_j$, for some $j\in\{1,\cdots,d+1\}$. The result now follows from the fact that $f_0\in X_1$.
\end{proof}

\subsection{Lamination for The Limit Set $\Lambda(G)$}
\label{External Rays for the Reflection Group}

Recall from Proposition~\ref{group_lamination_prop} that for $\Gamma \in \overline{\beta(\pmb{\Gamma}_d)}$, there exists a continuous semi-conjugacy $\phi_\Gamma: \mathbb{T} \rightarrow \Lambda(\Gamma)$ between $\rho_{\pmb{\Gamma}_d}|_{\mathbb{T}}$ and $\rho_{\Gamma}|_{\Lambda(\Gamma)}$, and $\phi_\Gamma$ sends cusps of $\partial \Pi(\pmb{\Gamma}_d)$ to cusps of $\partial \Pi(\Gamma)$ with labels preserved. 

\begin{rem} The fibers of the map $\phi_\Gamma: \mathbb{T} \rightarrow \Lambda(\Gamma)$ of Proposition~\ref{group_lamination_prop} induce an equivalence relation on $\mathbb{T}$, and we will denote the set of all equivalence classes of this relation by $\lambda(\Gamma)$. 
\end{rem}

\noindent Adapting the arguments in the proof of Lemma \ref{fixed_point_ray_lem}, we have:

\begin{lem}\label{group_lamination_prop2}  Let $\Gamma\in\overline{\beta(\pmb{\Gamma}_d)}$. Then: \begin{enumerate}

\item For any cusp $\eta$ of $\partial T(\Gamma)$, we have $|\phi_\Gamma^{-1}(\eta)|=1$, and $\rho_{\pmb{\Gamma}_d}(\phi_\Gamma^{-1}(\eta)) = \phi_\Gamma^{-1}(\eta)$. 

\item For each double point $\eta$ of $\partial T(\Gamma)$, we have $|\phi_\Gamma^{-1}(\eta)|=2$, and the elements of $\phi_\Gamma^{-1}(\eta)$ form a 2-cycle under $\rho_{\pmb{\Gamma}_d}$. 
\end{enumerate}
\end{lem}

\begin{rem}\label{group_markov_partition} Let $\Gamma\in\overline{\beta(\pmb{\Gamma}_d)}$. Consider the set of angles $$\Theta:=\{\phi_\Gamma^{-1}(\eta): \eta\ \textrm{is\ a\ cusp\ or\ a\ double\ point\ of}\ \partial T(\Gamma)\}.$$ These angles cut $\mathbb{T}$ into finitely many pieces that form a Markov partition for $\rho_{\pmb{\Gamma}_d}:\mathbb{T}\to\mathbb{T}$. Analogously, the union of the cusps and double points of $\partial T(\Gamma)$ determines a Markov partition for $\rho_\Gamma:\Lambda(\Gamma)\to\Lambda(\Gamma)$.                                        
\end{rem}

\noindent Using the Markov partition of Remark~\ref{group_markov_partition}, the proof of Lemma \ref{schwarz_lamination_prop} may be adapted to show the following:

\begin{prop}\label{group_lamination_prop3}  Let $\Gamma\in\overline{\beta(\pmb{\Gamma}_d)}$, and $x\in\lambda(\Gamma)$ be a non-trivial equivalence class. Then:
\begin{enumerate}
\item $\vert x\vert=2$, and there is a double point $\eta$ of $\partial T(\Gamma)$ such that $\rho_{\pmb{\Gamma}_d}^{\circ n}(x)=\phi_\Gamma^{-1}(\eta)$ for some $n\geq0$.

\item If $n_0$ is the smallest non-negative integer with the above property, then $x$ is contained in a connected component of $\mathbb{T}\setminus \rho_{\pmb{\Gamma}_d}^{-(n_0-1)}(\Theta)$.
\end{enumerate}

\end{prop}

\begin{rem}\label{geod_lamination}
For $f\in\Sigma_{d,k}^*$ with $k\geq 1$, the index two Kleinian subgroup $\Gamma_f^+$ of $\Gamma_f$ has $k$ accidental parabolics. These accidental parabolics correspond to a collection of $k$ simple, closed, essential geodesics on $S^-=\D/\pmb{\Gamma}_d^+$ that can be pinched to obtain $\Gamma_f^+$. These geodesics lift by $\pmb{\Gamma}_d$ to the universal cover $\D$ giving rise to a geodesic lamination of $\D$ \cite[\S 3.9]{Marden}. By \cite{MS} (also compare \cite[p. 266]{Marden}), the quotient of $\mathbb{T}$ by identifying the endpoints of the leaves of this lamination produces a topological model of the limit set $\Lambda(\Gamma_f^+)=\Lambda(\Gamma_f)$. Therefore, up to rotation by a $(d+1)$-st root of unity, the set of equivalence classes of this geodesic lamination is equal to $\lambda(\Gamma_f)$. Moreover, the continuous map $\phi_{\Gamma_f}:\mathbb{T}\to\Lambda(\Gamma_f)$ is a \emph{Cannon-Thurston} map for $\Gamma_f$ (see \cite[\S 2.2]{MS} for a discussion of Cannon-Thurston maps). 
\end{rem}

\subsection{Relating The Laminations of Schwarz and Kleinian Limit Sets}\label{lamination_relation_sec} Given $f\in\Sigma_d^*$, we discussed the lamination of $\mathbb{T}$ induced by $\sigma_f$ in Subsection \ref{External Rays in the Schwarz Reflection Plane}, and the lamination of $\mathbb{T}$ induced by $\Gamma_f$ in Subsection \ref{External Rays for the Reflection Group}. The purpose of Subsection \ref{lamination_relation_sec} is to relate these two laminations.


\begin{prop}\label{laminations_correspond} Let $f\in\Sigma_d^*$. Then the homeomorphism $\mathcal{E}_d: \mathbb{T} \rightarrow \mathbb{T}$ descends to a homeomorphism \[\mathcal{E}_d:  \mathbb{T}/\lambda(\Gamma_f) \rightarrow \mathbb{T}/\lambda(\sigma_f).\] 
\end{prop}

\begin{figure}[ht!]
{\includegraphics[width=0.8\textwidth]{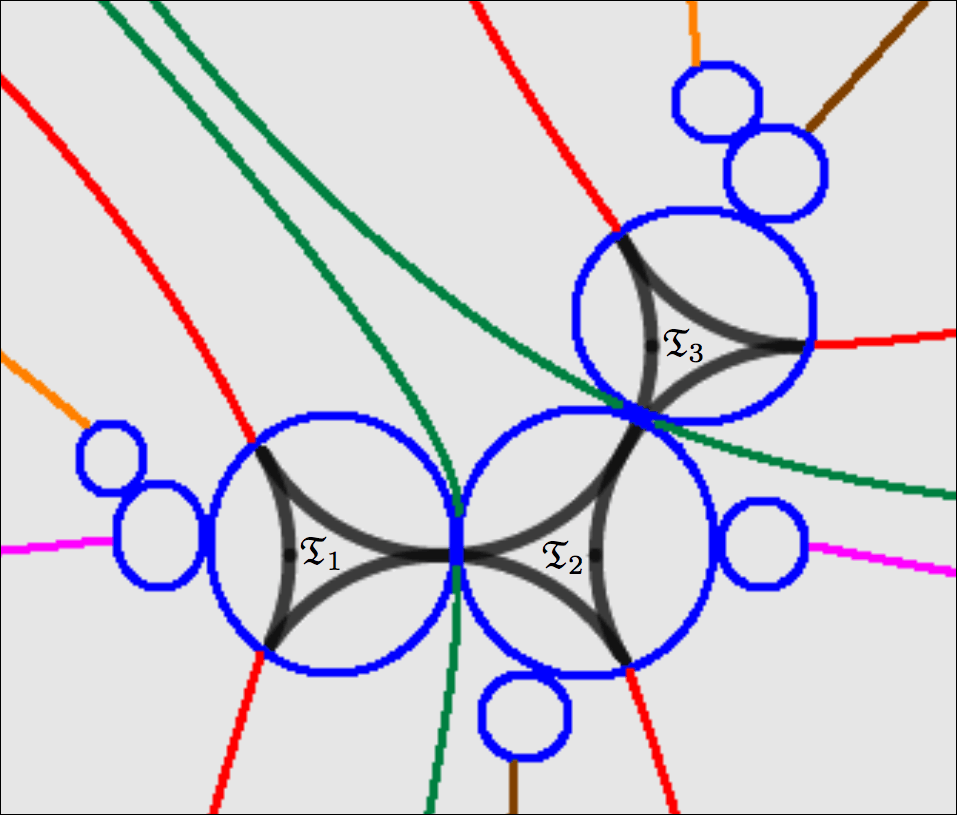}}
\caption{A cartoon of the rays of period $1$ and $2$ landing on $\Lambda(\sigma)$. Each cusp is the landing point of a unique fixed ray (in red), and each double point is the landing point of exactly two rays of period two (in green). The other rays of period two land at non-cut points of $\Lambda(\sigma)$. These rays are colored such that the two rays of the same color form a $2$-cycle. That the same pattern holds for the limit set $\Lambda(\Gamma)$ is the crux of the proof of Proposition \ref{laminations_correspond}. }
\label{fig:period_two_rays}
\end{figure}

\begin{proof} We let $0\leq k\leq d-2$ and fix $f\in\Sigma_{d,k}^*$. We abbreviate $\sigma:=\sigma_f$, $\Gamma:=\Gamma_f$. We denote by $\phi_\sigma$ the B\"ottcher coordinate for $\sigma$, and $\phi_\Gamma$ the map of Proposition~\ref{group_lamination_prop}.  Recall that the homeomorphism \[ h:  T(\Gamma) \rightarrow T(\sigma) \] of Proposition \ref{existence_of_group} is label-preserving. By Proposition \ref{schwarz_lamination_prop}, $\lambda(\sigma)$ is generated by \[ \left\{ \phi_\sigma^{-1}(\zeta) : \zeta \textrm{ is a double point of } f(\mathbb{T})    \right\}. \] Moreover, $\zeta$ is a double point of $f(\mathbb{T})$ if and only if $h^{-1}(\zeta)$ is a double point of $\partial T(\Gamma)$. Thus, by Proposition~\ref{group_lamination_prop3}, $\lambda(\Gamma)$ is generated by \[ \left\{ \phi_\Gamma^{-1}(h^{-1}(\zeta)) : \zeta \textrm{ is a double point of } f(\mathbb{T})    \right\}. \]  Thus it will suffice to show that \begin{align}\label{want_to_show} \tag{$\star$} \mathcal{E}_d^{-1}( \phi_\sigma^{-1}(\zeta) ) =   \phi_\Gamma^{-1}(h^{-1}(\zeta)) \textrm{ for each double point } \zeta \textrm{ of } f(\mathbb{T}). \end{align}

Let $\zeta$ be a double point of $f(\mathbb{T})$. By Lemma \ref{fixed_point_ray_lem}, $\phi_\sigma^{-1}(\zeta)$ is a $2$-cycle for $m_{-d}$ on $\mathbb{T}$. Similarly, by Lemma \ref{group_lamination_prop2}, $\phi_{\Gamma}^{-1}(h^{-1}(\zeta))$ is a 2-cycle for $\rho_{\pmb{\Gamma}_{d+1}}$. Note that the maps $m_{-d}$ and $\rho_{\pmb{\Gamma}_{d+1}}$ both have the same fixed points on $\mathbb{T}\cong\R/\Z$ which we label counter-clockwise as $\theta_1,\cdots, \theta_{d+1}$ with $\theta_1=0$. From the Markov property, there is a simple description of all $2$-cycles of $m_{-d}$ on $\mathbb{T}$: there is exactly one $2$-cycle $\{x, m_{-d}(x)\}$ in each pair of non-adjacent intervals $(\theta_i, \theta_{i+1})$, $(\theta_j, \theta_{j+1})$ with $x\in(\theta_i, \theta_{i+1})$, $m_{-d}(x)\in(\theta_j, \theta_{j+1})$. The same description holds for all $2$-cycles of $\rho_{\pmb{\Gamma}_{d+1}}$, and by definition via the Markov-property, the map $\mathcal{E}_d^{-1}$ sends the $2$-cycle of $m_{-d}$ in $(\theta_i, \theta_{i+1})$, $(\theta_j, \theta_{j+1})$ to the $2$-cycle of $\rho_{\pmb{\Gamma}_{d+1}}$ in $(\theta_i, \theta_{i+1})$, $(\theta_j, \theta_{j+1})$.

Now observe that by Proposition \ref{zero_ray_landing_cusp} and the label-preserving statement in Proposition \ref{group_lamination_prop}, we have the relation: \begin{align}\label{right_norm} h^{-1}(\phi_\sigma(1)) = \phi_\Gamma(1). \end{align} For $1\leq i \leq d+1$, $\phi_\sigma(\theta_i)$ is a cusp of $f(\mathbb{T})$ by Lemma \ref{fixed_point_ray_lem}, and $\phi_\Gamma(\theta_i)$ is a cusp of $\partial T(\Gamma)$ by Proposition~\ref{group_lamination_prop}. Since $h$ is label-preserving, it then follows from (\ref{right_norm}) that: \begin{align}\label{referenced_in_proof_of_theorem_A} h^{-1}(\phi_\sigma(\theta_i)) = \phi_\Gamma(\theta_i) \textrm{ for } 1\leq i \leq d+1. \end{align}

\noindent Thus it follows from the mapping properties of $h$ that \begin{align}\nonumber   \zeta\in\phi_\sigma(\theta_i, \theta_{i+1})\cap\phi_\sigma(\theta_j, \theta_{j+1}) \textrm{ if and only if } h^{-1}(\zeta)\in\phi_\Gamma(\theta_i, \theta_{i+1})\cap\phi_\Gamma(\theta_j, \theta_{j+1}).  \end{align}

\noindent Hence, the 2-cycle $\phi_\sigma^{-1}(\zeta)$ for $m_{-d}$ lies in $(\theta_i, \theta_{i+1})$, $(\theta_j, \theta_{j+1})$ if and only if the 2-cycle $\phi_\Gamma^{-1}(h^{-1}(\zeta))$ for $\rho_{\pmb{\Gamma}_{d+1}}$ lies in $(\theta_i, \theta_{i+1})$, $(\theta_j, \theta_{j+1})$. By the definition of the homeomorphism $\mathcal{E}_d$ via the Markov-partitions for $m_{-d}$ and $\rho_{\pmb{\Gamma}_{d+1}}$, it follows then that $\mathcal{E}^{-1}_d(\phi_\sigma^{-1}(\zeta)) = \phi_\Gamma^{-1}(h^{-1}(\zeta))$, as needed. 
\end{proof}

\begin{rem}\label{consequence_of_laminations_correspond1} Let notation be as in Proposition \ref{laminations_correspond}, and denote by $\phi_{\sigma_f}$ the B\"ottcher coordinate of $\sigma_f$, and $\phi_{\Gamma_f}$ the map of Proposition \ref{group_lamination_prop}. It follows from Proposition \ref{laminations_correspond} that 
\begin{align*}
\phi_{\sigma_f} \circ \mathcal{E}_d \circ \phi_{\Gamma_f}^{-1}: \Lambda(\Gamma_f) \rightarrow \Lambda(\sigma_f)
\end{align*}
is well defined, and indeed a topological conjugacy (see Figure~\ref{various_conjugacy}). 

\end{rem}

\subsection{Proof of Conformal Mating}\label{Proof of Conformal Mating} With Proposition \ref{laminations_correspond} in hand, we can finally prove the conformal mating statement of Theorem \ref{theorem_A}. We follow Definition \ref{mating} of conformal mating. 
\vspace{2mm}

\begin{proof}[Proof of Theorem~\ref{theorem_A}]
 The map \begin{align} \Sigma_d^* \rightarrow \overline{\beta(\pmb{\Gamma}_{d+1})} \nonumber \\ f\mapsto\Gamma_f \phantom{as} \nonumber \end{align} was already defined and proven to be a homeomorphism in Section \ref{homeomorphism}. The uniqueness statement of Theorem \ref{theorem_A} is evident since if $\Gamma\in\overline{\beta(\pmb{\Gamma}_{d+1})}$ is such that $\sigma_f$ is a conformal mating of $\Gamma$ and $w\mapsto\overline{w}^d$, then Condition (2) of Definition \ref{mating} and the uniqueness statement in Proposition \ref{existence_of_group} imply that $\Gamma=\Gamma_f$. Thus it only remains to show that $\sigma_f$ is indeed a conformal mating of $\Gamma_f$ and $w\mapsto\overline{w}^d$. Fix $f\in\Sigma_d^*$. We will abbreviate $\sigma:=\sigma_f$ and $\Gamma:=\Gamma_f$.

Recall from Remark \ref{bottcher_coordinate} the  B\"ottcher coordinate \begin{align}\nonumber \phi_\sigma: \mathbb{D}^* \rightarrow \mathcal{B}_\infty(\sigma) \textrm{ satisfying } \phi_\sigma^{-1} \circ \sigma \circ \phi_\sigma(u)= \overline{u}^d,\ \forall\ u \in\D^*.\end{align} By Corollary \ref{decomp_cor}, we have the relation \begin{equation} \widehat{\mathbb{C}} = \mathcal{B}_\infty(\sigma) \sqcup \Lambda(\sigma) \sqcup \mathcal{T}_\infty(\sigma). \end{equation} Thus $\widehat{\mathbb{C}}\setminus \mathcal{T}_\infty(\sigma)= \mathcal{B}_\infty(\sigma) \sqcup \Lambda(\sigma)$. By Proposition \ref{local_connectivity_prop}, $ \Lambda(\sigma)=\partial \mathcal{B}_\infty(\sigma)$ is locally connected, so that $\phi_\sigma$ extends as a semi-conjugacy $\mathbb{T}\rightarrow \Lambda(\sigma)$. Thus taking $p(w):=\overline{w}^d$ (so that $\mathcal{K}(p)=\overline{\mathbb{D}}$), and $\psi_p(w):=\phi_\sigma(1/w)$ for $w\in\overline{\mathbb{D}}$, it is evident that $\psi_p$ is conformal in $\Int{\mathcal{K}(p)}=\mathbb{D}$ and satisfies Condition (1) of Definition \ref{mating}.

Let $h: T(\Gamma) \rightarrow T(\sigma)$ be the mapping of Proposition \ref{existence_of_group} applied to $f$. Define $\psi_\Gamma(z):=h(z)$ for $z\in T^o(\Gamma)$. Note that $\psi_\Gamma$ is label-preserving by Proposition \ref{existence_of_group}. Lifting $\psi_\Gamma$ by $\rho_\Gamma$ and $\sigma$, we extend $\psi_\Gamma$ to a conformal map \[ \psi_\Gamma: \bigcup_{n=0}^\infty \rho_\Gamma^{-n}(T^o(\Gamma)) \rightarrow \bigcup_{n=0}^\infty \sigma^{-n}(T^o(\sigma)). \] Recall our notation $\mathcal{K}(\Gamma):=\mathbb{C}\setminus\Omega_\infty(\Gamma)$. Then $\Int{\mathcal{K}(\Gamma)}$ is the union of all bounded components of $\Omega(\Gamma)$, and we have \begin{align}\label{again_decomp} \mathcal{K}(\Gamma)=\Int{(\mathcal{K}(\Gamma))}\sqcup\Lambda(\Gamma) \textrm{ and } \overline{\mathcal{T}_\infty(\sigma)}=\mathcal{T}_\infty(\sigma)\sqcup\Lambda(\sigma). \end{align} By Proposition \ref{regular_fund_pre_image_prop} and Definition \ref{Schwarz_reflection}, we have \[ \bigcup_{n=0}^\infty \rho_\Gamma^{-n}(T^o(\Gamma)) = \Int{\mathcal{K}(\Gamma)} \textrm{ and } \bigcup_{n=0}^\infty \sigma^{-n}(T^o(\sigma)) = \mathcal{T}_\infty(\sigma). \] Thus $\psi_\Gamma: \Int{\mathcal{K}(\Gamma)} \rightarrow \mathcal{T}_\infty(\sigma)$ is conformal. Moreover, by the definition of $\psi_\Gamma$ via lifting, we have \begin{align} \psi_\Gamma\circ \rho_\Gamma(z)=\sigma_f\circ\psi_\Gamma(z) \textrm{ for } z\in \Int{\mathcal{K}(\Gamma)}\setminus \Int{T^o(\Gamma)}. \end{align} Thus in order to conclude that Condition (2) of Definition \ref{mating} holds, by (\ref{again_decomp}) it only remains to show that $\psi_\Gamma$ extends to a semi-conjugacy $\Lambda(\Gamma) \rightarrow \Lambda(\sigma)$. We will show that in fact  $\psi_\Gamma$ extends as a topological conjugacy.

Let $\phi_{\Gamma}: \mathbb{D}^* \rightarrow \Omega_{\infty}(\Gamma)$ denote the conformal map of Proposition \ref{group_lamination_prop}. As observed in Remark \ref{consequence_of_laminations_correspond1}, Proposition \ref{laminations_correspond} implies that the map \[\phi_\sigma\circ\mathcal{E}_d\circ\phi_{\Gamma}^{-1}: \Lambda(\Gamma) \rightarrow \Lambda(\sigma)\] is a well-defined homeomorphism, so that we only need to show that $\phi_\sigma\circ\mathcal{E}_d\circ\phi_{\Gamma}^{-1}$ is an extension of $\psi_\Gamma: \Int{\mathcal{K}(\Gamma)} \rightarrow \mathcal{T}_\infty(\sigma_f)$. Note that by construction and the normalization in Remark \ref{bottcher_coordinate}, $\psi_\Gamma$ and $\phi_\sigma\circ\mathcal{E}_d\circ\phi_{\Gamma}^{-1}$ agree on the cusps of $\partial T(\Gamma)$. One may then verify via the definition of $\psi_\Gamma$ (by lifting $\rho_\Gamma$ and $\sigma$) and $\mathcal{E}_d$ (in Remark~\ref{defn_of_conjug}) that $\psi_\Gamma$ and $\phi_\sigma\circ\mathcal{E}_d\circ\phi_{\Gamma}^{-1}$ agree on all preimages of cusps of $\partial T(\Gamma)$. As these preimages form a dense subset of $\Lambda(\Gamma)$, it follows that $\phi_\sigma\circ\mathcal{E}_d\circ\phi_{\Gamma}^{-1}$ is the desired homeomorphic extension of $\psi_\Gamma$.

It remains only to show Condition (3) of Definition \ref{mating}. Let $t\in\mathbb{T}$ and consider $\phi_\Gamma(t)\in\Lambda(\Gamma)$, $\phi_p\circ\overline{\mathcal{E}_d(t)}\in\mathcal{J}(p)$, where we note $\phi_p\equiv \textrm{id}$. We readily compute that \[  \psi_p(\phi_p\circ\overline{\mathcal{E}_d(t)} ) = \phi_\sigma(1/\overline{\mathcal{E}_d(t)} ) = \phi_\sigma(\mathcal{E}_d(t)) = \phi_\sigma\circ\mathcal{E}_d\circ\phi_\Gamma^{-1}(\phi_\Gamma(t)) = \psi_\Gamma(\phi_\Gamma(t)).  \] Thus for $z\in\mathcal{K}(\Gamma)$, $w\in\mathcal{K}(p)$, we see that $z\sim w \implies \psi_\Gamma(z)=\psi_p(w)$. 

If, conversely, $\psi_\Gamma(z)=\psi_p(w)$, we must have firstly that $\psi_\Gamma(z)=\psi_p(w)\in\Lambda(\sigma)$. Thus $z\in\Lambda(\Gamma)$, and $w\in\mathbb{T}$. Recalling $\psi_\Gamma=\phi_\sigma\circ\mathcal{E}_d\circ\phi_\Gamma^{-1}$ on $\Lambda(\Gamma)$ and $\psi_p(w)=\phi_\sigma(\overline{w})$ for $w\in\mathbb{T}$, we see that \begin{align}\label{good!} \phi_\sigma\circ\mathcal{E}_d\circ\phi_{\Gamma}^{-1}(z)=\phi_\sigma(\overline{w}). \end{align} As already noted, $\psi_\Gamma=\phi_\sigma\circ\mathcal{E}_d\circ\phi_\Gamma^{-1}: \Lambda(\Gamma)\mapsto\Lambda(\sigma)$ is a homeomorphism, so that we deduce from (\ref{good!}) that $z=\phi_\Gamma\circ\mathcal{E}_d^{-1}(\overline{w})$. Letting $t=\mathcal{E}_d^{-1}(\overline{w})$, we see that $\phi_\Gamma(t)=z$ and $\phi_p\circ\overline{\mathcal{E}_d(t)}=w$, so that by Definition~\ref{conf_mating_equiv_reltn},  $z\sim w$ as needed. 
\end{proof}

\section{Sullivan's Dictionary}\label{sullivan_sec}

\begin{definition}\label{abstract_tree_def}
An \emph{abstract angled tree} is a triple $(\mathcal{T},\textrm{deg},\angle)$, where: 

\begin{enumerate}
\item $\mathcal{T}$ is a tree,

\item $\textrm{deg}:V(\mathcal{T})\to\N$ is a function with $\textrm{deg}(v)\geq2$ for each vertex $v$ of $\mathcal{T}$,

\item $\textrm{valence}(v)\leq 1+\textrm{deg}(v)$ for each vertex $v$ of $\mathcal{T}$, and

\item $\angle$ is a skew-symmetric, non-degenerate, additive function defined on pairs of edges incident at a common vertex, and takes values in $\faktor{\frac{2\pi}{1+\textrm{deg}(v)}\Z}{2\pi\Z}$. 
\end{enumerate}

\end{definition}

\begin{rem}\label{abstract_tree_def2} If $(\mathcal{T},\textrm{deg},\angle)$ is an abstract angled tree, the positive integer \[d:=1+\sum_{v\in V(\mathcal{T})}\left(\textrm{deg}(v)-1\right)\] is called the \emph{total degree} of the angled tree. Two angled trees are said to be isomorphic if there is a tree isomorphism between them that preserves the functions $\textrm{deg}$ and $\angle$.
\end{rem}

\begin{example}\label{tree_to_f} To any $f\in\Sigma_{d,k}^*$, we will associate an abstract angled tree $\mathcal{T}(f)$ with $k+1$ vertices as follows. Denote by  $T_1,\cdots, T_{k+1}$ the components of $T^o(\sigma_f)$. Let $j_i\geq0$ be such that the boundary of $T_i$ has $3+j_i$ cusps. Assign a vertex $v_i$ to each component $T_i$, and connect two vertices $v_i$, $v_j$ by an edge if and only if $T_i$, $T_j$ share a common boundary point. We define the deg function by: \begin{align} \nonumber \textrm{deg}:V(\mathcal{T}(f))\rightarrow\N  \\ v_i\mapsto 2+j_i. \phantom{as} \nonumber \end{align} It remains to define the $\angle$ function for two edges $e$, $e'$ meeting at a vertex $v_i$. Suppose $e$, $e'$ correspond to two cusps $\zeta$, $\zeta' \in \partial T_i$, and denote by $\gamma_i$ the component of $\partial T_i\setminus\{\zeta, \zeta'\}$ which, when traversed counter-clockwise, is oriented positively with respect to $T_i$. Then \[  \angle(e,e') := \frac{2\pi}{3+j_i}\cdot\left(1 + \#\left\{ \textrm{cusps of $\partial T_i$ on the curve } \gamma_i  \right\}\right).  \] We leave it to the reader to verify that $(\mathcal{T}(f), \textrm{deg}, \angle)$ satisfies Definition \ref{abstract_tree_def} of an abstract angled tree. Note that if $f\in\Sigma_{d,d-2}^*$, then the tree $(\mathcal{T}(f),\textrm{deg}, \angle)$ is simply a \emph{bi-angled tree} in the language of \cite[\S 2.5]{LMM1}.
\end{example}

\begin{prop}\label{existence_poly_prop} For each $f\in\Sigma_{d,k}^*$, there exists an anti-polynomial $p_f$ of degree $d$ such that: 
\begin{enumerate} 
\item $p_f$ has a total of $k+1$ distinct critical points in $\mathbb{C}$, 
\item Each critical point of $p_f$ is fixed by $p_f$, and 
\item The angled Hubbard tree of $p_f$ is isomorphic to $(\mathcal{T}(f), \emph{deg}, \angle)$.  
\end{enumerate} 
\end{prop}

\begin{proof} We continue to use the notation introduced in Example \ref{tree_to_f}. One readily verifies that: \[ -(2+j_i)\left(\frac{2\pi n}{3+j_i}\right)=\frac{2\pi n}{3+j_i}\ \left(\textrm{mod}\ 2\pi\right) \textrm{ for } i=1,\cdots,k+1.  \] Thus, $\mathrm{id}:\mathcal{T}(f)\to\mathcal{T}(f)$ is an orientation-reversing angled tree map (see \cite[\S 2.7]{LMM1}) of degree \[1+\sum_{i=1}^{k+1}(\textrm{deg}(v_i)-1)=1+\sum_{i=1}^{k+1}(1+j_i)=d.\]  Moreover, since all vertices of $\mathcal{T}$ are critical and fixed under $\mathrm{id}$, it follows that all vertices are of \emph{Fatou type} (again, see \cite[\S 2.7]{LMM1}). Hence, the realization theorem \cite[Theorem~5.1]{Poi3} applied to the orientation-reversing angled tree map $\textrm{id}$ yields a postcritically finite anti-polynomial $p_f$ of degree $d$  such that the angled Hubbard tree of $p_f$ is isomorphic to $(\mathcal{T}(f),\textrm{deg},\angle)$. That $p_f$ satisfies (1), (2) follows since the Hubbard tree of $p_f$ is isomorphic to $(\mathcal{T}(f),\textrm{deg},\angle)$.
\end{proof}

\begin{prop}\label{cut_point_structure} Let $f\in\Sigma_{d,k}^*$, and $p_f$ as in Proposition \ref{existence_poly_prop}. Denote by $U_1,\cdots, U_{k+1}$ the immediate attracting basins of the fixed critical points of $p_f$. Then \[ U:= \bigcup_{i=1}^{k+1} \overline{U_i} \] is connected. Moreover, $p_f$ has exactly $2k+d+2$ fixed points in $\mathbb{C}$, of which: 
\begin{enumerate} 
\item $k+1$ are critical points, 
\item $k$ are cut-points of $U$ and belong to $\mathcal{J}(p_f)$, and 
\item $d+1$ are not cut-points of $U$ and belong to $\mathcal{J}(p_f)$.  
\end{enumerate}
\end{prop}

\begin{proof} We abbreviate $p:=p_f$. Enumerate the critical points of $p$ by $c_1, \cdots, c_{k+1}$. Since \begin{align}\label{fixed_points} p(c_i)=c_i \textrm{ for } 1\leq i \leq k+1, \end{align} $p$ can not have any indifferent fixed point. It also follows from (\ref{fixed_points}) that $(c_i)_{i=1}^{k+1}$ are the only attracting fixed points of $p$, since any basin of attraction of $p$ must contain a critical value. It then follows from the Lefschetz fixed point formula (see \cite[Lemma~6.1]{2014arXiv1411.3415L}) that $p$ has a total of $2k+d+2$ fixed points in $\mathbb{C}$, of which $k+d+1$ are repelling and thus belong to $\mathcal{J}(p)$.  

Note that $p|_{U_i}$ is conformally conjugate to $\overline{z}^{2+j_i}|_{\mathbb{D}}$. As $p$ is hyperbolic, $\partial U_i$ is locally connected (see \cite[Lemma~19.3]{MR2193309}), and so this conjugacy extends to the boundary. Thus, $p|_{\partial U_i}$ has $3+j_i$ fixed points for each $i$. We claim that: \begin{itemize}  \item[($\star$)]  a repelling fixed point can lie on the boundary of at most two $(U_i)_{i=1}^{k+1}$, and \item[($\star\star$)] there is at least one repelling fixed point which is on the boundary of precisely one $(U_i)_{i=1}^{k+1}$. \end{itemize} Statement ($\star$) follows from the fact that the basins of attraction are invariant under $p$ and that $p$ is an orientation-reversing homeomorphism in a neighborhood of a repelling fixed point. Statement ($\star\star$) follows from fullness of the filled Julia set $\mathcal{K}(p)$. Thus if we suppose, by way of contradiction, that, say $\overline{U_1}$ is disjoint from $\cup_{i=2}^{k+1}\overline{U_i}$, a counting argument yields that $p$ has at least \[   \underbrace{4+j_1}_{\textrm{ fixed points in }\overline{U_1}} + \underbrace{k}_{\textrm{ critical fixed points in }U_i\textrm{, } i>1} + \underbrace{ \sum_{i>1}(3+j_i) }_{ \textrm{ fixed points in } \partial U_i\textrm{, } i>1 } - \underbrace{(k-1)}_{\textrm{shared fixed points}} = \phantom{asd} 2k+d+3  \] fixed points in $\mathbb{C}$, which is a contradiction. Thus $U$ is connected.

An elementary argument using fullness of $\mathcal{K}(p)$ shows that two $\overline{U_i}$, $\overline{U_j}$ can intersect in at most one point, and that an intersection point of $\overline{U_i}$, $\overline{U_j}$ must be a fixed point of $p$ (see \cite[Proposition~6.2]{LMM1}). Thus by ($\star$) above and connectedness of $U$, it follows that $U$ has at least $k$ cut-points which are in $\mathcal{J}(f)$. Futhermore, by fullness of $\mathcal{K}(p)$, $U$ can not have more than $k$ cut-points. 
\end{proof}

\begin{notation} Let $p_f$ be as in Propositions \ref{existence_poly_prop}, \ref{cut_point_structure}. We denote by $\textrm{Rep}(p_f)$ the set of repelling fixed points of $p_f$, and by $\textrm{Cut}(p_f)$ the cut points of $U$. 

\end{notation}

\begin{rem}\label{bottcher_poly} Consider a B\"ottcher coordinate $\phi_p: \mathbb{D}^*\rightarrow \mathcal{B}_\infty(p)$ for $p=p_f$ as in Proposition \ref{existence_poly_prop}. Note that $\mathcal{J}(p)$ is connected as each finite critical point of $p$ is fixed. Thus since $p$ is hyperbolic, it follows that $\mathcal{J}(p)$ is locally connected \cite[Theorem~19.2]{MR2193309}. Hence $\phi_p$ extends continuously to a surjection $\phi_p: \mathbb{T} \rightarrow \mathcal{J}(p)$ which semi-conjugates $m_{-d}|_{\mathbb{R}/\mathbb{Z}}$ to $p|_{\mathcal{J}(p)}$, and all external rays of $\mathcal{B}_\infty(p)$ land. The fibers of $\phi_p|_{\mathbb{T}}$ induce an equivalence relation on $\mathbb{T}$ which we denote by $\lambda(p)$. Note that $\lambda(p)$ depends on a normalization of the B\"ottcher coordinate.

\end{rem}

\begin{lem}\label{cut_point_sara_cro} Let $f\in\Sigma_{d}^*$, $p_f$ as in Proposition \ref{existence_poly_prop}, and $\phi_p$ any B\"ottcher coordinate for $p_f$. Then: 
\begin{enumerate} 
\item Each $\beta\in\emph{Rep}(p_f)\setminus\emph{Cut}(p_f)$ is the landing point of a unique external ray. The angle of this external ray is fixed by $m_{-d}$. 
\item Each $\beta\in\emph{Rep}(p_f)\cap\emph{Cut}(p_f)$ is the landing point of exactly two external rays. The angles of these two rays form a $2$-cycle under $m_{-d}$. 
\end{enumerate}
\end{lem}

\begin{proof} We abbreviate $p:=p_f$, and continue to use the notation of Propositions \ref{existence_poly_prop}, \ref{cut_point_structure}. By Remark \ref{bottcher_poly}, each $\beta\in\textrm{Rep}(p)$ is the landing point of at least one external ray.

\vspace{2mm}

\noindent \emph{Proof of (1)}: Let $\beta\in\textrm{Rep}(p_f)\setminus\textrm{Cut}(p_f)$. Denote by $S$ the component of $\mathcal{K}(p)\setminus\{\beta\}$ containing $U\setminus\{\beta\}$. It follows from the covering properties of $p$ that $\cup_{r=0}^n p^{-r}(U\setminus\{\beta\})$ is connected for each $n\geq0$. Thus \[\bigcup_{r=0}^n p^{-r}(U\setminus\{\beta\}) \subseteq S \textrm{ for all } n\geq0. \] Observe that \begin{align}\label{interior_of_filled_julia} \Int{\mathcal{K}(p)}=\bigcup_{r=0}^\infty p^{-r}\left(\bigcup_{i=1}^{k+1} U_i\right). \end{align} We then have \begin{align}\label{filled_relation} \mathcal{K}(p) = \overline{\Int{\mathcal{K}(p)}} \subseteq\overline{\bigcup_{r=0}^\infty p^{-r}(U\setminus\{\beta\})}\subseteq \overline{S}=S\cup\{\beta\}, \end{align} where the first equality in (\ref{filled_relation}) follows from \cite[Corollary~4.12]{MR2193309}, and the proceeding $\subseteq$ relation follows from (\ref{interior_of_filled_julia}). By definition of $S$, we have $\mathcal{K}(p)=S\cup\{\beta\}$, so that $\beta$ is not a cut point of $\mathcal{K}(p)$. Thus $\beta$ is the landing point for exactly one external ray. Since $\beta$ is fixed, it follows that the angle of the external ray landing at $\beta$ is fixed under $m_{-d}$.

\vspace{2mm}

\noindent \emph{Proof of (2)}: Let $\beta\in\textrm{Rep}(p_f)\cap\textrm{Cut}(p_f)$. Let $S_1$, $S_2$ be the two components of $\mathcal{K}(p)\setminus\{\beta\}$ such that $U\setminus\{\beta\}\subseteq S_1\cup S_2$. A similar argument as for (1) shows that $S_1$ and $S_2$ are the only components of $\mathcal{K}(p)\setminus\{\beta\}$. Thus there are only two accesses to $\beta$ in $\mathcal{B}_\infty(p)$, and hence exactly two external rays landing at $\beta$. Since the $d+1$ fixed external rays land at the $d+1$ points of $\textrm{Rep}(p_f)\setminus\textrm{Cut}(p_f)$ by (1), it follows that the external rays landing at $\beta$ must have period $2$, and hence form a 2-cycle under $m_{-d}$.
\end{proof}

\begin{figure}[ht!]
{\includegraphics[width=0.9\textwidth]{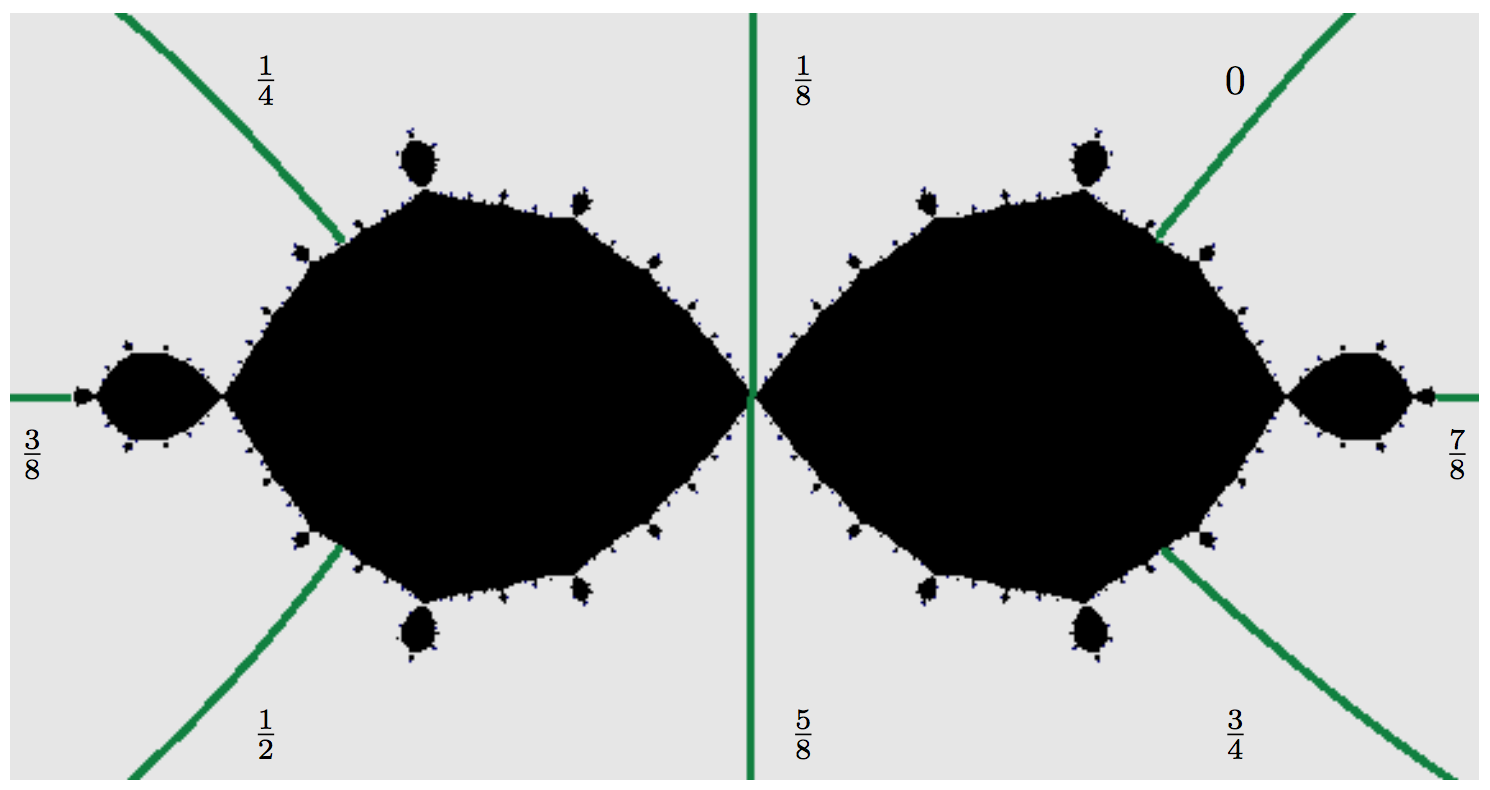}}
\caption{Illustrated is the dynamical plane of $p_f$, where $f(z):=z+2/(3z)-1/(3z^3)$. For this particular $f$, an explicit formula is known for $p_f$: $p_f(z)=\overline{z}^3-\frac{3i}{2}\overline{z}$ (the figure displayed is a $\pi/4$-rotate
of the actual dynamical plane). Also shown are all external rays of $p_f$ of period $1$ and $2$ (with angles indicated). The idea of the proof of Proposition~\ref{laminations_correspond2} is to show that the external rays of $\sigma_f$ landing at the double points of $f(\mathbb{T})$ have the same landing pattern as for those rays landing at the cut-points of the immediate basins of attraction for $p_f$.}
\label{fig:cubic_crit_fixed}
\end{figure}

\begin{prop}\label{laminations_correspond2} Let $f\in\Sigma_d^*$, and $p_f$ as in Proposition \ref{existence_poly_prop}. There is a normalization of the B\"ottcher coordinate for $p_f$ such that  $\lambda(p_f)=\lambda(\sigma_f)$. 
\end{prop}

\begin{rem} The idea of the proof is similar to that of Proposition \ref{laminations_correspond}, for which we refer to Figure \ref{fig:cubic_crit_fixed}. \end{rem}

\begin{proof} Let $k$ be such that $f\in\Sigma_{d,k}^*$. We abbreviate $\sigma:=\sigma_f$, $p:=p_f$. Consider the isomorphism of the angled Hubbard tree of $p$ with the abstract angled tree of $f$ as defined in Example \ref{tree_to_f}. Thus there is, first of all, a bijection between the attracting basins $U_1, \cdots, U_{k+1}$ of $p$ and the components $T_1, \cdots, T_{k+1}$ of $T^o(\sigma)$. We ensure the labeling is such that $U_i$ is mapped to $T_i$. Since the $\textrm{deg}$ function is preserved, the number of singular points on each $\partial T_i$ is equal to the number of fixed points of $p|_{\partial U_i}$. Moreover, since the $\angle$ function is preserved, for each $1\leq i\leq k+1$ there is a bijection \[\chi_i: \partial U_i \cap \textrm{Rep}(p) \to \partial T_i \cap \textrm{Sing}(f(\mathbb{T}))\] satisfying: 
\begin{enumerate} 
\item For $1\leq j\leq k+1$, one has $\chi_i(\beta)\in\partial T_i\cap \partial T_j$ if and only if $\beta\in\partial U_i\cap \partial U_j$;

\item $(\beta_1, \beta_2, \beta_3)$ is oriented positively with respect to $U_i$ if and only if

\noindent $(\chi_i(\beta_1), \chi_i(\beta_2), \chi_i(\beta_3))$ is oriented positively with respect to $T_i$. \end{enumerate} 
By (1), the map $\chi: \textrm{Rep}(p) \rightarrow \textrm{Sing}(f(\mathbb{T}))$ defined piecewise as $\chi_i$ on each $\textrm{Rep}(p)\cap\partial U_i$ is well-defined, whence it follows that $\chi$ is a bijection.

Denote by $\phi_p$, $\phi_{\sigma}$ the B\"ottcher coordinates for $p$, $\sigma$, respectively. We normalize $\phi_p$ so that \[ \chi\circ\phi_p(1) = \phi_\sigma(1). \] Recall that the cusps of $\partial T_i$ are the landing points of the fixed rays in $\partial \mathcal{B}_\infty(\sigma)$ by Lemma \ref{fixed_point_ray_lem}, and the points $\textrm{Rep}(p)\setminus\textrm{Cut}(p)$ are the landing points of the fixed rays in $\partial \mathcal{B}_\infty(p)$ by Lemma \ref{cut_point_sara_cro}. The fixed rays of $\partial \mathcal{B}_\infty(\sigma)$ and $\partial \mathcal{B}_\infty(p)$ have the same angles, and we enumerate them $\theta_1,\cdots, \theta_{d+1}$ where $\theta_1:=0$. There is a 2-cycle (under $m_{-d}$) on $\mathbb{T}$ in each pair of non-adjacent intervals $(\theta_i, \theta_{i+1})$, $(\theta_j, \theta_{j+1})$, and this constitutes all 2-cycles of $m_{-d}$.

 By Lemma \ref{cut_point_sara_cro}, for each $\beta \in\textrm{Rep}(p)\cap\textrm{Cut}(p)$, the set $\phi_p^{-1}(\beta)$ is a 2-cycle on $\mathbb{T}$.  The $2$-cycle $\phi_p^{-1}(\beta)$ is in the pair of intervals $(\theta_i, \theta_{i+1})$, $(\theta_j, \theta_{j+1})$ if and only if $\beta$ lies on both $\phi_p((\theta_i,\theta_{i+1}))$ and $\phi_p((\theta_j,\theta_{j+1}))$. Similarly, for each double point $\zeta$ of $f(\mathbb{T})$, the set $\phi_\sigma^{-1}(\zeta)$ is a $2$-cycle on $\mathbb{T}$ by Lemma \ref{fixed_point_ray_lem}. And moreover, the $2$-cycle $\phi_\sigma^{-1}(\zeta)$ is in the pair of intervals $(\theta_i, \theta_{i+1})$, $(\theta_j, \theta_{j+1})$ if and only if $\zeta$ lies on both $\phi_\sigma((\theta_i,\theta_{i+1}))$ and $\phi_\sigma((\theta_j,\theta_{j+1}))$. Thus, by the definition of $\chi$, for $\beta \in\textrm{Rep}(p)\cap\textrm{Cut}(p)$, the $2$-cycle $\phi_p^{-1}(\beta)$ is in the pair of intervals $(\theta_i, \theta_{i+1})$, $(\theta_j, \theta_{j+1})$ if and only if $\phi_\sigma^{-1}(\chi(\beta))$ is in the same pair of intervals. As there is only one $2$-cycle in any such pair, it follows that $\phi_p^{-1}(\beta)=\phi_\sigma^{-1}(\chi(\beta))$. 
 
Recall that by Proposition \ref{schwarz_lamination_prop}, the pairs $\phi_\sigma^{-1}(\zeta)$ over all double points $\zeta$ of $f(\mathbb{T})$ generate $\lambda(\sigma)$. A completely analogous proof to that of Proposition \ref{schwarz_lamination_prop} shows $\lambda(p)$ is generated by pairs $\phi_p^{-1}(\beta)$ where $\beta$ ranges over $\textrm{Rep}(p)\cap\textrm{Cut}(p)$. Thus since $\chi$ is a bijection and $\phi_p^{-1}(\beta)=\phi_\sigma^{-1}(\chi(\beta))$ for all $\beta \in \textrm{Rep}(p)\cap\textrm{Cut}(p)$, it follows that  $\lambda(\sigma)=\lambda(p)$. 
\end{proof}

\begin{figure}[ht]
\begin{tikzpicture}
\node at (9.6,5.4) {$\R/\Z$};
\node at (3.1,5.4) {$\R/\Z$};
\node at (3,6.48) {$m_{-d}$};
\node at (6.4,6) {$\mathcal{E}_d$};
\node at (9.6,6.48) {$\pmb{\rho}_{\Gamma_{d+1}}$};
\node at (1.6,1.4) {$\Lambda(\sigma_f)$};
\node at (5.12,1.4) {$\mathcal{J}(p_f)$};
\node at (9.5,1.4) {$\Lambda(\Gamma_f)$};
\node at (1.5,0.28) {$\sigma_f$};
\node at (5.12,0.28) {$p_f$};
\node at (9.5,0.28) {$\rho_{\Gamma_f}$};
\node at (1.8,3.5) {$\phi_{\sigma_f}$};
\node at (4.8,3.5) {$\phi_{p_f}$};
\node at (10.2,3.5) {$\phi_{\Gamma_f}$};

\node at (1.5,0.8) {\begin{huge}$\circlearrowleft$\end{huge}};
\node at (5.1,0.8) {\begin{huge}$\circlearrowleft$\end{huge}};
\node at (9.5,0.8) {\begin{huge}$\circlearrowleft$\end{huge}};
\node at (9.6,6) {\begin{huge}$\curvearrowright$\end{huge}};
\node at (3.1,6) {\begin{huge}$\curvearrowleft$\end{huge}};

\draw [<-] (3.6,5.4) to (9,5.4);
\draw [->] (3,5) to (1.8,1.8);
\draw [->] (3.3,5) to (5,1.8);
\draw [->] (9.6,5) to (9.6,1.8);

\end{tikzpicture}
\caption{Various topological conjugacies.}

\label{various_conjugacy}
\end{figure}
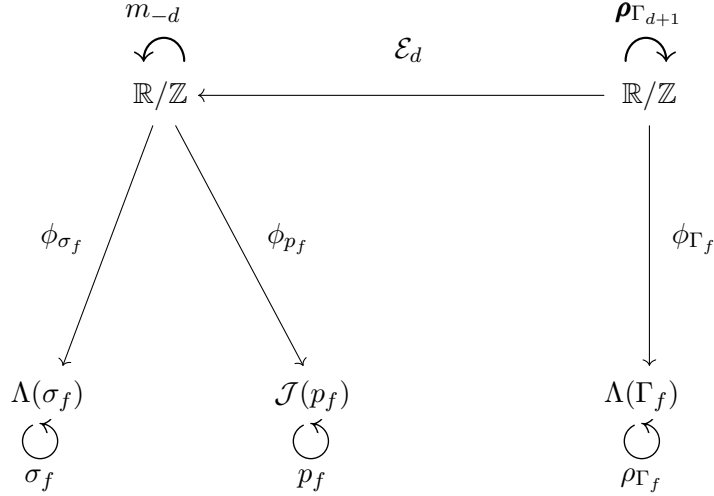

\begin{rem}\label{consequence_of_lamination_prop} Let notation be as in Proposition \ref{laminations_correspond2}, and denote by $\phi_{p_f}$, $\phi_{\sigma_f}$ the B\"ottcher coordinates of $p_f$, $\sigma_f$ (respectively) with $\phi_{p_f}$ normalized as in Proposition \ref{laminations_correspond2}. It follows from Proposition \ref{laminations_correspond2} that \begin{align} \phi_{p_f} \circ \phi_{\sigma_f}^{-1}: \Lambda(\sigma_f) \rightarrow \mathcal{J}(p_f)  \nonumber \end{align} is well-defined, and indeed a topological conjugacy (see Figure~\ref{various_conjugacy}).

\end{rem}

\noindent We note that Theorem \ref{theorem_B} follows immediately from:

\begin{thmx}\label{theorem_B_modified} 
Let $f\in\Sigma_d^*$. Denote by $\sigma_f$, $\Gamma_f$, $p_f$ the Schwarz reflection map, Kleinian reflection group, and critically fixed anti-polynomial determined by Definition \ref{Schwarz_reflection}, Theorem \ref{theorem_A}, and Proposition \ref{existence_poly_prop}, respectively. Then the dynamical systems \begin{align} \sigma_f: \Lambda(\sigma_f) \rightarrow \Lambda(\sigma_f), \nonumber \\ \rho_{\Gamma_f}: \Lambda(\Gamma_f) \rightarrow \Lambda(\Gamma_f), \nonumber \\  p_f: \mathcal{J}(p_f) \rightarrow \mathcal{J}(p_f) \nonumber \end{align} are pairwise topologically conjugate. \end{thmx}

\begin{proof}[Proof of Theorem~\ref{theorem_B_modified}] That $\sigma_f|_{\Lambda(\sigma_f)}$ and $p_f|_{\mathcal{J}(p_f)}$ are topologically conjugate is a consequence of Proposition~\ref{laminations_correspond2} as explained in Remark \ref{consequence_of_lamination_prop}. That $\sigma_f|_{\Lambda(\sigma_f)}$ and $\rho_{\Gamma_f}|_{\Lambda(\Gamma_f)}$ are topologically conjugate follows from Proposition~\ref{laminations_correspond}, as explained in Remark~\ref{consequence_of_laminations_correspond1}.
\end{proof}

\begin{rem} In the spirit of \cite[Theorem~7.2]{LLMM4}, it is natural to ask whether $\Lambda(\sigma_f)$, $\Lambda(\Gamma_f)$, $\mathcal{J}(p_f)$ can be distinguished by their quasisymmetry groups. \end{rem}

\begin{rem}\label{monic_normalization}
In light of Proposition~\ref{laminations_correspond2}, we can conjugate $p_f$ by an affine map to assume that $p_f$ is monic, centered, and $\lambda(p_f)=\lambda(\sigma_f)$, where $\lambda(p_f)$ is determined by the B{\"o}ttcher coordinate of $p_f$ that is tangent to the identity at $\infty$. In fact, $p_f$ becomes unique with such normalization. Moreover, it directly follows from the proof of Proposition~\ref{laminations_correspond} and Remark~\ref{geod_lamination} that the circle homeomorphism $\mathcal{E}_d$ transports the geodesic lamination that produces a topological model for $\Lambda(\Gamma_f)$ to the lamination that produces a topological model for $\mathcal{J}(p_f)$.
\end{rem}



\begin{thebibliography}{LLMM18b}

\bibitem[AM96]{AnMa}
J.~W. Anderson and B.~Maskit.
\newblock On the local connectivity of limit sets of {K}leinian groups.
\newblock {\em Complex Variables, Theory and Application}, 31(2):177--183,
  1996.

\bibitem[AS76]{AS}
D.~Aharonov and H.~S. Shapiro.
\newblock Domains on which analytic functions satisfy quadrature identities.
\newblock {\em J. Analyse Math.}, 30:39--73, 1976.

\bibitem[Ber60]{Bers60}
L.~Bers.
\newblock Simultaneous uniformization.
\newblock {\em Bulletin of the American Mathematical Society}, 66(2):94--97,
  1960.

\bibitem[BL20]{BL20}
S.~Bullett and L.~Lomonaco.
\newblock Mating quadratic maps with the modular group {II}.
\newblock {\em Inventiones Mathematicae}, 220:185--210, 2020.

\bibitem[BP94]{BP}
S.~Bullett and C.~Penrose.
\newblock Mating quadratic maps with the modular group.
\newblock {\em Inventiones Mathematicae}, 115:483--511, 1994.

\bibitem[BS79]{BoSe}
R.~Bowen and C.~Series.
\newblock Markov maps associated with {F}uchsian groups.
\newblock {\em Publications Math{\'e}matiques de L'I.H.{\'E}.S}, 50(153-170),
  1979.

\bibitem[CG93]{CG1}
L.~Carleson and T.~W. Gamelin.
\newblock {\em Complex Dynamics}.
\newblock Springer, Berlin, 1993.

\bibitem[DH85]{orsay}
Adrien Douady and John~H. Hubbard.
\newblock {\em {\'E}tude dynamique des polyn\^omes complexes {I}, {II}}.
\newblock Publications Math\'ematiques d'Orsay. Universit\'e de Paris-Sud,
  D\'epartement de Math\'ematiques, Orsay, 1984 - 1985.

\bibitem[Dou83]{Dou83}
A.~Douady.
\newblock Syst{\`e}mes dynamiques holomorphes.
\newblock In {\em S{\'e}minaire Bourbaki}, volume 1982/83, pages 39--63.
  Ast{\'e}risque, 105--106, Soc. Math. France, Paris, 1983.

\bibitem[KT90]{KeTh}
S.~P. Kerckhoff and W.~P. Thurston.
\newblock Non-continuity of the action of the modular group at {B}ers' boundary
  of {T}eichmuller space.
\newblock {\em Inventiones mathematicae}, 100:25--47, 1990.

\bibitem[LLMM18a]{LLMM1}
S.-Y. Lee, M.~Lyubich, N.~G. Makarov, and S.~Mukherjee.
\newblock Dynamics of {S}chwarz reflections: the mating phenomena.
\newblock \url{https://arxiv.org/abs/1811.04979}, 2018.

\bibitem[LLMM18b]{LLMM2}
S.-Y. Lee, M.~Lyubich, N.~G. Makarov, and S.~Mukherjee.
\newblock {S}chwarz reflections and the {T}ricorn.
\newblock \url{https://arxiv.org/abs/1812.01573}, 2018.

\bibitem[LLMM19]{LLMM4}
R.~Lodge, M.~Lyubich, S.~Merenkov, and S.~Mukherjee.
\newblock On dynamical gaskets generated by rational maps, {K}leinian groups,
  and {S}chwarz reflections.
\newblock \url{https://arxiv.org/abs/1912.13438}, 2019.

\bibitem[LM97]{LyMi}
M.~Lyubich and Y.~Minsky.
\newblock Laminations in holomorphic dynamics.
\newblock {\em J. Differential Geom.}, 47:17--94, 1997.

\bibitem[LM14]{2014arXiv1411.3415L}
S.-Y. {Lee} and N.~{Makarov}.
\newblock {Sharpness of connectivity bounds for quadrature domains}.
\newblock \url{https://arxiv.org/abs/1411.3415}, November 2014.

\bibitem[LMM19]{LMM1}
K.~Lazebnik, N.~G. Makarov, and S.~Mukherjee.
\newblock Univalent polynomials and {H}ubbard trees.
\newblock \url{https://arxiv.org/abs/1908.05813}, 2019.

\bibitem[LMMN20]{LMMN}
M.~Lyubich, S.~Merenkov, S.~Mukherjee, and D.~Ntalampekos.
\newblock David extension of circle homeomorphisms, mating, and removability.
\newblock \url{https://arxiv.org/abs/2010.11256}, 2020.

\bibitem[LV73]{MR0344463}
O.~Lehto and K.~I. Virtanen.
\newblock {\em Quasiconformal mappings in the plane}.
\newblock Springer-Verlag, New York-Heidelberg, second edition, 1973.
\newblock Translated from the German by K. W. Lucas, Die Grundlehren der
  mathematischen Wissenschaften, Band 126.

\bibitem[Mar07]{Marden}
A.~Marden.
\newblock {\em Outer circles: an introduction to hyperbolic 3-manifolds}.
\newblock Cambridge University Press, 2007.

\bibitem[McM95]{McM95}
C.~T. McMullen.
\newblock The classification of conformal dynamical systems.
\newblock In R.~Bott, M.~Hopkins, A.~Jaffe, I.~Singer, D.~W. Stroock, and S.-T.
  Yau, editors, {\em Current Developments in Mathematics}, pages 323-- 360.
  International Press, 1995.

\bibitem[Mil06]{MR2193309}
John Milnor.
\newblock {\em Dynamics in one complex variable}, volume 160 of {\em Annals of
  Mathematics Studies}.
\newblock Princeton University Press, Princeton, NJ, third edition, 2006.

\bibitem[MS13]{MS}
M.~Mj and C.~Series.
\newblock Limits of limit sets {I}.
\newblock {\em Geom. Dedicata}, 167:35--67, 2013.

\bibitem[Pil03]{Pil03}
K.~M. Pilgrim.
\newblock {\em Combinations of complex dynamical systems}.
\newblock Lecture Notes in Mathematics. Springer, 2003.

\bibitem[PM12]{AFST_2012_6_21_S5_839_0}
Carsten~Lunde Petersen and Daniel Meyer.
\newblock On the notions of mating.
\newblock {\em Annales de la Facult\'e des sciences de Toulouse :
  Math\'ematiques}, Ser. 6, 21(S5):839--876, 2012.

\bibitem[Poi13]{Poi3}
A.~Poirier.
\newblock Hubbard forests.
\newblock {\em Ergodic Theory and Dynamical systems}, 33:303--317, 2013.

\bibitem[SM98]{SulMc}
D.~Sullivan and C.~T. McMullen.
\newblock Quasiconformal homeomorphisms and dynamics {III}: The
  {T}eichm{\"u}ller space of a holomorphic dynamical system.
\newblock {\em Advances in Mathematics}, 135:351--395, 1998.

\bibitem[Ste05]{MR2131318}
Kenneth Stephenson.
\newblock {\em Introduction to circle packing: The theory of discrete analytic
  functions}.
\newblock Cambridge University Press, Cambridge, 2005.

\bibitem[Sul85]{Sul}
D.~Sullivan.
\newblock Quasiconformal homeomorphisms and dynamics {I}. solution of the
  {F}atou-{J}ulia problem on wandering domains.
\newblock {\em Annals of Mathematics}, 122(2):401--418, 1985.

\bibitem[Tuk85]{MR783351}
Pekka Tukia.
\newblock On isomorphisms of geometrically finite {M}\"{o}bius groups.
\newblock {\em Inst. Hautes \'{E}tudes Sci. Publ. Math.}, 61:171--214, 1985.

\bibitem[VS93]{Vin1}
E.~B. Vinberg and O.~V. Shvartsman.
\newblock {\em Geometry II: Spaces of Constant Curvature}, volume~29 of {\em
  Encyclopaedia of Mathematical Sciences}.
\newblock Springer-Verlag, 1993.

\end{thebibliography}
\end{document}